\documentclass{amsart}

\usepackage{xr}
\usepackage{lineno,hyperref}
\usepackage{amssymb}
\usepackage{amsmath}
\usepackage{amsthm}
\usepackage{mathrsfs}
\usepackage{bbm}
\usepackage{cite}
\usepackage{color}

\newtheorem{thm}{Theorem}[section]

\newtheorem{defn}[thm]{Definition}
\newtheorem{lem}[thm]{Lemma}
\newtheorem{prop}[thm]{Proposition}
\newtheorem{exa}[thm]{Example}
\newtheorem{rmk}[thm]{Remark}

\newcommand{\diag}{\mathop{\mathrm{diag}}}
\newcommand{\Id}{\mathop{\mathrm{Id}}}
\newcommand{\id}{\mathop{\mathrm{id}}}

\begin{document}

\title{Riemann slice-domains over quaternions I}
\author{Xinyuan Dou}
\email[X.~Dou]{douxy@mail.ustc.edu.cn}
\author{Guangbin Ren}
\email[G.~Ren]{rengb@ustc.edu.cn}
\date{\today}
\address{Department of Mathematics, University of Science and Technology of China, Hefei 230026, China}
\keywords{Functions of hypercomplex variable; domains of holomophy; Riemann domains; quaternions; slice regular functions; representation formula}
\thanks{This work was supported by the NNSF of China (11771412)}

\subjclass[2010]{Primary: 30G35; Secondary: 32A30, 32D05, 32D26}

\begin{abstract}
	We construct a counterexample to a well-known extension theorem for slice regular functions, which motivates us to develop a theory of Riemann slice-domains by introducing a new topology on quaternions. By some paths describing axial symmetry in Riemann slice-domains, we rectify the classical extension formula in the theory of slice regular functions and prove a representation formula over slice-domains of regularity. This proof involves an intertwining relation between imaginary units of quaternions and a fixed matrix corresponding to a complex structure.
\end{abstract}

\maketitle

\section{Introduction}
The theory of slice regular functions is an extension of the theory of one complex variable to quaternions, which is initiated by  Gentili and  Struppa \cite{Gentili2006001,Gentili2007001} and  has been in full development in the last decade (see \cite{Colombo2011001B,Gentili2013001B,Alpay2016001B}). It has been generalized to Clifford algebras \cite{Gentili2008003,Colombo2009002,Colombo2015001,Ren2017001}, octonions \cite{Gentili2008002,Gentili2010001,Colombo2010003,Wang2017001}, and real alternative *-algebras \cite{Ghiloni2011001,Ghiloni20171001,Ghiloni20171002}. In particular, a theory of quaternionic operators \cite{Colombo2009003,Colombo2011001,Ghiloni2018001} based on slice regular functions provides rigorous mathematical tools for quaternionic quantum mechanics \cite{Ghiloni2013001}.

A key result in this theory is a so-called representation formula over axially symmetric domains \cite{Colombo2009001,Colombo2010001}. This formula recovers the values of a slice regular function on an axially symmetric domain from its values on a single slice complex plane. Consequently, several results from complex analysis can be extended easily to the theory of slice regular functions. Therefore, this formula plays an important role in many respects, including the power series \cite{Gentili2012001,Stoppato2012001}, sheaves \cite{Colombo2012001,Gentili2016001}, Schur analysis \cite{Alpay2012001} and quaternionic operators \cite{Alpay2015001,Alpay2016001,Colombo2018001}. It also induces the definition of slice functions in the theory of slice regular functions on real alternative *-algebras \cite{Ghiloni2011001,Ghiloni20171001,Ghiloni20171002}.

However, the domain of definition of a slice regular function may not be axially symmetric. An extension theorem in  \cite{Colombo2009001} guarantees that every slice regular function defined on a slice domain $\Omega$ can extend slice regularly to an axially symmetric domain, the axailly symmetric completion of $\Omega$. Unfortunately, the proof of this extension theorem uses an identity principle implicitly, without verifying the connectedness of the intersection of domains. There is a counterexample in Example \ref{th-ce}, which indicates that the skew field of quaternions is not ``large" enough for this extension theorem. In fact, the slice regular extension of a given slice regular function on a slice domain $\Omega$ might be a multivalued function, and there may not exist a branch of the multivalued function on the axially symmetric completion of $\Omega$.

Since this counterexample is a little intricate, we consider a simple classical example (see e.g. \cite[Problem 1.1]{Teleman2003001B}) to explain the motivation of introducing the concept of Riemann slice-domains. We can define the square root function
\begin{equation*}
f_\mathbb{R}:\mathbb{R}^+\rightarrow\mathbb{R}^+,\qquad x\mapsto\sqrt{x},
\end{equation*}
with no ambiguities, by selecting positive square roots, where $\mathbb{R}^+:=(0,+\infty)$.

But in the complex case, $f_\mathbb{R}$ can extend holomorphically to $f_\imath:\Omega_\imath\rightarrow\mathbb{C}$, $\imath=1,2$, respectively, where
\begin{equation*}
\Omega_\imath:=\mathbb{C}\backslash(\{0\}\cup(-1)^\imath i\mathbb{R}^+).
\end{equation*}
We notice that
\begin{equation*}
f_1(-x)=i\sqrt{x}\qquad\mbox{and}\qquad f_2(-x)=-i\sqrt{x},\qquad\forall\ x\in\mathbb{R}^+.
\end{equation*}
Therefore, the ``largest" holomorphic extension of $f_\mathbb{R}$ in $\mathbb{C}$ is a multivalued function. To find out a single-valued holomorphic extension of a given holomorphic function $f$, Weierstrass constructed the domain of existence of $f$ (which is a Riemann domain over $\mathbb{C}$, see \cite{Fritzsche2002001B}) by analytic continuations (mentioned in \cite[Page 67]{Teleman2003001B}).

In the quaternionic story, the skew field of quaternions $\mathbb{H}$ can be decomposed as
\begin{equation*}
\mathbb{H}=\bigcup_{I\in\mathbb{S}}\mathbb{C}_I.
\end{equation*}
where
\begin{equation*}
\mathbb{S}:=\{q\in\mathbb{H}:q^2=-1\}\qquad\mbox{and}\qquad\mathbb{C}_I:=\{x+yI:x,y\in\mathbb{R}\}.
\end{equation*}
Obviously, $f_\mathbb{R}$ can extend holomorphically to its domain of existence in $\mathbb{C}_I$ for each $I\in\mathbb{S}$. This implies that there also exists a multivalued slice regular extension of $f_\mathbb{R}$ on $\mathbb{H}\backslash\{0\}$.
Thence, we should develop in our quaternionic setting a theory of Riemann domains over $\mathbb{H}$ which Weierstrass has done in complex analysis.

Since the slice regularity of slice regular functions depends only on each slice complex plane, the Euclidean topology is not suitable for the definition of Riemann domains over $\mathbb{H}$. Instead, we introduce a new topology $\tau_s$ on $\mathbb{H}$, called the slice-topology (see Definition \ref{df-st}). This topology is strictly finer than the Euclidean topology (see Proposition \ref{pr-st}). Therefore a Euclidean connected set may not be connected in the slice-topology, e.g., a ball in the $\mathbb H$ may fail to be a domain in $(\mathbb H,\tau_s)$ (see Example \ref{ex-sd}). We define slice regular functions on open sets in $(\mathbb H,\tau_s)$ (see Definition \ref{df-sr}). Fortunately, the slice regularity of these functions is similar to the classical one. There also exist a splitting lemma and an identity principle for slice regular functions on domains in $(\mathbb H,\tau_s)$ (see Lemma \ref{lm-sp} and Theorem \ref{th-dps}).

Following the classical idea in complex analysis (see e.g. \cite[Pages 87-91]{Fritzsche2002001B}), we introduce a concept of Riemann slice-domains and prove some fundamental properties in Section \ref{sc-rsh}. To describe axial symmetry in Riemann slice-domains, we introduce finite-part paths, which in Riemann slice-domains play a similar role of the Cartesian coordinate in quaternions (see Section \ref{sc-rf}). In the classical theory of slice regular functions, points $x+yI$ and $x+yJ$ in $\mathbb{H}$ are treated as axially symmetric points, where $x,y\in\mathbb{R}$ and $I,J\in\mathbb{S}$. Now, we fix a certain path (a so-called $N$-part path, see Definition \ref{df-1npp}) $\gamma$ in $\mathbb{C}$, and choose some imaginary units $J_\imath\in\mathbb{S}^N$, $\imath=1,2,..,2^N$. Then we get liftings $\gamma^{J_\imath}_\mathcal{G}$, $\imath=1,2,..,2^N$, of the path $\gamma$ in $G$ (see Definition \ref{df:1lni} and Proposition \ref{pr-1gph}), for some Riemann slice-domain $\mathcal{G}=(G,\pi,x_0)$. And we treat $\gamma^{J_\imath}_\mathcal{G}(1)$, $\imath=1,2,..,2^N$, as axially symmetric points. Due to the topological intricacy of Riemann slice-domains, the representation formula demands $2^N$ axially symmetric points (see Theorem \ref{th-dere}), while the classical one \cite[Theorem 3.2]{Colombo2009001} needs only two such points.

We introduce real-connected sets in Riemann slice-domains (see \ref{df-rc}) and then prove that for any two points in Riemann slice-domains can be connected by a finite-part path (see Theorem \ref{th-efpsp}). We also define the unions of Riemann slice-domains and envelopes of regularity in Sections \ref{sc-urs} and \ref{sc-esh}, following \cite[Pages 91-100]{Fritzsche2002001B}.

We rectify the classical extension formula in the theory of slice regular functions (see Proposition \ref{pr-xfh}). And then we prove a new representation formula over slice-domains of regularity (see Theorem \ref{th-dere}), by some technical lemmas (see Lemmas \ref{le-ct} and \ref{pr-ct}). In particular, if the slice-domain of regularity is an axially symmetric slice domain in $\mathbb{H}$, our representation formula is the same as the classical one (see Remark \ref{rm-sdr}).

In the coming article \cite{Dou2018002}, we will generalize the representation formula from slice-domains of regularity to Riemann slice-domains. This result allows us to extend the $*$-product to some suitable Riemann slice-domains by introducing holomorphic stem systems and tensor holomorphic functions. In the theory of slice-regular functions, the $*$-product is introduced in \cite{Gentili2008001} and generalized \cite{Colombo2009001} to axially symmetric slice domains in $\mathbb{H}$ with important applications \cite{Alpay2013001,Gentili2014001}. It is difficult to discriminate on which Riemann slice-domains we can introduce the $*$-product of slice regular functions. In fact, it is related to analytic continuations, which in the case of several complex variables initiates the theories of analytic spaces and sheaf cohomology.

\section{A counterexample of the extension theorem over quaternions}\label{sc-ceet}
In this section, we give a counterexample of the extension theorem (see \cite[Theorem 4.1]{Colombo2009001}). This counterexample shows that the slice regular extension also initiates the multivalued functions in $\mathbb{H}$ and the Euclidean topology is not suitable for analytic continuations. Thence we introduce the concept of Riemann slice-domains in Sections \ref{sc-rsh}, by a new topology on quaternions (see Section \ref{sc-sdh}).

Let $\mathbb{S}$ be the 2-sphere of imaginary units of quaternions $\mathbb{H}$, i.e.,
\begin{equation*}
\mathbb{S}:=\{q\in\mathbb{H}:q^2=-1\}.
\end{equation*}
For each $I\in\mathbb{S}$, let
\begin{equation*}
\mathbb{C}_I:=\{x+yI:x,y\in\mathbb{R}\}
\end{equation*}
be the complex slice plane of $\mathbb{H}$ containing $I$, and the topology of $\mathbb{C}_I$ be the Euclidean topology. The skew field of quaternions $\mathbb{H}$ can be decomposed as
\begin{equation*}
\mathbb{H}=\bigcup_{I\in\mathbb{S}}\mathbb{C}_I.
\end{equation*}
For each subset $U$ of $\mathbb{H}$ and $I\in\mathbb{S}$, we set
\begin{equation*}
U_\mathbb{R}:=U\cap\mathbb{R}\qquad\mbox{and}\qquad U_I:=U\cap\mathbb{C}_I.
\end{equation*}

\begin{defn}
	Let  $I\in\mathbb{S}$	and  $\Omega$  be an open set in $\mathbb{C}_I$.  A function $f:\Omega\rightarrow\mathbb{H}$ is said to be (left) holomorphic, if f has continuous partial derivatives and satisfies
	\begin{equation*}
	\bar\partial_I f(x+yI):=\frac{1}{2}(\frac{\partial}{\partial x}+I \frac{\partial}{\partial y}) f(x+yI)=0\qquad\mbox{on}\qquad\Omega.
	\end{equation*}
\end{defn}

In the following, we endow $\mathbb{H}$ with the Euclidean topology.

\begin{defn} \label{def:slicefun}
	Let $\Omega$ be a domain in $\mathbb{H}$. A function $f:\Omega\rightarrow\mathbb{H}$ is said to be (left) slice regular if, for each $I\in\mathbb{S}$, the restriction of f to $\Omega_I$ denoted by $f_I$ is left holomorphic.
\end{defn}

\begin{defn} \label{def:slice-domain}
	Let $\Omega$ be a domain in $\mathbb{H}$. We say that $\Omega$ is a slice domain if, $\Omega\cap\mathbb{R}$ is nonempty and $\Omega_I$ is a domain in $\mathbb{C}_I$ for all $I\in\mathbb{S}$.
\end{defn}

\begin{thm}\label{th-etq}
	(\cite[Theorem 4.1]{Colombo2009001}, Extension Theorem)
	Let $\Omega\subseteq\mathbb{H}$ be an slice domain, and let $f:\Omega\rightarrow\mathbb{H}$ be a slice regular function. There exists a unique slice regular extension $\widetilde{f}:\widetilde{\Omega}\rightarrow\mathbb{H}$, where $$\widetilde{\Omega}:=\bigcup_{x+yI\in\Omega}x+y\mathbb{S}$$
	is the axially symmetric completion of $\Omega$.
\end{thm}

Let $i$ be an imaginary unit in the complex field $\mathbb{C}$. We consider the field isomorphism  $P_I:\mathbb{C}\rightarrow\mathbb{C}_I$, defined by
\begin{equation*}
P_I(x+yi):=x+yI,\qquad\forall\ x,y\in\mathbb{R}\ \mbox{and}\ I\in\mathbb{S}.
\end{equation*}
Then, for each $I,J\in\mathbb{S}$,
\begin{equation*}
P_I^J:=P_J\circ P_I^{-1}:\mathbb{C}_I\rightarrow\mathbb{C}_J
\end{equation*}
is also an isomorphism.

For each $x\in\mathbb{R}$, we set
\begin{equation*}
\lfloor x\rfloor:=\max\{N\in\mathbb{Z}:N\le x\}
\end{equation*}
the floor integral part of $x$,
\begin{equation*}
\lceil x\rceil:=\min\{N\in\mathbb{Z}:N\ge x\}
\end{equation*}
the ceiling integral part of $x$, and
\begin{equation*}
\{x\}:=x-\lfloor x\rfloor
\end{equation*}
the fractional part of $x$.

A path in a topological space $X$ is a continuous function $f$ from the unit interval $[0,1]$ to $X$.
We fix $I\in\mathbb{S}$ and let $\gamma_0,\gamma_1$ be two (continuous) paths in $\mathbb{C}_I$, defined by
\begin{equation*}
\gamma_\imath (t):=\left\{
\begin{split}
&-1+2I+e^{(-1)^\imath 2\pi I t}, \qquad&&t\in[0,\frac{1}{2}],
\\&2I-(1-t)^{-1}, &&t\in(\frac{1}{2},1],
\end{split}
\right.\qquad\imath=0,1.
\end{equation*}

We set paths
\begin{equation*}
\gamma_s:=(1-s)\gamma_0+s\gamma_1,\qquad\forall\ s\in[0,1],
\end{equation*}
and  define a function
\begin{equation*}
g:\mathbb{R}^+ +\{2I\}\rightarrow\mathbb{R},\qquad x+2I\mapsto\ln(x), \qquad\forall\ x\in\mathbb R^+,
\end{equation*}
where $\ln:\mathbb{R}^+\rightarrow\mathbb{R}$ is the natural logarithm function,
\begin{equation*}
\mathbb{R}^+:=\{x\in\mathbb{R}:x>0\}\qquad\mbox{and}\qquad\mathbb{R}^+ +\{2I\}:=\{x+2I\in\mathbb{H}:x\in\mathbb{R}^+\}.
\end{equation*}

For each $t\in[0,1]$, we consider
$$f_t:\mathbb{C}_I\backslash\gamma_t\rightarrow\mathbb{C}_I,$$
which is the holomorphic extension of $g$.
By the identity principle, we have
\begin{equation*}
f_s|_{\mathbb{R}}=f_t|_{\mathbb{R}}=:f_{\mathbb{R}}, \qquad\forall\  s,t\in[0,1].
\end{equation*}

For each $J\in\mathbb{S}$, we set
\begin{equation*}
T(J):=\min\{|J-I|,1\}\qquad\mbox{and}\qquad\mathbb{C}_J^+:=\{x+yJ\in\mathbb{C}_J:x,y\in\mathbb{R},y\ge 0\},
\end{equation*}
and let $U_J^+$ be the subset of $\mathbb{C}_J^+$ by removing a path corresponding to  $\gamma_{T(J)}$, i.e.,
\begin{equation*}
U_J^+:=\mathbb{C}_J^+\backslash P_I^J (\gamma_{T(J)}).
\end{equation*}
For each $J\in\mathbb{S}$,  we define a function on $U_J^+$
\begin{equation*}
F_J^+(x+yJ):=\frac{1-JI}{2}f_{T(J)}(x+yI)+\frac{1+JI}{2}f_{T(J)}(x-yI)
\end{equation*}
for each $x,y\in\mathbb{R}$ with $x+yJ\in U_J^+$.
By direct calculation (see the proof of \cite[Theorem 3.2]{Colombo2009001}), $F_J^+$ is a holomorphic extension of $f_\mathbb{R}:\mathbb{R}^+\rightarrow\mathbb{H}$ on $U_J^+$ in the complex slice plane $\mathbb{C}_J$. Then there exists a slice regular extension $G:\Omega\rightarrow\mathbb{H}$ of $f_{\mathbb{R}}$ on
\begin{equation*}
\Omega:=\bigcup_{J\in\mathbb{S}}U_J^+,
\end{equation*}
defined by
\begin{equation}\label{eq-ce}
G|_{U_J^+}:=F_J^+,\qquad\forall J\in\mathbb{S}.
\end{equation}
Note that
$T:\mathbb{S}\rightarrow[0,1]$ and $\gamma:[0,1]^2\rightarrow\mathbb{H}$ are continuous, where
\begin{equation*}
\gamma(s,t):=\gamma_s(t).
\end{equation*}
It follows that $$\mathbb{H}\backslash\Omega=\bigcup_{{J\in\mathbb{S}}} P_I^J(\gamma_{T(J)})$$ is closed in $\mathbb{H}$. Thus $\Omega$ is an open set in $\mathbb{H}$. We notice that $$\Omega\cap\mathbb{C}_J=U_J^+\cup U_{-J}^+$$ is a domain in $\mathbb{C}_J$, and since $\mathbb{R}\subset\Omega\cap\mathbb{C}$, then $\Omega$ is path-connected. Therefore, $\Omega$ is a slice domain in $\mathbb{H}$.

\begin{exa}\label{th-ce}
	(A counterexample to  Theorem \ref{th-etq}) There is no slice regular extension on $\widetilde\Omega$ of the slice regular function $G$.
\end{exa}

\begin{proof}
	Suppose Theorem \ref{th-etq} holds. This means there exists a slice regular extension $\widetilde{G}:\widetilde{\Omega}\rightarrow\mathbb{H}$ of $G$. Let $\mathcal{L}:[0,1]\rightarrow\mathbb{C}_I$ be a path in $\mathbb{C}_I$, denoted by
	\begin{equation*}
	\mathcal{L}(t):=2I-1-1/t,\qquad\forall\ t\in(0,1].
	\end{equation*}
	We set
	\begin{equation*}
	A:=\mathcal{L}([0,1])\cup\{2I\}.
	\end{equation*}
	According to the fact in complex analysis, the function $g$ can not extend holomorphically to $\mathbb{C}_I^+\backslash A$. However, $\widetilde{G}$ is a holomorphic extension of $g$ and $\mathbb{C}_I^+\backslash A\subset\widetilde{\Omega}\cap\mathbb{C}_I$, which is a contradiction.
\end{proof}

\section{Slice-domains in $\mathbb{H}$}\label{sc-sdh}
In this section, we will define slice-domains by introducing a new topology. In the classical case, a slice regular function is defined on a domain $\Omega$ in $\mathbb{H}$, see Definition \ref{def:slicefun}. However, the analyticity of this function only depends on each slice complex plane $\Omega_I:=\Omega\cap\mathbb{C}_I$ of $\Omega$, where $I\in\mathbb{S}$. Consequently, there is little relevance between domains in $\mathbb{H}$ and analytic continuations, e.g., the domain of convergence for a power series is not a domain in $\mathbb{H}$, when its center is not on $\mathbb{R}$ (see \cite[Theorem 8]{Gentili2012001}). In order to define the concept of Riemann slice-domains in Section \ref{sc-rsh},  we now introduce a new topology on $\mathbb{H}$, the so-called slice-topology.

For each $I\in\mathbb{S}$, let
\begin{equation*}
\mathbb{C}'_I:=(\mathbb{C}_I,I)=\{(z,I):z\in\mathbb{C}_I\}
\end{equation*}
be a field, and we set a surjective map
\begin{equation*}
\varphi:\bigsqcup_{J\in\mathbb{S}} \mathbb{C}'_J\rightarrow\mathbb{H},\qquad (q,K)\mapsto q,\qquad\forall\ K\in\mathbb{S}\ \mbox{and}\ q\in\mathbb{C}_K.
\end{equation*}
Obviously, for each $I\in\mathbb{S}$, $P_I^{-1}\circ\varphi|_{\mathbb{C}_I'}$ is a field isomorphism from $\mathbb{C}_I'$ to $\mathbb{C}$. For each $I\in\mathbb{S}$, let
\begin{equation*}
\tau(\mathbb{C}'_I):=\{\varphi|_{\mathbb{C}_I'}^{-1}(P_I(U)):U\in\tau(\mathbb{C})\}
\end{equation*}
be a topology induced on $\mathbb{C}_I'$ by $P_I^{-1}\circ\varphi|_{\mathbb{C}_I'}$, where $\tau(\mathbb{C})$ is the Euclidean topology of $\mathbb{C}$. And let $\tau(\sqcup_{I\in\mathbb{S}} \mathbb{C}'_I)$ be the disjoint union topology. Let $\tau_s(\mathbb{H})$ be the quotient space topology induced by $\varphi$. Then $\varphi$ is the quotient map.

\begin{defn}\label{df-st}
	We call the topology $\tau_s$, the slice-topology of $\mathbb{H}$.
\end{defn}

Open sets, connectedness and paths in the slice-topology are  called respectively as slice-open sets, slice-connectedness and slice-paths, and so on.

\begin{prop} \label{pr-sdo}
	A subset $\Omega$ of $\mathbb{H}$ is slice-open  if and only if, $\Omega_I$ is open in $\mathbb{C}_I$ for each $I\in\mathbb{S}$.
\end{prop}

\begin{proof}
	1. Suppose $\Omega$ is a slice-open set in $\mathbb{H}$. Since the quotient map $\varphi$ is continuous, it follows that $\varphi^{-1}(\Omega)$ is open in $\sqcup_{I\in\mathbb{S}}\mathbb{C}'_I$. Thus for each $I\in\mathbb{S}$, $\varphi^{-1}(\Omega)\cap\mathbb{C}'_I$ is open in $\mathbb{C}'_I=(\mathbb{C}_I,I)$. We notice that
	\begin{equation*}
	(\Omega_I,I)=\varphi^{-1}(\Omega)\cap\mathbb{C}'_I,
	\end{equation*}
	then $\Omega_I$ is open in $\mathbb{C}_I$.
	
	2. Suppose for each $I\in\mathbb{S}$, $\Omega_I$ is open in $\mathbb{C}_I$. We notice that
	\begin{equation*}
	\varphi^{-1}(\Omega)\cap\mathbb{C}'_I=(\Omega_I,I)
	\end{equation*}
	is open in $\sqcup_{J\in\mathbb{S}} \mathbb{C}'_J$. It follows that
	\begin{equation*}
	\varphi^{-1}(\Omega)=\bigsqcup_{J\in\mathbb{S}}(\varphi^{-1}(\Omega)\cap\mathbb{C}'_J)
	\end{equation*}
	is open in $\sqcup_{J\in\mathbb{S}} \mathbb{C}'_J$. It indicates that $\Omega$ is slice-open in $\mathbb{H}$.
\end{proof}

According to Proposition \ref{pr-sdo}, for each open set $U$ in $\mathbb{H}$, $U$ is slice-open. Therefore slice topology is Hausdorff.

Fix $I\in\mathbb{S}$. For any $J\in\mathbb{S}$, we consider an ellipse in $\mathbb{C}_J$, defined by
\begin{equation*}
U_J:=\left\{
\begin{aligned}
&\{x+yJ\in\mathbb{C}_J:x^2+\frac{y^2}{\mbox{dist}(J,\mathbb{C}_I)}<1\},\qquad&&J\neq\pm I,
\\&\{x+yJ\in\mathbb{C}_J:x^2+y^2<1\}, &&J=\pm I,
\end{aligned}\right.
\end{equation*}
where $\mbox{dist}(J,\mathbb{C}_I)$ is the Euclidean distance from $J$ to $\mathbb{C}_I$. Thanks to Proposition \ref{pr-sdo},
\begin{equation*}
U:=\bigcup_{J\in\mathbb{S}}U_J
\end{equation*}
is a slice-domain including $0$ in $\mathbb{H}$. However, $0$ is not in the Euclidean interior of $U$ (the semi-minor axis $\sqrt{\mbox{dist}(J,\mathbb{C}_I)}$ of $U_J$ tends to zero when $J$ approaches $I$), which means that $U$ is not open in $\mathbb{H}$. In summary, the slice topology is strictly finer than the Euclidean topology.

\begin{prop}\label{pr-st}
	$(\mathbb{H},\tau_s)$ is a Hausdorff space, and $\tau\subsetneq\tau_s$, where $\tau$ is the Euclidean topology of $\mathbb{H}$.
\end{prop}

For each $q\in\mathbb{H}$, $z\in\mathbb{C}$, $r\in\mathbb{R}^+$, and $I\in\mathbb{S}$, we set
\begin{equation*}
\begin{split}
B_{\mathbb{H}}(q,r):=\{p\in\mathbb{H}:|p-q|<r\},\qquad
B_{\mathbb{C}}(z,r):=\{p\in\mathbb{C}:|p-z|<r\},
\\
B_I(q,r):=\mathbb{B}_\mathbb{H}(q,r)\cap\mathbb{C}_I\qquad \mbox{and}\qquad B_{\mathbb{R}}(q,r):=\{p\in\mathbb{R}:|p-q|<r\}.
\end{split}
\end{equation*}

According to Proposition \ref{pr-st}, the slice-topology is strictly finer than the Euclidean topology. It follows that a Euclidean connected set may not connected in slice-topology.

\begin{exa}\label{ex-sd}
	We take the ball $\Omega:=B_{\mathbb H}(I, 1/2)$ as an example. It is a slice-open set according to Proposition \ref{pr-sdo}. However, it is not slice-connected since
its can be partitioned into disjoint open sets
$$\Omega=\Omega_I   \bigsqcup  \left(\bigcup_{J\in\mathbb S\setminus\{\pm I\}} \Omega_J\right). $$
This means that the ball $B_{\mathbb H}(I, 1/2)$ is not a slice-domain in $\mathbb H$.
\end{exa}

Notice that the subspace topology on $\mathbb{R}$ induced by $\tau_s$ coincides with the Euclidean topology on $\mathbb{R}$, thereby slice-connectedness and connectedness coincide in $\mathbb{R}$.

\begin{defn} 
	A subset $U$ of $\mathbb{H}$ is called real-connected, if $U_{\mathbb{R}}$ is connected in $\mathbb{R}$.
\end{defn}

We remark that the empty set is always taken to be a connected set in this paper.

\begin{prop} \label{pr-rc1} 
	Let $U$ be a slice-open set in $\mathbb{H}$. Then for each $q\in U$, there exists a real-connected slice-domain $V\subset U$ containing $q$.
\end{prop}

\begin{proof}
    If $q\in\mathbb{R}$, let $A$ be the connected component of $U_{\mathbb{R}}$ containing $q$ in $\mathbb{R}$; otherwise, let $A$ be the empty set. We notice that
    \begin{equation*}
    ((U\backslash U_{\mathbb{R}})\cup A)\cap\mathbb{R}=((U\backslash U_{\mathbb{R}})\cap\mathbb{R})\cup(A\cap\mathbb{R})=\varnothing\cup A=A
    \end{equation*}
    is connected in $\mathbb{R}$. Consequently, the slice-connected component of $(U\backslash U_{\mathbb{R}})\cup A$ containing $q$ is a real-connected slice-domain.
\end{proof}

\begin{defn} 
	Let $U$ be a subset of $\mathbb{H}$. A slice-path $\gamma$ in $\mathbb{H}$ is called a slice preserving path in $U$, if there exists $I\in\mathbb{S}$ such that $\gamma([0,1])\subset U_I$.
\end{defn}

For each $N\in\mathbb{N}^+$, topological space $X$ and paths $\gamma_\imath$, $\imath=1,2,...,N$ in $X$, we denote the composition of paths $\{\gamma_\imath\}_{\imath=1}^N$ by
\begin{equation*}
\prod_{l=1}^N \gamma_l:=\gamma_1\gamma_2....\gamma_N,
\end{equation*}
i.e.,
\begin{equation*}
(\prod_{l=1}^N\gamma_l)(t):=\left\{
\begin{split}
&\gamma_{\lfloor tN\rfloor+1}(\{tN\}),\qquad &&t\in[0,1),
\\&\gamma(1), &&t=1.
\end{split}
\right.
\end{equation*}

\begin{prop} \label{pr-rc2} 
	Let $U$ be a slice-domain in $\mathbb{H}$. The following assertions hold.
	
	(a). If $U_\mathbb{R}=\varnothing$, then there exists $I\in\mathbb{S}$ such that $U\subset\mathbb{C}_I$.
	
	(b). If $U$ is real-connected with $U_\mathbb{R}\neq\varnothing$, then for each $q\in U$ and $x\in\mathbb{R}$, there exists a slice preserving path $\gamma$ from $q$ to $x$.
	
	(c). If $U$ is real-connected, then for each $I\in\mathbb{S}$, $U_I$ is a domain in $\mathbb{C}_I$.
	
	(d). If $U$ is real-connected, then for each $p,q\in U$, there exist two slice preserving paths $\gamma_1,\gamma_2$ in $U$ such that
	\begin{equation*}
	\gamma_1(0)=p,\qquad\gamma_1(1)=\gamma_2(0)\qquad\mbox{and}\qquad\gamma_2(1)=q.
	\end{equation*}
\end{prop}

\begin{proof}
	(a). We notice that for each $J\in\mathbb{S}$, $\mathbb{C}_J\backslash\mathbb{R}$ is slice-open in $\mathbb{H}$. If $U_\mathbb{R}=\varnothing$, then
	\begin{equation*}
	U\subset\bigsqcup_{J\in\mathbb{S}'}(\mathbb{C}_J\backslash\mathbb{R}),
	\end{equation*}
	where $\mathbb{S}'$ is a subset of $\mathbb{S}$ such that for each $J\in\mathbb{S}$, the cardinality of $\mathbb{S}'\cap\{\pm J\}$ is one. By the connectivity of $U$, there exists $I\in\mathbb{S}$ such that $U\subset\mathbb{C}_I\backslash\mathbb{R}$.
	
	(b). For each $q\in U$ and $x\in U_\mathbb{R}$, there exists $I\in\mathbb{S}$ such that $q\in\mathbb{C}_I$. We denote the connected component of $U_I$ containing $q$ by $V$. If $V_\mathbb{R}=\varnothing$, then $V$ and $U\backslash V$ are slice-open. And since $U$ is slice-connected, it follows that
	\begin{equation*}
	V=U\qquad\mbox{and}\qquad U_\mathbb{R}=\varnothing,
	\end{equation*}
	which is a contradiction. Therefore, $V_\mathbb{R}\neq\varnothing$. Let $x_0\in V_\mathbb{R}$. Then there exist a path $\alpha$ in $V$ from $q$ to $x_0$, and a path $\beta$ in $U_\mathbb{R}$ from $x_0$ to $x$. It follows that $\alpha\beta$ is a slice preserving path from $q$ to $x$.
	
	(c). If $U_\mathbb{R}=\varnothing$, according to (a), there exists $I\in\mathbb{S}$ such that $U\subset\mathbb{C}_I$. We suppose $U\neq\varnothing$. Let $V$ be a connected component of $U$ in $\mathbb{C}_I$. We notice that $V$ and $U\backslash V$ are slice-open sets in $\mathbb{H}$. And since $U$ is slice-connected, it follows that
	\begin{equation*}
	V=U\qquad\mbox{and}\qquad U_\mathbb{R}=\varnothing.
	\end{equation*}
	Thence (c) holds.
	
	Otherwise, for each $I\in\mathbb{S}$, $p,q\in U_I$ and $x\in U_\mathbb{R}$, thanks to (b), there exist a slice preserving path $\alpha$ in $U_I$ from $p$ to $x$, and a slice preserving path $\beta$ in $U_I$ from $x$ to $q$. Hence $\alpha\beta$ is a slice preserving path in $U_I$ from $p$ to $q$. It is clear that $U_I$ is path-connected in $\mathbb{C}_I$. And since Proposition \ref{pr-sdo}, $U_I$ is an open set in $\mathbb{C}_I$. It follows that $U_I$ is a domain in $\mathbb{C}_I$.
	
	(d). If $U_\mathbb{R}=\varnothing$, then there exists $I\in\mathbb{S}$ such that $U$ is a domain in $\mathbb{C}_I$, by (a) and (c). Thus (d) holds.
	
	Otherwise, for each $p,q\in U$, let $x\in U_\mathbb{R}$, $\gamma_1$ be a slice preserving path in $U$ from $p$ to $x$, and $\gamma_2$ be a slice preserving path in $U$ from $x$ to $q$. Therefore
	\begin{equation*}
	\gamma_1(0)=p,\qquad\gamma_1(1)=x=\gamma_2(0)\qquad\mbox{and}\qquad\gamma_2(1)=q.
	\end{equation*}
	Hence (d) holds.
\end{proof}

\begin{prop} \label{pr-lp} 
	The topological space $(\mathbb{H},\tau_s)$ is connected and local path-connected. In particular, $(\mathbb{H},\tau_s)$ is path-connected.
\end{prop}

\begin{proof}
	According to Propositions \ref{pr-rc1} and \ref{pr-rc2} (d), $(\mathbb{H},\tau_s)$ is local path-connected. We notice that $\mathbb{H}\cap\mathbb{C}_I=\mathbb{C}_I\supset\mathbb{R}$, then $(\mathbb{H},\tau_s)$ is connected. It follows that $(\mathbb{H},\tau_s)$ is path-connected.
\end{proof}

\begin{defn}\label{df-1npp}
	Let $N\in\mathbb{N}^+$ and $\gamma_\imath$ be a path in $\mathbb{C}$ for each $\imath\in\{1,2,...,N\}$.
	\begin{equation*}
		\gamma:=(\gamma_1,\gamma_2,...,\gamma_N)
	\end{equation*}
	is called an $N$-part path in $\mathbb{C}$, if
	\begin{equation*}
	\gamma_\imath(1)=\gamma_{\imath+1}(0)\in\mathbb{R},\qquad \imath=1,2,...,N-1.
	\end{equation*}
	We set
	\begin{equation*}
		\gamma(t):=\left\{
		\begin{split}
			&\gamma_{\lfloor tN\rfloor+1}(\{tN\}),\qquad &&t\in[0,1),
			\\&\gamma_N(1), &&t=1.
		\end{split}
		\right.
	\end{equation*}
	We say that $\gamma$ is from $\gamma(0)$ to $\gamma(1)$.
	
	The set of all $N$-part paths in $\mathbb{C}$ is denoted by $\mathcal{P}^N(\mathbb{C})$. We set
	\begin{equation*}
		\mathcal{P}^{\infty}(\mathbb{C}):=\bigsqcup_{\imath\in\mathbb{N}^+} \mathcal{P}^\imath(\mathbb{C}).
	\end{equation*}
	We say that the elements of $\mathcal{P}^{\infty}(\mathbb{C})$ are finite-part paths in $\mathbb{C}$.
	
	We call $\gamma(0)$ the initial point of $\gamma$. For each $z\in\mathbb{C}$ and $N\in\mathbb{N}^+$, we denote by $\mathcal{P}^{\infty}_z(\mathbb{C})$ (resp. $\mathcal{P}^{N}_z(\mathbb{C})$) the set of all the finite-part (resp. $N$-part) paths in $\mathbb{C}$ with the initial point $z$.
\end{defn}

\begin{defn} 
	Let $N\in\mathbb{N}^+$ and $\gamma_\imath$ be a slice preserving path in $\mathbb{H}$ for each $\imath\in\{1,2,...,N\}$.
	\begin{equation*}
		\gamma:=(\gamma_1,\gamma_2,...,\gamma_N)
	\end{equation*}
	is called an $N$-part path in $\mathbb{H}$, if
	\begin{equation*}
	\gamma_\imath(1)=\gamma_{\imath+1}(0)\in\mathbb{R},\qquad \imath=1,2,...,N-1.
	\end{equation*}
	We set
	\begin{equation*}
		\gamma(t):=\left\{
		\begin{split}
			&\gamma_{\lfloor tN\rfloor+1}(\{tN\}),\qquad &&t\in[0,1),
			\\&\gamma_N(1), &&t=1.
		\end{split}
		\right.
	\end{equation*}
	We say that $\gamma$ is from $\gamma(0)$ to $\gamma(1)$.
	
	The set of all $N$-part paths in $\mathbb{H}$ is denoted by $\mathcal{P}^N(\mathbb{H})$. We set
	\begin{equation*}
		\mathcal{P}^{\infty}(\mathbb{H}):=\bigsqcup_{\imath\in\mathbb{N}^+} \mathcal{P}^\imath(\mathbb{H}).
	\end{equation*}
	We say that the elements of $\mathcal{P}^{\infty}(\mathbb{H})$ are finite-part paths in $\mathbb{H}$.
\end{defn}

For each $N\in\mathbb{N}^+$ and $I=(I_1,I_2,...I_N)\in\mathbb{S}^N$, we define a map
\begin{equation*}
	\phi_I:\mathcal{P}^N (\mathbb{C})\rightarrow\mathcal{P}^N (\mathbb{H})
\end{equation*}
by
\begin{equation*}
\phi_I(\gamma):=(P_{I_1} (\gamma_1),P_{I_2} (\gamma_2),...,P_{I_N} (\gamma_N))
\end{equation*}
for each $\gamma=(\gamma_1,\gamma_2,...,\gamma_N)\in\mathcal{P}^N(\mathbb{C})$.

Obviously, for each $N\in\mathbb{N}^+$ and $\alpha\in\mathcal{P}^N(\mathbb{H})$, there exist $I\in\mathbb{S}^N$ and $\beta\in\mathcal{P}^(\mathbb{C})$ such that $\phi_I(\beta)=\alpha$.

\begin{defn}\label{df:1lni}
	Let $N\in\mathbb{N}^+$, $I\in\mathbb{S}^N$, and $\gamma\in\mathcal{P}^N(\mathbb{C})$. We call  $\phi_I(\gamma)$  the $I$-lifting of $\gamma$ to $\mathbb{H}$, denoted by $\gamma^I$.
\end{defn}

\section{Splitting lemma and identity principle on slice-domains}
The classical splitting lemma and the identity principle are stated in \cite{Gentili2006001,Gentili2007001}. Now we will prove the corresponding results in the case of slice-domains exactly as \cite[Lemma 2.5]{Gentili2007001} and \cite[Theorem 3.1]{Gentili2007001}.

\begin{defn}\label{df-sr} 
	Let $\Omega$ be a slice-open set in $\mathbb{H}$. A function $f:\Omega\rightarrow\mathbb{H}$ is called (left) slice regular, if for each $I\in\mathbb{S}$, $f_I:=f|_{\Omega_I}$ is left holomorphic.
\end{defn}

\begin{lem}\label{lm-sp} 
	(Splitting Lemma)
	Let $f$ be a function defined on a slice-open subset $\Omega$ of $\mathbb{H}$. Then $f$ is slice regular, if and only if for all $I,J\in\mathbb{S}$ with $I\bot J$, there exist two complex-valued holomorphic functions $F,G: \Omega_I\rightarrow\mathbb{C}_I$ such that
	\begin{equation*}
	f_I(z)=F(z)+G(z)J,\qquad \forall\ z\in\Omega_I.
	\end{equation*}
\end{lem}

\begin{proof}
``$\Rightarrow$" For each $I,J\in\mathbb{S}$ satisfying $I\bot J$, there exist two complex-value functions $F,G: \Omega_I\rightarrow\mathbb{C}_I$ such that $f_I=F+GJ$. If $f$ is slice regular, then
\begin{equation*}
\bar\partial_I F+\bar\partial_I GJ=\bar\partial_I f_I=0.
\end{equation*}
For each $q\in\Omega_I$, we have
\begin{equation*}
\bar\partial_I F(q)\in\mathbb{C}_I\qquad\mbox{and}\qquad\bar\partial_I G(q)J\in\mathbb{C}_I J,
\end{equation*}
thus
\begin{equation*}
\bar\partial_I F(q)=\bar\partial_I G(q)=0.
\end{equation*}
It follows that $F$, $G$ are holomorphic.
	
``$\Leftarrow$" If for each $I,J\in\mathbb{S}$ with $I\bot J$, there exist two holomorphic functions $F_I,G_I:\Omega_I\rightarrow\mathbb{C}_I$ such that
\begin{equation*}
f_I(z)=F_I(z)+G_I(z)J,\qquad \forall\ z\in\Omega_I.
\end{equation*}
Then
\begin{equation*}
\bar\partial_I f=\bar\partial_I F_I+\bar\partial_I G_IJ=0,\qquad\forall\ I\in\mathbb{S}.
\end{equation*}
It follows that $f$ is slice regular.
\end{proof}

\begin{thm} \label{th-dps} 
	(Identity Principle)
	Let $\Omega$ be a real-connected slice-domain in $\mathbb{H}$, and $f,g$ be two slice regular functions on $\Omega$. If there exists $I\in\mathbb{S}$ such that $f$ and $g$ coincide on a subset of $\Omega_I$ having an accumulation point in $\Omega_I$, then $f=g$ on $\Omega$.
\end{thm}

\begin{proof}
	 Note that $\Omega_I\neq\varnothing$, and by Proposition \ref{pr-rc2} (c), we get that $\Omega_I$ is a nonempty domain in $\mathbb{C}_I$. Therefore $f$ and $g$ coincide on $\Omega_I$. If $\Omega_{\mathbb{R}}=\varnothing$, then $\Omega=\Omega_I$ by Proposition \ref{pr-rc2} (a). If $\Omega_{\mathbb{R}}\neq\varnothing$, we have $f=g$ on $\Omega_\mathbb{R}$, and hence on $\Omega_J$ for each $J\in\mathbb{S}$.
	 Consequently,
	 \begin{equation*}
	 f=g\qquad\mbox{on}\qquad\Omega=\bigcup_{J\in\mathbb{S}}\Omega_J.
	 \end{equation*}
\end{proof}

\section{Riemann slice-domains over $\mathbb{H}$}\label{sc-rsh}

In this section, we introduce a concept of Riemann slice-domains over $\mathbb{H}$ following \cite{Fritzsche2002001B}.

\begin{defn}\label{df-rsd} 
	A (Riemann) slice-domain over $\mathbb{H}$ is a pair $(G,\pi)$ with the following properties:
	
	1. $(G,\tau(G))$ is a connected Hausdorff   space,
	
	2. $\pi:G\rightarrow\mathbb{H}$ is a local slice-homeomorphism, i.e., locally homeomorphic with respect to $\tau(G)$ and $\tau_s(\mathbb{H})$.
\end{defn}

\begin{rmk}\label{rk-pc}
	Let $(G,\pi)$ be a Riemann slice domain over $\mathbb{H}$. Then $G$ is path-connected and local path-connected.
\end{rmk}

\begin{proof}
	According to $\mathbb{H}$ is local slice path-connected, $G$ is local path-connected. And since $G$ is connected, $G$ is path-connected.
\end{proof}

\begin{rmk}\label{rk-p}
	Let $(G,\pi)$ be a slice-domain over $\mathbb{H}$ and $U$ be an open set in $G$. Then $\pi(U)$ is a slice-open set in $\mathbb{H}$.

Moreover, if $U$ is a domain in $G$, then $\pi(U)$ is a slice-domain in $\mathbb{H}$.
\end{rmk}

\begin{proof}
	Since $\pi$ is locally slice-homeomorphic, $\pi(U)$ is slice-open in $\mathbb{H}$.
	
	If $U$ is a domain in $G$, and since $G$ is local path-connected, $U$ is local path-connected. It follows that $U$ is path-connected. For each $p,q\in U$, there exists a path $\alpha$ in $U$ from $p$ to $q$. Then $\pi(\alpha)$ is a slice-path from $\pi(p)$ to $\pi(q)$. It is clear that $\pi(U)$ is slice path-connected. Then $\pi(U)$ is a slice-domain in $\mathbb{H}$.
\end{proof}

\begin{defn} 
	A (Riemann) slice-domain over $\mathbb{H}$ with distinguished point is a triple $\mathcal{G}=(G,\pi,x)$ for which $(G,\pi)$ is a slice-domain over $\mathbb{H}$ and $x\in G$. We denote all the slice-domains over $\mathbb{H}$ with distinguished point by $\mathscr{R}$.
\end{defn}

\begin{prop}\label{pr-ul}
	(On the uniqueness of lifting) Let $(G,\pi)$ be a slice-domain over $\mathbb{H}$ and $X$ be a connected topological space. If $x\in X$ is a point and $\psi_1,\psi_2:X\rightarrow G$ are continuous mappings with $\psi_1(x)=\psi_2(x)$ and $\pi\circ\psi_1=\pi\circ\psi_2$, then $\psi_1=\psi_2$.
\end{prop}

\begin{proof}
	We prove this proposition with the same approach as in \cite[Proposition 8.1]{Fritzsche2002001B}. We set
	\begin{equation*}
	A:=\{y\in X:\psi_1(y)=\psi_2(y)\}.
	\end{equation*}
	By assumption, $x\in A$. It is clear that $A\neq\varnothing$. Since $G$ is a Hausdorff space, it follows immediately that $A$ is closed. Now let $y\in A$ be chosen arbitrarily, and set
	\begin{equation*}
	q:=\psi_1(y)=\psi_2(y).
	\end{equation*}
	There exists a domain $U\subset G$ containing $q$ such that $\pi|_U:U\rightarrow\pi(U)$ is a slice-homeomorphism. Let $V$ be the intersection of the preimage of $U$ under $\psi_1$ and $\psi_2$, i.e.,
	\begin{equation*}
	V:=\psi_1^{-1}(U)\cap\psi_2^{-1}(U).
	\end{equation*}
	According to
	\begin{equation*}
	\pi\circ\psi_1=\pi\circ\psi_2,
	\end{equation*}
	we have
	$$\psi_1|_V=(\pi|_U)^{-1}\circ\pi\circ\psi_1|_V=(\pi|_U)^{-1}\circ\pi\circ\psi_2|_V=\psi_2|_V,$$
	and therefore $V\subset A$. Since $\psi_1,\psi_2$ are continuous, $V$ is open in $X$. Hence $A$ is open in $X$, and since $X$ is connected and closed, it follows that $A=X$.
\end{proof}

\begin{defn}
	Let $\mathcal{G}_\lambda=(G_\lambda,\pi_\lambda,x_\lambda)$, $\lambda=1,2$, be two slice-domains over $\mathbb{H}$ with distinguished point. We say that $\mathcal{G}_1$ is contained in $\mathcal{G}_2$ (denoted by $\mathcal{G}_1\prec\mathcal{G}_2$), if there exists a continuous map $\varphi:G_1\rightarrow G_2$ with the following properties:
	
	1. $\pi_2\circ\varphi=\pi_1$ (called ``$\varphi$ preserves fibers").
	
	2. $\varphi(x_1)=x_2$.
\end{defn}

\begin{prop}\label{pr-ufpm}
	Let $\mathcal{G}_\lambda=(G_\lambda,\pi_\lambda,x_\lambda)$, $\lambda=1,2$ be two slice-domains over $\mathbb{H}$ with distinguished point. If $\mathcal{G}_1\prec\mathcal{G}_2$, then the fiber preserving map $\varphi:G_1\rightarrow G_2$ with $\varphi(x_1)=x_2$ is uniquely determined. We call $\varphi$ the fiber preserving map from $\mathcal{G}_1$ to $\mathcal{G}_2$.
\end{prop}

\begin{proof}
	This follows immediately from Proposition \ref{pr-ul}.
\end{proof}

\begin{prop}\label{pr-movebase}
	Let $\mathcal{G}_\lambda=(G_\lambda,\pi_\lambda,x_\lambda)$, $\lambda=1,2$ be slice-domains over $\mathbb{H}$ with distinguished point. If $\mathcal{G}_1\prec\mathcal{G}_2$ and $\varphi$ is the fiber preserving map from $\mathcal{G}_1$ to $\mathcal{G}_2$, then $(G_1,\pi_1,y)\prec(G_2,\pi_2,\varphi(y))$ for each $y\in G_1$.
\end{prop}

\begin{proof}
	If $\mathcal{G}_1\prec\mathcal{G}_2$, then $\varphi$ is also the fiber preserving from $(G_1,\pi_1,y)$ to $(G_2,\pi_2,\varphi(y))$, i.e., $(G_1,\pi_1,y)\prec(G_2,\pi_2,\varphi(y))$ for each $y\in G_1$.
\end{proof}

\begin{prop}\label{pr-weak}
	For each slice-domains $\mathcal{G}_\lambda$ over $\mathbb{H}$ with distinguished point with $\lambda=1,2,3$, we have
	
	1. $\mathcal{G}_1\prec\mathcal{G}_1$.
	
	2. If $\mathcal{G}_1\prec\mathcal{G}_2$ and $\mathcal{G}_2\prec\mathcal{G}_3$, then $\mathcal{G}_1\prec\mathcal{G}_3$.
\end{prop}

\begin{proof}
	This proposition is proved directly by the definition.
\end{proof}

\begin{defn}
	Two slice-domains $\mathcal{G}_1,\mathcal{G}_2$ over $\mathbb{H}$ with distinguished point are called isomorphic or equivalent (symbolically $\mathcal{G}_1\cong\mathcal{G}_2$), if $\mathcal{G}_1\prec\mathcal{G}_2$ and $\mathcal{G}_2\prec\mathcal{G}_1$. We denote the equivalence class of $\mathcal{G}_1$ in $\mathscr{R}$ by $[\mathcal{G}_1]$.
\end{defn}

\begin{prop}\label{pr-fpei}
	Let $\mathcal{G}_\imath=(G_\imath,\pi_\imath,x_\imath)$, $\imath=1,2$ be two slice-domains over $\mathbb{H}$ with distinguished point. If $\mathcal{G}_1\cong\mathcal{G}_2$, and $\varphi_\imath:G_\imath\rightarrow G_{3-\imath}$, $\imath=1,2$ is the fiber preserving map from $\mathcal{G}_\imath$ to $\mathcal{G}_{3-\imath}$, then $\varphi_1$ is a homeomorphism from $G_1$ to $G_2$, and $\varphi_1^{-1}=\varphi_2$.
\end{prop}

\begin{proof}
	We notice that
	\begin{equation*}
	\varphi_1\circ\varphi_2:G_1\rightarrow G_1
	\end{equation*}
	is a fiber preserving map from $\mathcal{G}_1$ to $\mathcal{G}_1$. On the other hand, $id_{G_1}$ is also a fiber preserving map from $\mathcal{G}_1$ to $\mathcal{G}_1$. According to Proposition \ref{pr-ufpm},
	\begin{equation*}
	\varphi_1\circ\varphi_2={\id}_{G_1}.
	\end{equation*}
	Similarly, we also have
	\begin{equation*}
	\varphi_2\circ\varphi_1={\id}_{G_2}.
	\end{equation*}
	Then
	\begin{equation*}
	\varphi_1^{-1}=\varphi_2.
	\end{equation*}
	And since $\varphi_1$ and $\varphi_1^{-1}=\varphi_2$ are continuous. Then $\varphi_1$ is a homeomorphism.
\end{proof}

\begin{defn}\label{df-1asd}
	A slice-domain $\mathcal{G}_1=(G,\pi,x)$ over $\mathbb{H}$ with distinguished point is called schlicht, if it is isomorphic to the slice-domain $\mathcal{G}_0=(\pi(G),id_{\pi(G)},\pi(x))$ over $\mathbb{H}$ with distinguished point with $\pi(G)$ being a slice-domain in $\mathbb{H}$.

 We call $(G,\pi)$   schlicht  if $(G,\pi,y)$ is schlicht for some $y\in G$. A domain $U$ in $G$ is called schlicht (with respect to $\pi$), if $(U,\pi|_U)$ is a schlicht slice-domain over $\mathbb{H}$.
\end{defn}

Let $(G,\pi,x)$ be a slice-domain over $\mathbb{H}$ with distinguished point. According to Proposition \ref{pr-movebase}, $(G,\pi,x)$ is schlicht, if and only if $(G,\pi,y)$ is schlicht for each $y\in G$. Then we have $(G,\pi,x)$ is schlicht, if and only if $(G,\pi)$ is schlicht (or $G$ is schlicht with respect to $\pi$).

\section{Real-connectedness and finite-part paths}\label{sc-rf}

In this section, we introduce two new concepts, the real-connectedness and finite-part paths. Then we will prove that any two points in a Riemann slice domain over $\mathbb{H}$ can be connected by a finite-part path (see Theorem \ref{th-efpsp}).

\subsection{Real-connectedness} We will introduce a technical concept, real-connectedness. It provides a tool for the proof of Theorem \ref{th-efpsp}.

For each slice-domain $(G,\pi)$ over $\mathbb{H}$ (resp. slice-domain $(G,\pi,x)$ over $\mathbb{H}$ with distinguished point), $I\in\mathbb{S}$, and $U\subset G$, we set
\begin{equation*}
U_I:=\{q\in U:\pi(q)\in\mathbb{C}_I\}\qquad\mbox{and}\qquad U_\mathbb{R}:=\{q\in U:\pi(q)\in\mathbb{R}\}.
\end{equation*}

\begin{defn}\label{df-rc} 
	Let $(G,\pi)$ be a slice-domain over $\mathbb{H}$. A domain $U\subset G$ is called real-connected (with respect to $\pi$), if $U_\mathbb{R}$ is connected in $G_\mathbb{R}$. If $G$ is real-connected with respect to $\pi$,  then we call both $(G,\pi,x)$ and $(G,\pi)$ are  real-connected   for any $x\in G$.
\end{defn}

\begin{prop}\label{pr-rceq} 
	If a slice-domain $\mathcal{G}$ over $\mathbb{H}$ with distinguished point is real-connected, then $\mathcal{G}'$ is real-connected for each $\mathcal{G}'\in[\mathcal{G}]$.
\end{prop}

\begin{proof}
	Suppose there exists $\mathcal{G}'=(G',\pi',x')\in[\mathcal{G}]$ being not real-connected. Then there exist two nonempty open sets $U_1$ and $U_2$ in $G'_\mathbb{R}$ with
	\begin{equation*}
	U_1\cap U_2=\varnothing,\qquad G'_{\mathbb{R}}\subset U_1\cup U_2\qquad\mbox{and}\qquad U_\imath\cap G'_{\mathbb{R}}\neq\varnothing,\ \imath=1,2.
	\end{equation*}
	We write $\mathcal{G}=(G,\pi,x)$, and let $\varphi:G\rightarrow G'$ be the fiber preserving map from $\mathcal{G}'$ to $\mathcal{G}$. According to Proposition \ref{pr-fpei}, $\varphi$ is a homeomorphism. And since $\varphi(G'_\mathbb{R})=G_\mathbb{R}$, it follows that $\varphi(U_1)$ and $\varphi(U_2)$ are two open sets in $G_\mathbb{R}$ with
	\begin{equation*}
	\varphi(U_1)\cap \varphi(U_2)=\varnothing,\qquad G_{\mathbb{R}}\subset \varphi(U_1)\cup \varphi(U_2)\qquad\mbox{and}\qquad\varphi(U_\imath)\cap G'_{\mathbb{R}}\neq\varnothing,\ \imath=1,2.
	\end{equation*}
	Then $G_\mathbb{R}$ is not connected in $G$. Therefore $G$ is not real-connected, which is a contradiction.
\end{proof}

\begin{prop}\label{pr-srch}
	Let $(G,\pi)$ be a slice-domain over $\mathbb{H}$ with distinguished point, and $U$ be a schlicht real-connected domain in $G$. Then $\pi(U)$ is a real-connected slice-domain in $\mathbb{H}$.
\end{prop}

\begin{proof}
	Let $x\in U$. Then $(U,\pi|_U,x)$ is schlicht. It follows that $(U,\pi|_U,x)$ and $(\pi(U),id_{\pi(U)},\pi(x))$ are equivalence. Thanks to Proposition \ref{pr-rceq}, $(\pi(U),id_{\pi(U)},\pi(x))$ is real-connected. Then $\pi(U)_{\mathbb{R}}$ is connected. According to Remark \ref{rk-p}, $\pi(U)$ is a slice-domain in $\mathbb{H}$. Therefore $\pi(U)$ is a real-connected slice-domain in $\mathbb{H}$.
\end{proof}

\begin{prop} \label{pr-sp} 
	Let $(G,\pi)$ be a slice-domain over $\mathbb{H}$, and $U$ be an open set in $G$. For each $q\in U$, there exists a schlicht real-connected domain $V\subset U$ containing $q$.
\end{prop}

\begin{proof}
	For each $q\in U$, there exists a domain $U'$ in $G$ and a slice-domain $V'$ in $\mathbb{H}$, such that $q\in U'$ and $\pi|_{U'}:U'\rightarrow V'$ is slice-homeomorphism. According to Proposition \ref{pr-rc1}, there exists a real-connected slice-domain $W\subset V'$ containing $\pi(q)$. Then $\pi|_{U'}^{-1}(W)$ is a schlicht real-connected domain, and $q\in\pi|_{U'}^{-1}(W)\subset U$.
\end{proof}

\subsection{Finite-part paths} In this subsection, we will introduce finite-part paths, which describe axial symmetry in Riemann slice-domains.

\begin{defn}
	Let $(G,\pi)$ be a slice-domain over $\mathbb{H}$. A path $\gamma$ in $G$ is called slice preserving, if $\pi\circ\gamma$ is a slice preserving path in $\mathbb{H}$.
\end{defn}

\begin{prop} \label{pr-pir}
	Let $(G,\pi)$ be a slice-domain over $\mathbb{H}$, and $U$ be a schlicht real-connected domain in $G$. Then for each $p,q\in U$, there exist  two slice preserving paths $\alpha$ and $\beta$ in $U$, such that the composition  $\alpha\beta$ is  a path from $p$ to $q$.
\end{prop}

\begin{proof}
	This follows immediately from Propositions \ref{pr-rc2} (d) and \ref{pr-srch}.
\end{proof}

\begin{defn}
	Let $(G,\pi)$ be a slice-domain over $\mathbb{H}$, $N\in\mathbb{N}^+$ and $\gamma_\imath$ be a slice preserving path in $G$ for each $\imath\in\{1,2,...,N\}$.
	\begin{equation*}
	\gamma:=(\gamma_1,\gamma_2,...,\gamma_N)
	\end{equation*}
	is called an $N$-part path in $G$, if
	\begin{equation*}
	\gamma_\imath(1)=\gamma_{\imath+1}(0)\in G_{\mathbb{R}},\qquad\imath=1,2,...,N-1.
	\end{equation*}
	We set
	\begin{equation*}
	\gamma(t):=\left\{
	\begin{aligned}
	&\gamma_{\lfloor tN\rfloor+1}(\{tN\}),\qquad &&t\in[0,1),
	\\&\gamma_N(1), &&t=1.
	\end{aligned}
	\right.
	\end{equation*}
	We call $\gamma$ is from $\gamma(0)$ to $\gamma(1)$.
	
	The set of all $N$-part paths in $G$ is denoted by $\mathcal{P}^N(G)$. We set
	\begin{equation*}
	\mathcal{P}^{\infty}(G):=\bigsqcup_{\imath\in\mathbb{N}^+} \mathcal{P}^\imath(G).
	\end{equation*}
	We say that elements of $\mathcal{P}^{\infty}(G)$ are finite-part paths in $G$.
\end{defn}

\begin{prop} \label{pr-fpp} 
	Let $(G,\pi)$ be a slice-domain over $\mathbb{H}$, $N\in\mathbb{N}^+$ and $\alpha_\imath$ be a slice preserving path in $G$ for each $\imath\in\{1,2,...,N\}$. If
	\begin{equation*}
	\alpha_\imath(1)=\alpha_{\imath+1}(0),\qquad\imath=1,2,...,N-1,
	\end{equation*}
	then there exist $m\in\mathbb{N}^+$, a finite-part path $\gamma=(\gamma_1,\gamma_2,...,\gamma_m)$ in $G$ and a monotonically increasing bijection $\phi:[0,1]\rightarrow[0,1]$, such that
	\begin{equation}\label{eq-spp}
	(\prod_{\imath=1}^m \gamma_\imath)\circ\phi=\prod_{\imath=1}^N \alpha_\imath.
	\end{equation}
\end{prop}

\begin{proof}
	There exist $m\in\mathbb{N}^+$, and $N_\imath\in\mathbb{N}^+$, $\imath=0,1,2,..,m+1$, such that
	\begin{equation*}
	N_0=0,\quad 1=N_1<N_2<...<N_{m+1}=N,\quad \alpha_n(1)\notin\mathbb{R}\quad\mbox{and}\quad \alpha_{N_\imath}(1)\in\mathbb{R}
	\end{equation*}
	for each $n\in\{1,2,...,N-1\}\backslash\{N_2,N_3,...,N_m\}$ and $\imath=2,3,...m$. Let
	\begin{equation*}
	\gamma_\imath:=\prod_{\jmath=N_\imath}^{N_{\imath+1}}\alpha_\jmath,\qquad \imath=1,2,...,m,
	\end{equation*}
	and
	\begin{equation*}
	\phi(t):=\frac{N_{\lfloor Nt\rfloor}+\{Nt\}(N_{\lceil Nt\rceil}-N_{\lfloor Nt\rfloor})}{N},\qquad\forall t\in[0,1].
	\end{equation*}
	Then (\ref{eq-spp}) holds.
\end{proof}

\begin{defn} 
	Let $N\in\mathbb{N}^+$, $\gamma$ be an $N$-part path in $\mathbb{H}$ and $\mathcal{G}=(G,\pi,x)$ be a slice-domain over $\mathbb{H}$ with distinguished point. We say that $\gamma$ is contained in $\mathcal{G}$ (denoted by $\gamma\prec\mathcal{G}$), if there exists an $N$-part path $\alpha$ in $G$ with
	\begin{equation*}
	\alpha(0)=x\qquad\mbox{and}\qquad\pi(\alpha)=\gamma,
	\end{equation*}
	where $\pi:\mathcal{P}^{\infty} (G)\rightarrow\mathcal{P}^{\infty} (\mathbb H)$ is the map, defined by
	\begin{equation*}
	\pi(\beta):=(\pi(\beta_1),\pi(\beta_2),...,\pi(\beta_m))
	\end{equation*}
	for each $m\in\mathbb{N}^+$ and $\beta\in\mathcal{P}^m(G)$. We remark the map $\pi$ in the right side is the projection from $\mathcal{G}=(G,\pi,x)$.
\end{defn}

\begin{prop} \label{def:notation-path}
	Let $\gamma$ be a finite-part path in $\mathbb{H}$ and $\mathcal{G}=(G,\pi,x)$ be a slice-domain over $\mathbb{H}$ with distinguished point. If $\gamma\prec\mathcal{G}$, there exists a unique finite-part path $\alpha$ in $G$, such that $\pi(\alpha)=\gamma$ and $\alpha(0)=x$. We call $\alpha$  the lifting of $\gamma$ to $\mathcal{G}$, denoted by $\gamma_{\mathcal{G}}$.
\end{prop}

\begin{proof}
	For each $N\in\mathbb{N}^+$ and $N$-part path $\gamma$ in $G$. If $\alpha,\beta$ are finite-part paths in $G$ with
	\begin{equation*}
	\pi(\alpha)=\gamma=\pi(\beta)\qquad\mbox{and}\qquad\alpha(0)=x=\beta(0).
	\end{equation*}
	It follows that $\alpha$ and $\beta$ are $N$-part paths and
	\begin{equation*}
	\pi(\alpha_\imath)=\gamma_\imath=\pi(\beta_\imath),\qquad\imath=1,2,...,N.
	\end{equation*}
	According to Proposition \ref{pr-ul} and recursion, we have
	\begin{equation*}
	\alpha_\imath=\beta_\imath\qquad\mbox{and}\qquad\alpha_{\jmath+1}(0)=\alpha_\jmath(1)=\beta_\jmath(1)=\beta_{\jmath+1}(0)
	\end{equation*}
	for each $\imath\in\{1,2,...,N\}$ and $\jmath\in\{1,2,...,N-1\}$. It is clear that $\alpha=\beta$.
\end{proof}

\begin{thm}\label{th-efpsp} 
	Let $\mathcal{G}=(G,\pi,x)$ be a slice-domain over $\mathbb{H}$ with distinguished point. Then for each $q\in G$, there exists a finite-part path $\gamma$ in $G$ from $x$ to $q$. 
\end{thm}

\begin{proof} 
	According to Remark \ref{rk-pc}, for each $q\in G$, there exists a path $\alpha$ in $G$ from $x$ to $q$. As a result of Proposition \ref{pr-sp}, there exists a schlicht real-connected domain $U_t\subset G$ containing $\alpha(t)$ for each $t\in[0,1]$. Notice that $\alpha$ is continuous, we see that the connected component $W_t$ of $\alpha^{-1}(U_t)$ containing $t$ is open in $[0,1]$. And since $[0,1]$ is compact, it follows that there exist a minimal $m\in\mathbb{N}^+$, and $t_\imath\in[0,1]$, $\imath=1,2,...,m$, such that
	\begin{equation*}
	[0,1]\subset\bigcup_{\imath=1}^m W_{t_\imath}\qquad\mbox{and}\qquad 0<t_0<t_1<...<t_m<1.
	\end{equation*}
	We notice that $\cup_{\imath=1}^m W_{t_\imath}$ covers each point in $[0,1]$ at most twice. Otherwise, there exists a point $q\in W_a\cap W_b\cap W_c$ in $[0,1]$, and $a,b,c\in\{t_1,t_2,...,t_m\}$ are different from each other. Then
	\begin{equation*}
	W:=W_a\cup W_b\cup W_c
	\end{equation*}
	is an open interval in $[0,1]$. Without loss of generality, we suppose
	\begin{equation*}
	\inf(W_a)=\inf(W)\qquad\mbox{and}\qquad\sup(W_b)=\sup(W)\qquad(\mbox{resp.}\ \sup(W_a)=\sup(W)).
	\end{equation*}
	Then $W=W_a\cup W_b$ (resp. $W=W_a$). It follows that we can remove $W_c$, and $m$ is not minimal. Therefore there exist a bijection
	\begin{equation*}
	L:\{1,2,...,m\}\rightarrow\{1,2,...,m\}
	\end{equation*}
	(reordering $\{W_{t_\imath}\}_{\imath=1}^m$ by the order of $\sup(W_{t_\imath})$), and
	\begin{equation*}
	s_\imath\in[0,1],\qquad\imath=1,2,...,m+1
	\end{equation*}
	with
	\begin{equation*}
	0=s_1<s_2<...<s_{m+1}=1\qquad\mbox{and}\qquad[s_\jmath,s_{\jmath+1}]\subset W_{t_{L(\jmath)}},\quad\jmath=1,2,...,m.
	\end{equation*}
	Then
	\begin{equation*}
	\alpha([s_\imath,s_{\imath+1}])\subset U_{t_{L(\imath)}},\qquad\imath=1,2,...,m.
	\end{equation*}
	Thanks to Proposition \ref{pr-pir}, there exist two slice preserving paths $\beta_{2\imath}$ and $\beta_{2\imath+1}$ such that $\beta_{2\imath}\beta_{2\imath+1}$ be a path from $\alpha(s_\imath)$ to $\alpha(s_{\imath+1})$ for each $\imath\in\{1,2,...,m\}$. Then $\Pi_{\imath=1}^{2m+1}\beta_\imath$ is a path from $x$ to $q$. According to Proposition \ref{pr-fpp}, there exists a finite-part path from $x$ to $q$.
\end{proof}

\section{Unions of Riemann slice-domains}\label{sc-urs}

In this section, we construct unions of Riemann slice-domains, following \cite[Pages 91-94]{Fritzsche2002001B}. This provides a basis to study the envelope of slice regularity in Section \ref{sc-esh}.

\begin{defn}
	Let $\Lambda$ be an index set, and $\mathcal{G}$, $\mathcal{G}_\lambda$, $\lambda\in\Lambda$, be slice-domains over $\mathbb{H}$ with distinguished point. $\mathcal{G}$ is called a upper (resp. lower) bound of $\{\mathcal{G}_\lambda\}_{\lambda\in\Lambda}$, if $\mathcal{G}_\lambda\prec\mathcal{G}$ (resp. $\mathcal{G}_\lambda\succ\mathcal{G}$) for each $\lambda\in\Lambda$.
\end{defn}

\begin{defn}\label{de-sdui}
	Let $\Lambda$ be an index set, and $\mathcal{G}$, $\mathcal{G}_\lambda$, $\lambda\in\Lambda$ be slice-domains over $\mathbb{H}$ with distinguished point. $\mathcal{G}$ is called a supremum (resp. infimum) or union (resp. intersection) of $\{\mathcal{G}_\lambda\}_{\lambda\in\Lambda}$, if $\mathcal{G}\prec\mathcal{G}'$ (resp. $\mathcal{G}\succ\mathcal{G}'$) for each upper (resp. lower) bound $\mathcal{G}'$ of $\{\mathcal{G}_\lambda\}_{\lambda\in\Lambda}$.
	
	We denote the set of all unions (resp. intersections) of $\{\mathcal{G}_\lambda\}_{\lambda\in\Lambda}$ by $\cup_{\lambda\in\Lambda}\mathcal{G}_\lambda$ (resp. $\cap_{\lambda\in\Lambda}\mathcal{G}_\lambda$).
\end{defn}

\begin{rmk}
	The definition of the union of Riemann domains in \cite[Page 94]{Fritzsche2002001B} has a contradiction. We set
	\begin{equation*}
	\mathbb{B}_0:=B_{\mathbb{H}}(0,1)\qquad\mbox{and}\qquad\mathcal{G}_0:=(\mathbb{B}_0,id_{\mathbb{B}_0},0).
	\end{equation*}
	Then the cardinality of $[\mathcal{G}_0]$ is infinite. We set
	\begin{equation*}
	\mathcal{L}:=\bigcup_{\mathcal{H}\in[\mathcal{G}_0]}\mathcal{H},
	\end{equation*}
	and notice that $\mathcal{L}\in[\mathcal{G}_0]$. However, the statuses of elements in $[\mathcal{G}_0]$ are equal. It implies that $\mathcal{L}$ would not be any element of $[\mathcal{G}_0]$, which is a contradiction.
\end{rmk}

Let $q\in\mathbb{H}$, $\Lambda$ be an index set and $\mathcal{G}_{\lambda}=(G_{\lambda},\pi_{\lambda},x_{\lambda})$, $\lambda\in\Lambda$, be slice-domains over $\mathbb{H}$ with distinguished point with $\pi_{\lambda}(x_{\lambda})=q$. Now, we construct a union of $\{\mathcal{G}_\lambda\}_{\lambda\in\Lambda}$ with a little revise, following \cite[Page 91-92]{Fritzsche2002001B}. According to the axiom of choice, there exists a subset $\Lambda'$ of $\Lambda$, such that the cardinality of $[\mathcal{G}_\lambda]\cap\{\mathcal{G}_\rho\}_{\rho\in\Lambda'}$ is one for each $\lambda\in\Lambda$. We set
\begin{equation*}
X:=\bigsqcup_{\lambda\in\Lambda'}G_{\lambda},
\end{equation*}
and let the topology of $X$ be the disjoint union topology. An equivalence relation $\sim$ on $X$ is said to have property $(P)$ if the followings hold:

1. $x_\lambda\sim x_\rho$ for each $\lambda,\rho\in\Lambda'$.

2. If $\alpha:[0,1]\rightarrow G_\lambda$ and $\beta:[0,1]\rightarrow G_\rho$ are continuous paths with $\alpha(0)\sim\beta(0)$ and $\pi_{\lambda}\circ\alpha=\pi_{\rho}\circ\beta$, then $\alpha(1)\sim\beta(1)$.

Let $\varphi_{\lambda}:G_{\lambda}\rightarrow X$ be the canonical injection, i.e.,
\begin{equation*}
\varphi_\lambda(x):=x,\qquad\forall\ \lambda\in\Lambda'\ \mbox{and}\ x\in G_\lambda.
\end{equation*}

And let $\pi_X:X\rightarrow\mathbb{H}$ be the map with
\begin{equation*}
\pi_X\circ\varphi_{\lambda}=\pi_{\lambda},\qquad\forall\ \lambda\in\Lambda'.
\end{equation*}

\begin{prop}\label{pr-pp} 
	Let $\sim$ be the equivalence relation on $X$, such that $x\sim y$, if and only if $\pi_X(x)=\pi_X(y)$, for each $x,y\in X$. Then $\sim$ has property $(P)$.
\end{prop}

\begin{proof}
	For each $\alpha,\beta\in\mathcal{P}(X)$ with $\pi_X\circ\alpha=\pi_X\circ\beta$, we have $\pi_X\circ\alpha(1)=\pi_X\circ\beta(1)$, then $\alpha(1)\sim\beta(1)$.
\end{proof}

We denote by $E$ the family of all equivalence relations on $X$ with property $(P)$. Thanks to Proposition \ref{pr-pp}, $E$ is not empty.

Let $\sim_P$ be the equivalence relations, such that $x\sim_P y$, if and only if $x\sim y$ for each $\sim$ in $E$, where $x,y\in X$.
\begin{prop}\label{pr-sfe}
$\sim_P$ is finer than each equivalence relation on $X$ with property $(P)$.
\end{prop}
\begin{proof}
	This proposition is proved directly by the definition of $\sim_P$.
\end{proof}
Let $\phi:X\rightarrow X/\sim_P$ be the projection of $\sim_P$, defined by
\begin{equation*}
\phi(x)=[x]_P,\qquad\forall x\in X,
\end{equation*}
where $[x]_P$ is the equivalence class of $x$ in $X$ with respect to $\sim_P$. It is clear that $\sim_P$ has the property $(P)$.
\begin{prop}\label{pr-pp1}
	For each $x,y\in X$ with $x\sim_P y$, we have $\pi_X(x)=\pi_X(y)$.
\end{prop}

\begin{proof}
	We notice that $\sim_P$ is finer than each equivalence relation on $X$ with property (P). And since Proposition \ref{pr-pp}, then $\pi_X(x)=\pi_X(y)$.
\end{proof}

Thence there exists a map $\widehat{\pi}:X/\sim_P\rightarrow\mathbb{H}$ such that
\begin{equation}\label{eq-ps}
\widehat{\pi}\circ\phi=\pi_X.
\end{equation}
We set
\begin{equation*}
\widehat{G}:=X/\sim_P,
\end{equation*}
and let the topology $\tau(\widehat{G})$ of $\widehat{G}$ be the quotient topology. We set
\begin{equation*}
\widehat{x}:=\varphi_{\lambda}(x_{\lambda}),\qquad\forall\lambda\in\Lambda'.
\end{equation*}

\begin{thm}\label{th-union}
	$\widehat{\mathcal{G}}:=(\widehat{G},\widehat{\pi},\widehat{x})$ is a Riemann slice-domain over $\mathbb{H}$ with distinguished point.
	
	Moreover, for each $\lambda\in\Lambda'$, we have $\mathcal{G}_{\lambda}\prec\widehat{\mathcal{G}}$, and $\phi\circ\varphi_{\lambda}$ is the fiber preserving mapping from $G_{\lambda}$ to $\widehat{G}$ with $\phi\circ\varphi_{\lambda}(x_\lambda)=\widehat{x}$.
	
	We call $\widehat{\mathcal{G}}$   the pre-union of $\{\mathcal{G}_\lambda\}_{\lambda\in\Lambda'}$.
\end{thm}

\begin{proof}
	We prove this theorem exactly as \cite[Theorem 8.7]{Fritzsche2002001B}.
	
	(1) For each $p\in\widehat{G}$, there exist $q\in\phi^{-1}(p)$ and $\lambda\in\Lambda'$ with $q\in G_{\lambda}$. For each $\lambda\in\Lambda$, since $G_{\lambda}$ is connected, there exists a path $\gamma$ in $G_{\lambda}$ from $x_{\lambda}$ to $q$. Notice that $\tau(\widehat{G})$ is the quotient topology, and since $\phi$ is continuous, it follows that $\phi\circ\gamma$ is a path in $\widehat{G}$ from $\widehat{x}$ to $q$. Then $\widehat{G}$ is path-connected. It is clear that $\widehat{G}$ is connected.
	
	(2) For each $p\in\widehat{G}$, $q\in\phi^{-1}(p)$ and $\lambda\in\Lambda'$ with $q\in G_{\lambda}$, let $U$ be a domain in $G_{\lambda}$ containing $q$, such that $\pi_{\lambda}|_U:U\rightarrow\pi_X(U)$ is a slice-homeomorphism. We notice that Proposition \ref{pr-pp1} and $\pi_{\lambda}|_U$ is a slice-homeomorphic, then we have
	\begin{equation*}
	(\phi\circ\varphi_{\lambda})|_U:U\rightarrow\phi\circ\varphi_{\lambda}(U)
	\end{equation*}
	is injective. Since $\tau(\widehat{G})$ is the quotient topology, it follows that $\phi\circ\varphi|_U$ is continuous. Obviously $\phi\circ\varphi|_U$ is surjective, it is clear that $\phi\circ\varphi|_U$ is a continuous bijection. We will prove that $(\phi\circ\varphi_{\lambda})|_U^{-1}$ is also continuous.
	
	For each domain $V$ contained in $U$, we set
	\begin{equation*}
	V':=\phi\circ\varphi_{\lambda}(V).
	\end{equation*}
	For each $\rho\in\Lambda'$ and $y\in\phi^{-1}(V')\cap G_{\rho}$, there exists a domain $W$ in $G_{\rho}$ containing $y$, such that $\pi_{\rho}|_W:W\rightarrow\pi_X(W)$ is a slice-homeomorphism with $\pi_X(W)\subset\pi_X(V)$. We denote by $y'$ the unique element in $\varphi_{\lambda}^{-1}(\phi|_U^{-1}(\phi(y)))=\phi|_U^{-1}(\phi(y))$. Since
	\begin{equation*}
	\phi(y')=\phi(\phi|_U^{-1}(\phi(y)))=\phi\circ\phi|^{-1}_U(\phi(y))=\phi(y),
	\end{equation*}
	we have $y\sim_P y'$. Since $\pi_X(W)$ is a slice-domain in $\mathbb{H}$, and due to Proposition \ref{pr-lp}, it follows that for each $z\in\pi_X(W)$, there exists a slice-path $\gamma$ in $\pi_X(W)$ from $\pi_X(y)=\pi_X(y')$ to $z$. According to the property $(P)$ of $\widehat{G}$, and since $\pi_\rho,\pi_\lambda$ are local slice-homeomorphic and $\pi_X(W)\subset\pi_X(V)$, we have
	\begin{equation*}
	\pi_{\rho}|_W^{-1}(z)\sim_P\pi_{\lambda}|_U^{-1}(z).
	\end{equation*}
	Thus, $W\subset\phi^{-1}(V')\cap G_{\rho}$. Note that $W$ is open in $X$, then $y$ is an interior point of $\phi^{-1}(V')$. Since $y\in\phi^{-1}(V')\cap G_{\rho}$ is chosen arbitrarily, then each point in $\phi^{-1}(V')\cap G_{\rho}$ is an interior point. Thus $\phi^{-1}(V')\cap G_{\rho}$ is open in $X$ for each $\rho\in\Lambda'$. It follows that $\phi^{-1}(V')$ is open in $X$. Therefore $V'$ is open in $\widehat{G}$. It is clear that
	\begin{equation*}
	((\phi\circ\varphi_{\lambda})|_U^{-1})^{-1}(V)=V'
	\end{equation*}
	is open in $\widehat{G}$ for each $V\in\tau(U)$. We conclude that $(\phi\circ\varphi_{\lambda})|_U^{-1}$ is continuous. Consequently, $(\phi\circ\varphi_{\lambda})|_U$ is a homeomorphism, and then
	\begin{equation*}
	\widehat{\pi}|_{\phi\circ\varphi_{\lambda}(U)}=\pi_X|_U\circ(\phi\circ\varphi_{\lambda})|_U^{-1}:\phi\circ\varphi_{\lambda}(U)\rightarrow\pi_X(U)
	\end{equation*}
	is a slice-homeomorphism. Thence $\widehat{\pi}$ is a local slice-homeomorphism.
	
	(3) For each $p,q\in\widehat{G}$, if $\widehat{\pi}(p)\neq\widehat{\pi}(q)$, there exist a slice-domain $U$ in $\mathbb{H}$ containing $\widehat{\pi}(p)$ and a slice-domain $V$ in $\mathbb{H}$ containing $\widehat{\pi}(q)$ such that $U\cap V=\varnothing$. Since $\widehat{\pi}$ is a local slice-homeomorphism, it follows that $\widehat{\pi}^{-1}(V)$ and $\widehat{\pi}^{-1}(U)$ are two disjoint open sets in $\widehat{G}$ containing $p$ and $q$ respectively.
	
	Otherwise, $\widehat{\pi}(p)=\widehat{\pi}(q)$, there exist a domain $U$ in $\widehat{G}$ containing $p$, a domain $V$ in $\widehat{G}$ containing $q$ and a slice-domain $W$ in $\mathbb{H}$, such that $\widehat{\pi}|_U:U\rightarrow W$ and $\widehat{\pi}|_V:V\rightarrow W$ are slice-homeomorphisms. If there exists $y\in U\cap V$, then $\widehat{\pi}|_U^{-1}$ and $\widehat{\pi}|_V^{-1}$ are two continuous mappings from $W$ to $\widehat{G}$. Notice that
	\begin{equation*}
	\widehat{\pi}|_U^{-1}(\widehat{\pi}(y))=y=\widehat{\pi}|_V^{-1}(\widehat{\pi}(y)),\qquad
	\widehat{\pi}\circ\widehat{\pi}|_U^{-1}=id_W=\widehat{\pi}\circ\widehat{\pi}|_V^{-1},
	\end{equation*}
	and Proposition \ref{pr-ul}, it follows that $\widehat{\pi}|_U^{-1}=\widehat{\pi}|_V^{-1}$. Consequently, $p=q$, which is a contradiction. It implies that $U\cap V=\varnothing$. Then $\widehat{G}$ is Hausdorff, and then we see from (1) and (2) that $\widehat{\mathcal{G}}$ is a slice-domain over $\mathbb{H}$ with distinguished point.
	
	(4) For each $\lambda\in\Lambda'$, since $\widehat{\pi}\circ\phi=\pi_X$ (see (\ref{eq-ps})), it follows that
	\begin{equation*}
	\pi_{\lambda}=\pi_X\circ\varphi_{\lambda}=\widehat{\pi}\circ\phi\circ\varphi_{\lambda}.
	\end{equation*}
	It is clear that $G_{\lambda}\prec\widehat{G}$ and $\phi\circ\varphi_\lambda$ is the fiber preserving mapping from $\mathcal{G}_\lambda$ to $\mathcal{G}$.
\end{proof}

\begin{thm}\label{th-uos} 
	$[\widehat{\mathcal{G}}]$ is the set of all unions of $\{\mathcal{G}_\lambda\}_{\lambda\in\Lambda}$, i.e.,
	\begin{equation*}
	[\widehat{\mathcal{G}}]=\bigcup_{\lambda\in\Lambda}\mathcal{G}_\lambda.
	\end{equation*}
\end{thm}

\begin{proof}
	According to Theorem \ref{th-union}, $G_{\lambda}\prec\widehat{G}$ for each $\lambda\in\Lambda'$, it follows that $\widehat{\mathcal{G}}$ is an upper bound of $\{\mathcal{G}_\lambda\}_{\lambda\in\Lambda'}$. For each $\lambda\in\Lambda$, there exist $\rho\in\Lambda'$ with $\mathcal{G}_\lambda\cong\mathcal{G}_\rho$. And according to Proposition \ref{pr-weak}, it is clear that $\mathcal{G}_\lambda\prec\mathcal{G}$ for each $\lambda\in\Lambda$. Then $\widehat{\mathcal{G}}$ is also an upper bound of $\{\mathcal{G}_\lambda\}_{\lambda\in\Lambda'}$.
	
	For each upper bound $\mathcal{G}'=(G',\pi',x')$ of $\{\mathcal{G}_\lambda\}_{\lambda\in\Lambda'}$, there exists a fiber preserving mapping $\varphi_\lambda':G_\lambda\rightarrow G'$ for each $\lambda\in\Lambda'$. Let $\cong$ be the equivalence relation on $X$, such that $x\cong y$, if and only if, there exist $\lambda,\rho\in\Lambda'$ with $x\in G_\lambda$, $y\in G_\rho$ and $\varphi_\lambda'(x)=\varphi_\rho'(y)$, for each $x,y\in X$. In view of Proposition \ref{pr-ul}, $\cong$ has the property $(P)$. Let $\varphi':X\rightarrow G'$ be the map defined by
	\begin{equation*}
	\varphi'|_{G_\lambda}:=\varphi_\lambda',\qquad\forall\lambda\in\Lambda.
	\end{equation*}
	According to Proposition \ref{pr-sfe}, $\sim_P$ is finer than $\cong$. We notice that, for each $q\in\widehat{G}$ and $x,y\in\phi^{-1}(q)$, $x\sim_P y$, then $x\cong y$ and $\varphi'(x)=\varphi'(y)$. It follows that there exists a continuous map $\varphi'':\widehat{G}\rightarrow G'$ defined by
	\begin{equation*}
	\varphi''(x):=\varphi'(\phi^{-1}(x)).
	\end{equation*}
	It is clear that $\varphi''$ is the fiber preserving mapping from $\widehat{\mathcal{G}}$ to $\mathcal{G}'$. Therefore $\widehat{\mathcal{G}}\prec\mathcal{G}'$. Then $\widehat{\mathcal{G}}$ is a union of $\{\mathcal{G}_\lambda\}_{\lambda\in\Lambda}$.
	
	If $\mathcal{G}'$ is another union of $\{\mathcal{G}_\lambda\}_{\lambda\in\Lambda}$, then we see from Definition \ref{de-sdui} that $\widehat{\mathcal{G}}\prec\mathcal{G}'$ and $\mathcal{G}'\prec\widehat{\mathcal{G}}$. Thence $\mathcal{G}'\in[\widehat{\mathcal{G}}]$. On the other hand, for each $\mathcal{G}''\in[\widehat{\mathcal{G}}]$, we have
	\begin{equation*}
	\mathcal{G}_\lambda\prec\widehat{\mathcal{G}}\prec\mathcal{G}'',\qquad\forall\lambda\in\Lambda
	\end{equation*}
	and
	\begin{equation*}
	\mathcal{G}''\prec\widehat{\mathcal{G}}\prec\mathcal{G}'''
	\end{equation*}
	for each upper bound $\mathcal{G}'''$ of $\{\mathcal{G}_\lambda\}_{\lambda\in\Lambda}$. Then $\mathcal{G}''$ is also a union of $\{\mathcal{G}_\lambda\}_{\lambda\in\Lambda}$.
\end{proof}

\begin{defn} 
	Let $\mathcal{G}=(G,\pi,x)$ be a slice-domain over $\mathbb{H}$ with distinguished point and $\gamma$ be a path in $\mathbb{H}$. We say that $\gamma$ is contained in $\mathcal{G}$ (denoted by $\gamma\prec\mathcal{G}$), if there exists a path $\alpha$ in $G$ such that $\alpha(0)=x$ and $\pi(\alpha)=\gamma$.
\end{defn}

\begin{prop}\label{pr-1gph}
	Let $\gamma$ be a path in $\mathbb{H}$ and $\mathcal{G}=(G,\pi,x)$ be a slice-domain over $\mathbb{H}$ with distinguished point. If $\gamma\prec\mathcal{G}$, there exists a unique path $\alpha$ in $G$ such that $\pi(\alpha)=\gamma$. We call $\alpha$ is the lifting of $\gamma$ to $\mathcal{G}$, denoted by $\gamma_\mathcal{G}$.
\end{prop}

\begin{proof}
	This follows immediately from Proposition \ref{pr-ul}.
\end{proof}

Let $\gamma$ be a path in $\mathbb{H}$. We denote by $\mathscr{R}_\gamma$ all the slice-domain over $\mathbb{H}$ with distinguished point containing $\gamma$.

\begin{prop} \label{pr-psdp} 
Let $\mathcal{G}$ is a slice-domain over $\mathbb{H}$ with distinguished point, and $\gamma$ be a path (resp. finite-part path) in $\mathbb{H}$. If $\mathcal{G}\in\mathscr{R}_\gamma$, then $\mathcal{G}'\in\mathscr{R}_\gamma$ for each $\mathcal{G}'\in\mathscr{R}$ with $\mathcal{G}'\succ\mathcal{G}$. In particular, $[\mathcal{G}]\subset\mathscr{R}_\gamma$.
\end{prop}

\begin{proof}
We write $\mathcal{G}=(G,\pi,x)$ and $\mathcal{G}'=(G',\pi',x')$. Let $\varphi:G\rightarrow G'$ be the fiber preserving map from $\mathcal{G}$ to $\mathcal{G}'$, and $\alpha$ be the lifting of $\gamma$ to $\mathcal{G}$. Then $\varphi\circ\alpha$ is the lifting of $\gamma$ to $\mathcal{G}'$.
\end{proof}

For each path $\gamma$ in $\mathbb{H}$, we set a map
\begin{equation*}
\mathcal{T}_\gamma:\mathscr{R}_\gamma\rightarrow\mathscr{R},\qquad (G,\pi,x)\mapsto (G,\pi,\gamma_{\mathcal{G}}(1)),\quad\forall\ (G,\pi,x)\in\mathscr{R}_\gamma.
\end{equation*}
Let $\gamma^{(-1)}$ be the inverse path of $\gamma$, i.e.,
\begin{equation*}
\gamma^{(-1)}(t)=\gamma(1-t),\qquad\forall\ t\in[0,1].
\end{equation*}
According to Proposition \ref{pr-ul}, we have
\begin{equation*}
\mathcal{T}_{\gamma^{(-1)}}\circ \mathcal{T}_\gamma=id_{\mathscr{R}_\gamma},
\end{equation*}
where $id_{\mathscr{R}_\gamma}$ is the identity map on $\mathscr{R}_\gamma$. Then
\begin{equation*}
\mathcal{T}_\gamma^{-1}=\mathcal{T}_{\gamma^{(-1)}}:\mathscr{R}_{\gamma^{(-1)}}\rightarrow\mathscr{R}.
\end{equation*}

\begin{prop} \label{pr-sdt} 
	Let $\gamma$ be a path in $\mathbb{H}$, $\Lambda$ be an index set, and $\mathcal{G},\mathcal{G}_\lambda\in\mathscr{R}_\gamma$ for each $\lambda\in\Lambda$. The following statements hold:
	
	(a). If $\lambda,\rho\in\Lambda$ and $\mathcal{G}_\lambda\prec\mathcal{G}_\rho$, then $\mathcal{T}_\gamma(\mathcal{G}_\lambda)\prec\mathcal{T}_\gamma(\mathcal{G}_\rho)$.
	
	(b). If $\mathcal{G}$ is an upper (resp. lower) bound of $\{\mathcal{G}_\lambda\}_{\lambda\in\Lambda}$, then $\mathcal{T}_\gamma(\mathcal{G})$ is an upper (resp. lower) bound of $\mathcal{T}_\gamma(\{\mathcal{G}_\lambda\}_{\lambda\in\Lambda})$, where
	\begin{equation*}
	\mathcal{T}_\gamma(\{\mathcal{G}_\lambda\}_{\lambda\in\Lambda}):=\{\mathcal{T}_\gamma(\mathcal{L}):\mathcal{L}\in \{\mathcal{G}_\lambda\}_{\lambda\in\Lambda}\}.
	\end{equation*}
	
	(c). If $\mathcal{G}$ is a supremum (resp. infimum) of $\{\mathcal{G}_\lambda\}_{\lambda\in\Lambda}$, then $\mathcal{T}_\gamma(\mathcal{G})$ is a supremum (resp. infimum) of $\mathcal{T}_\gamma(\{\mathcal{G}_\lambda\}_{\lambda\in\Lambda})$.
	
	(d). $\mathcal{T}_\gamma$ commutes with the union of slice-domains over $\mathbb{H}$ with distinguished point, i.e., $\cup_{\lambda\in\Lambda}\mathcal{T}_\gamma(\mathcal{G}_\lambda)=\mathcal{T}_\gamma(\cup_{\lambda\in\Lambda}\mathcal{G}_\lambda)$.
\end{prop}

\begin{proof}
	(a). Let $\varphi$ be the fiber preserving map from $\mathcal{G}_\lambda$ to $\mathcal{G}_\rho$. Notice that
	\begin{equation*}
	\pi_\rho\circ\varphi\circ\gamma_{\mathcal{G}_\lambda}=\pi_\lambda\circ\gamma_{\mathcal{G}_\lambda}=\gamma=\pi_\rho\circ\gamma_{\mathcal{G}_\rho},
	\end{equation*}
	and $\varphi(\gamma_{\mathcal{G}_\lambda}(0))=\gamma_{\mathcal{G}_\rho}(0)$, it follows from Proposition \ref{pr-ul} that
	\begin{equation*}
	\varphi\circ\gamma_{\mathcal{G}_\lambda}=\gamma_{\mathcal{G}_\rho}\qquad\mbox{and}\qquad\varphi(\gamma_{\mathcal{G}_\lambda}(1))=\gamma_{\mathcal{G}_\rho}(1).
	\end{equation*}
	Therefore $\varphi$ is also the fiber preserving map from $\mathcal{G}_\lambda$ to $\mathcal{G}_\rho$.
	
	(b). According to (a), $\mathcal{T}_\gamma(\mathcal{G})\succ\mathcal{T}_\gamma(\mathcal{G}_\lambda)$ (resp. $\mathcal{T}_\gamma(\mathcal{G})\prec\mathcal{T}_\gamma(\mathcal{G}_\lambda)$) for each $\lambda\in\Lambda$.
	
	(c). According to (b), $\mathcal{T}_\gamma(\mathcal{G})$ is an upper (resp. lower) bound of $\mathcal{T}_\gamma(\{\mathcal{G}_\lambda\}_{\lambda\in\Lambda})$. For each upper (resp. lower) bound $\mathcal{G}'$ of $\mathcal{T}_\gamma(\{\mathcal{G}_\lambda\}_{\lambda\in\Lambda})$, $\mathcal{T}_{\gamma^{(-1)}}(\mathcal{G}')$ is a upper (resp. lower) bound of $\{\mathcal{G}_\lambda\}_{\lambda\in\Lambda}$, by (b). Then
	\begin{equation*}
	\mathcal{G}\prec\mathcal{T}_{\gamma^{(-1)}}(\mathcal{G}')\qquad(\mbox{resp.}\ \mathcal{G}\succ\mathcal{T}_{\gamma^{(-1)}}(\mathcal{G}'))
	\end{equation*}
	for each upper (resp. lower) bound $\mathcal{G}'$ of $\mathcal{T}_\gamma(\{\mathcal{G}_\lambda\}_{\lambda\in\Lambda})$. And thanks to (a), we have
	\begin{equation*}
	\mathcal{T}_\gamma(\mathcal{G})\prec\mathcal{G}'\qquad(\mbox{resp.}\ \mathcal{T}_\gamma(\mathcal{G})\succ\mathcal{G}')
	\end{equation*}
	for each upper (resp. lower) bound $\mathcal{G}'$ of $\mathcal{T}_\gamma(\{\mathcal{G}_\lambda\}_{\lambda\in\Lambda})$. It follows that $\mathcal{T}_\gamma(\mathcal{G})$ is a supremum (resp. infimum) of $\mathcal{T}_\gamma(\{\mathcal{G}_\lambda\}_{\lambda\in\Lambda})$.
	
	(d). For each $\mathcal{G}\in\cup_{\lambda\in\Lambda}\mathcal{T}_\gamma(\mathcal{G}_\lambda)$, according to (c), $\mathcal{T}_{\gamma^{(-1)}}(\mathcal{G})$ is a supremum of
	\begin{equation*}
	\mathcal{T}_{\gamma^{(-1)}}(\{\mathcal{T}_\gamma(\mathcal{G}_\lambda)\}_{\lambda\in\Lambda})=\{\mathcal{G}_\lambda\}_{\lambda\in\Lambda}.
	\end{equation*}
	It follows that $\mathcal{G}\in\mathcal{T}_\gamma(\cup_{\lambda\in\Lambda}\mathcal{G}_\lambda)$.
	
	Conversely, for each $\mathcal{G}\in\mathcal{T}_\gamma(\cup_{\lambda\in\Lambda}\mathcal{G}_\lambda)$, we have $\mathcal{T}_{\gamma^{(-1)}}(\mathcal{G})$ is a supremum of $\{\mathcal{G}_\lambda\}_{\lambda\in\Lambda}$. And by (c), $\mathcal{G}\in\cup_{\lambda\in\Lambda}\mathcal{T}_\gamma(\mathcal{G}_\lambda)$.
\end{proof}

\begin{prop}\label{pr-tsp}
	Let $\alpha$, $\beta$ be paths in $\mathbb{H}$, and $\mathcal{G}\in\mathscr{R}_\alpha$. If $\mathcal{T}_\alpha(\mathcal{G})\in\mathscr{R}_\beta$, then $\mathcal{G}\in\mathscr{R}_{\alpha\beta}$.
\end{prop}

\begin{proof}
	We write $\mathcal{G}=(G,\pi,x)$. Note that $\alpha_\mathcal{G}\beta_{\mathcal{T}_\alpha(\mathcal{G})}$ is a path in $G$ with
	\begin{equation*}
	\alpha_\mathcal{G}\beta_{\mathcal{T}_\alpha(\mathcal{G})}(0)=x\qquad\mbox{and}\qquad\pi(\alpha_\mathcal{G}\beta_{\mathcal{T}_\alpha(\mathcal{G})})=\alpha\beta.
	\end{equation*}
	It is clear that $\mathcal{G}\in\mathscr{R}_{\alpha\beta}$.
\end{proof}

\section{Envelopes of slice regularity}\label{sc-esh}

In this section, we define envelopes of slice regularity following the complex case (see \cite[Pages 96-100]{Fritzsche2002001B}). We prove an identity principle of slice regular functions on Riemann slice-domains (see Theorem \ref{th-idh}).

\begin{defn}  \label{def:slicefunction}
	Let $(G,\pi)$ be a slice-domain over $\mathbb{H}$. A function $f:G\rightarrow\mathbb{H}$ is called slice regular at a point $x\in G$ if there exists an open neighborhood $U\subset G$ of $x$ and a slice-open set $V\subset\mathbb{H}$ such that $\pi|_U:U\rightarrow V$ is a slice-homeomorphism and $f\circ\pi|_U^{-1}:V\rightarrow\mathbb{H}$ is slice regular.
	
	The function $f$ is called slice regular on $G$ if $f$ is slice regular at every point $x\in G$. We denote the set of all slice regular functions on $G$ by $\mathcal{SR}(G)$.
\end{defn}

\begin{defn}\label{def:Riemanndom}
	For any  $I\in\mathbb{S}$, a (Riemann) domain over $\mathbb{C}_I$ is a pair $(G,\pi)$ with the following properties:
	
	1. $G$ is a connected Hausdorff space.
	
	2. $\pi:G\rightarrow\mathbb{C}_I$ is a local homeomorphism.
	
	For each $x\in G$, $(G,\pi,x)$ is called a (Riemann) domain over $\mathbb{C}_I$ with distinguished point.
\end{defn}

\begin{defn}
	Let $I\in\mathbb{S}$ and $(G,\pi)$ be a  (Riemann) domain over $\mathbb{C}_I$. A function $f:G\rightarrow\mathbb{H}$ is called holomorphic at a point $x\in G$ if there are open neighborhoods
	\begin{equation*}
	U=U(x)\subset G\qquad\mbox{and}\qquad V=V(\pi(x))\subset\mathbb{C}_I,
	\end{equation*}
	such that $\pi|_U:U\rightarrow V$ is a homeomorphism and $f\circ\pi|_U^{-1}:V\rightarrow\mathbb{H}$ is holomorphic.
	
	The function $f$ is called holomorphic on $G$ if $f$ is holomorphic at every point $x\in G$.
\end{defn}

\begin{defn}
	Let $I\in\mathbb{S}$, $\mathcal{G}=(G,\pi,x)$ be a domain over $\mathbb{C}_I$ with distinguished point and $\gamma$ is a path in $\mathbb{C}_I$. We say that $\gamma$ is contained in $\mathcal{G}$ (denoted by $\gamma\prec\mathcal{G}$), if there exists a path $\alpha$ in $G$ with
	\begin{equation*}
	\alpha(0)=x\qquad\mbox{and}\qquad\pi(\alpha)=\gamma.
	\end{equation*}
\end{defn}

\begin{prop}\label{pt-ipsd}
	Let $I\in\mathbb{S}$, $\gamma$ be a path in $\mathbb{C}_I$, and $\mathcal{G}=(G,\pi,x)$ be a domain over $\mathbb{C}_I$ with distinguished point. If $\gamma\prec\mathcal{G}$, there exists a unique path $\alpha$ in $G$ with
	\begin{equation*}
	\alpha(0)=x\qquad\mbox{and}\qquad\pi(\alpha)=\gamma.
	\end{equation*}
	We say that $\alpha$ the lifting of $\gamma$ to $\mathcal{G}$, denoted by $\gamma_\mathcal{G}$.
\end{prop}

\begin{proof}
	This follows immediately from Proposition \ref{pr-ul}.
\end{proof}

Let $(G,\pi,x)$ be a slice-domain over $\mathbb{H}$ with distinguished point. We notice that $(G_I',\pi|_{G_I'})$ is a (Riemann) domain over $\mathbb{C}_I$, for each slice-connected component $G_I'$ of $G_I$ and $I\in\mathbb{S}$. We call a function $f$ defined on $G_I$ is holomorphic if $f|_{G_I'}$ is holomorphic on each slice-connected component $G_I'$ of $G_I$. We can get the following lemma with the similar proof as in  Lemma \ref{lm-sp}.

\begin{lem}\label{le-slsd}
	(Splitting)
	Let $\mathcal{G}=(G,\pi,x)$ be a slice-domain over $\mathbb{H}$ with distinguished point, and $f$ is a function defined on $G$. The function $f$ is slice regular, if and only if for all $I,J\in\mathbb{S}$ with $I\perp J$, there exist two complex-valued holomorphic functions $F_1,F_2:G_I\rightarrow\mathbb{C}_I$, such that
	$$f_I(z)=F_1(z)+F_2(z)J,\qquad \forall z\in G_I,$$
	where $f_I:=f|_{G_I}$.
\end{lem}

\begin{thm} \label{th-idh}
	(Identity Principle) Let $(G,\pi)$ be a slice domain over $\mathbb{H}$, and $f,g$ be regular functions on $G$. If there exists $I\in\mathbb{S}$ such that $f$ and $g$ coincide on a subset of $G_I$ with an accumulation point $q_0$ in $G_I$. Then $f=g$ in $G$.
\end{thm}

\begin{proof}
	We set
	\begin{equation*}
	A:=\{x\in G:\exists\ V\in\tau(G),\ \mbox{s.t.}\ x\in V,\ \mbox{and}\ f=g\ \mbox{on}\ V\}.
	\end{equation*}
	For each $q_0\in A$, according to Proposition \ref{pr-sp}, there exists a schlicht real-connected domain $U$ in $G$ containing $q_0$. Thence $\pi|_U:U\rightarrow\pi(U)$ is a slice-homeomorphism, and $\pi(U)$ is a real-connected slice-domain in $\mathbb{H}$. According to the Identity Principle \ref{th-dps},
	\begin{equation*}
	f\circ\pi|_U^{-1}=g\circ\pi|_U^{-1}\qquad\mbox{on}\qquad\pi(U).
	\end{equation*}
	Then $f=g$ on $U$, and $U\subset A$. Obviously, $A$ is a nonempty open subset of $G$.
	
	If $x\in G\backslash A$, let $W$ be a schlicht real-connected domain containing $x$ in $G$. Then $\pi|_W:W\rightarrow \pi(W)$ is a slice-homeomorphism, and $\pi(W)$ is a real-connected slice-domain in $\mathbb{H}$. If $A\cap W\neq\varnothing$, let $y\in A\cap W$. We notice that $A\cap W$ is a open set in $G$, then there is an open neighbourhood $V_y$ of y, such that $f=g$ on $V_y$. Consequently,
	\begin{equation*}
	f\circ\pi|_W^{-1}=g\circ\pi|_W^{-1}\qquad\mbox{on}\qquad\pi(V_y\cap W).
	\end{equation*}
	According to the Identity Principle \ref{th-dps},
	\begin{equation*}
	f\circ\pi|_W^{-1}=g\circ\pi|_W^{-1}\qquad\mbox{on}\qquad\pi(W).
	\end{equation*}
	It is clear that $x\in A$, which is a contradiction. So $W\cap A=\varnothing$. Thus $G\backslash A$ is open in $G$. Therefore $A$ is closed in $G$. And since $A$ is open in $G$ and $G$ is connected, it follows that $A=G$.
\end{proof}

\begin{defn}
	Let $\mathcal{G_\lambda}=(G_\lambda,\pi_\lambda,x_\lambda)$, $\lambda=1,2$, be slice-domains over $\mathbb{H}$ with distinguished point, and $\mathcal{G}_1\prec\mathcal{G}_2$ by the fiber preserving mapping $\varphi:G_1\rightarrow G_2$ with $\varphi(x_1)=x_2$. For every function $f$ on $G_2$, we define
	\begin{equation*}
	f|_{G_1}:=f\circ\varphi.
	\end{equation*}
\end{defn}

\begin{prop}
	Let $\mathcal{G_\lambda}=(G_\lambda,\pi_\lambda,x_\lambda)$, $\lambda=1,2$, be slice-domains over $\mathbb{H}$ with distinguished point. If $f:G_2\rightarrow\mathbb{H}$ is slice regular and $\mathcal{G}_1\prec\mathcal{G}_2$, then $f|_{G_1}$ is slice regular on $G_1$.
\end{prop}

\begin{proof}
	Trivial, since the fiber preserving mapping $\varphi$ from $\mathcal{G}_1$ to $\mathcal{G}_2$ is a local homeomorphism with $\pi_2\circ\varphi=\pi_1$.
\end{proof}

\begin{defn}
	1. Let $(G,\pi)$ be a slice-domain over $\mathbb{H}$, and $x\in G$ be a point. If $f$ is a slice regular function   near $x$, then the pair $(f,x)$ is called a local slice regular function at $x$.
	
	2. Let $(G_1,\pi_1)$, $(G_2,\pi_2)$ be slice-domains over $\mathbb{H}$, and $x_\imath\in G_\imath$, $\imath=1,2$ with $\pi_1(x_1)=\pi_2(x_2)$. Two locally holomorphic functions $(f_1,x_1)$, $(f_2,x_2)$ are called equivalent if there exist an open neighborhood $U$ of $x_1$, an open neighborhood $V$ of $x_2$ and a slice domain $W$ in $\mathbb{H}$, such that
	\begin{equation*}
	\pi_1|_U:U\rightarrow W\qquad\mbox{and}\qquad\pi_2|_V:V\rightarrow W
	\end{equation*}
	are slice-homeomorphisms, and
	\begin{equation*}
	f_1\circ\pi_1|_U^{-1}=f_2\circ\pi_2|_V^{-1}.
	\end{equation*}
	
	3. The equivalence class of a local slice regular function $(f,x)$ is denoted by $f_x$.
\end{defn}

\begin{prop}
	Let $\mathcal{G_\lambda}=(G_\lambda,\pi_\lambda,x_\lambda)$, $\lambda=1,2$, be slice-domains over $\mathbb{H}$ with distinguished point, and $\mathcal{G}_1\prec\mathcal{G}_2$. Then for every slice regular function $f$ on $G_1$, there exists at most one slice regular function $F$ on $G_2$ with $F|_{G_1}=f$.
\end{prop}

\begin{proof}
	This follows immediately from the Identity Principle \ref{th-idh}.
\end{proof}

\begin{defn}
	Let $\mathcal{G}=(G,\pi,x)$ be a slice-domain over $\mathbb{H}$ with distinguished point, and $\mathscr{F}$ be a nonempty set of slice regular functions on $G$. We call that  $\mathscr{F}$ can be  extended slice regularly to a slice-domain $\breve{\mathcal{G}}=(\breve{G},\breve{\pi},\breve{x})$ over $\mathbb{H}$ with distinguished point, if $\mathcal{G}\prec \breve{\mathcal{G}}$ and for each $f\in\mathscr{F}$, there exists a slice regular function $F$ on $\breve G$ with $F|_{\breve G}=f$.
	
	We say that $\breve{\mathcal G}$ is $\mathscr{F}$-extendible. Let $\mathscr{G}$ be the system of all $\mathscr{F}$-extendible slice-domains over $\mathbb{H}$ with distinguished point. $\mathscr{G}$ is called the $\mathscr{F}$-extendible set.
	
	A slice-domain $\mathcal{G}'$ over $\mathbb{H}$ with distinguished point is called the $\mathscr{F}$-hull of $\mathcal{G}$, if $$\mathcal{G}'\in\bigcup_{\mathcal{\breve G}\in\mathscr{G}}\mathcal{\breve G}.$$  And
	\begin{equation*}
	H_{\mathscr{F}}(\mathcal{G}):=\bigcup_{\mathcal{\breve G}\in\mathscr{G}}\mathcal{\breve G}
	\end{equation*}
	is called the set of $\mathscr{F}$-hulls of $\mathcal{G}$.
	
	
	Let $\mathcal{O}(G)$ be the set of all slice regular functions on $G$. Then
	\begin{equation*}
	H(\mathcal{G}):=H_{\mathscr{O}(G)}(\mathcal{G})
	\end{equation*}
	is called the set of envelopes of slice regularity of $\mathcal{G}$. If $\mathscr{F}=\{f\}$ for some slice regular function $f$ on $G$, then
	\begin{equation*}
	H_f(\mathcal{G}):=H_{\{f\}}(\mathcal{G})
	\end{equation*}
	is called the set of slice-domains of existence of the function $f$ with respect to $\mathcal{G}$.
\end{defn}

\begin{defn}
	A slice-domain $\mathcal{G}=(G,\pi,x)$ over $\mathbb{H}$ is called a slice-domain of (slice) regularity if there exists a slice regular function $f$ on $G$ such that $\mathcal{G}$ is a slice-domain of existence of $f$ with respect to $\mathcal{G}$, i.e., $\mathcal{G}\in H_{f}(\mathcal{G})$.
\end{defn}

\begin{thm} \label{th-exsd}
	Let $\mathcal{G}=(G,\pi,x)$ be a slice-domain over $\mathbb{H}$, $\mathscr{F}$ be a nonempty set of slice regular functions on $G$, and $\breve{\mathcal{G}}=(\breve{G},\breve{\pi},\breve{x})$ be a $\mathscr{F}$-hull of $\mathcal{G}$. Then the following hold:
	
	1. $\mathcal{G}\prec\breve{\mathcal{G}}$.
	
	2. For each function $f\in\mathscr{F}$, there exists exactly one slice regular function $F$ on $\breve{G}$ with $F|_G=f$.
	
	3. If $\mathcal{G}'=(G',\pi',x')$ is a domain over $\mathbb{H}$ with distinguished point, such that $\mathcal{G}\prec\mathcal{G}'$ and every function $f\in\mathscr{F}$ can be extended slice regularly to $G'$, then $\mathcal{G}'\prec\breve{\mathcal{G}}$.
\end{thm}

\begin{proof}
	Let $\mathscr{G}=\{\mathcal{G}_\lambda\}_{\lambda\in\Lambda}$ be the $\mathscr{F}$-extendible set, where $\Lambda$ is an index set. For each $\lambda\in\Lambda$ and $f\in\mathscr{F}$, let $\mathcal{G}_\lambda=(G_\lambda,\pi_\lambda,x_\lambda)$ and $f_\lambda:G_\lambda\rightarrow\mathbb{H}$ be the slice regular extension of $f$.
	
	1. We notice that $\mathcal{G}\in\{\mathcal{G}_\lambda\}_{\lambda\in\Lambda}$ and $\breve{\mathcal{G}}$ is a union of $\{\mathcal{G}_\lambda\}_{\lambda\in\Lambda}$, then $\mathcal{G}\prec\breve{\mathcal{G}}$.
	
	2. According to the axiom of choice, there exists a subset $\Lambda'$ of $\Lambda$, such that the cardinality of $[\mathcal{G}_\lambda]\cap\{\mathcal{G}_\rho\}_{\rho\in\Lambda'}$ is one for each $\lambda\in\Lambda$, and $\mathcal{G}\in\{\mathcal{G}_\rho\}_{\rho\in\Lambda}$. We set
	\begin{equation*}
	X:=\bigsqcup_{\lambda\in\Lambda'} G_\lambda.
	\end{equation*}
	We define an equivalence $\sim$ on $X$ by $p\sim q$, if and only if, $\pi_\lambda(p)=\pi_\rho (q)$ and $(f_\lambda)_p=(f_\rho)_q$, where $\lambda,\rho\in\Lambda'$, $p\in G_\lambda$ and $q\in G_\rho$. For each $\lambda,\rho\in\Lambda'$, let $\gamma_\lambda$ be a path in $G_\lambda$, and $\gamma_\rho$ be a path in $G_\rho$, such that $\pi_\lambda(\gamma_\lambda)=\pi_\rho(\gamma_\rho)$. For each $t\in[0,1]$, there exist a domain $U$ in $G_\lambda$ containing $\gamma_\lambda(t)$, a domain $V$ in $G_\rho$ containing $\gamma_\lambda(t)$, and a slice-domain $L$ in $\mathbb{H}$, such that
	\begin{equation*}
	\pi_\lambda|_U:U\rightarrow L\qquad\mbox{and}\qquad\pi_\rho|_V:V\rightarrow L
	\end{equation*}
	are homeomorphisms. Then
	\begin{equation*}
	\psi:=\pi_\rho|_V^{-1}\circ\pi_\lambda|_U:U\rightarrow V
	\end{equation*}
	is also a homeomorphism. According to the Identity Principle \ref{th-idh}, $f_\rho|_V\circ\psi=f_\lambda|_U$, if and only if, there exists $p\in U$ with $(f_\rho)_{\psi(p)}=(f_\lambda)_p$. Since $\gamma_\lambda$ is continuous, there exists a domain $W$ in $[0,1]$ containing $t$, such that $\gamma(W)\subset U$. If there exists $t'\in W$ with $(f_\lambda)_{\gamma_\lambda(t')}=(f_\rho)_{\gamma_\rho(t')}$, then
	\begin{equation*}
	(f_\lambda)_{\gamma_\lambda(t'')}=(f_\rho)_{\gamma_\rho(t'')},\qquad\forall\ t''\in W.
	\end{equation*}
	It follows that
	\begin{equation*}
	A:=\{s\in[0,1]:(f_\lambda)_{\gamma_\lambda(s)}=(f_\rho)_{\gamma_\rho(s)}\}
	\end{equation*}
	is open and closed in $[0,1]$. Thence
	\begin{equation*}
	A=[0,1]\qquad\mbox{or}\qquad A=\varnothing.
	\end{equation*}
	We notice that
	\begin{equation*}
	(f_\lambda)_{\gamma_\lambda(0)}=(f_\rho)_{\gamma_\rho(0)}\quad\iff\quad(f_\lambda)_{\gamma_\lambda(1)}=(f_\rho)_{\gamma_\rho(1)}.
	\end{equation*}
	And since $(f_\lambda)_{x_\lambda}=(f_\rho)_{x_\rho}$, then $\sim$ has property $(P)$, it follows that the union equivalence relation $\sim_P$ of $\{G_\iota\}_{\iota\in\Lambda'}$ is finer than $\sim$. Let $\widehat{\mathcal{G}}=(\widehat{G},\widehat{\pi},\widehat{x})$ be the pre-union of $\{\mathcal{G}_\iota\}_{\iota\in\Lambda'}$. According to $\mathcal{G}\in\{\mathcal{G}_\lambda\}_{\lambda\in\Lambda'}$, we can define an extension $F'$ on $\widehat{G}=X/\sim_P$ of $f$, by
	\begin{equation}\label{eq-fgi}
	F'|_{G_\iota}=f_\iota,\qquad\forall\iota\in\Lambda'.
	\end{equation}
	Thanks to (\ref{eq-fgi}), $F'$ is locally slice regular. So $F'$ is the slice regular extension of $f$. According to Theorems \ref{th-union} and \ref{th-uos}, $\breve{G}$ and $\widehat{G}$ are unions of  $\{G_\iota\}_{\iota\in\Lambda'}$. It is clear that $\widehat{G}\cong\breve{G}$. Let $\varphi':\widehat{G}\rightarrow\breve{G}$ be the fiber preserving map from $\widehat{\mathcal{G}}$ to $\breve{\mathcal{G}}$. Thanks to Proposition \ref{pr-fpei}, $\varphi'$ is a homeomorphism. Then we can define the slice regular extension $F$ on $\breve{G}$ of $f$, by
	\begin{equation*}
	F:=F'\circ\varphi'.
	\end{equation*}
	
	3. We notice that $\mathcal{G}'\in\{\mathcal{G}_\lambda\}_{\lambda\in\Lambda}$, then $\mathcal{G}'\prec\breve{G}$.
\end{proof}

\begin{prop}\label{pr-tsh}
	Let $\mathcal{G}:=(G,\pi,x)$ be a slice-domain of regularity and $\gamma$ be a path in $\mathbb{H}$ with $\gamma\prec\mathcal{G}$. Then $\mathcal{T}_\gamma(\mathcal{G})$ is a slice-domain of regularity.
	
	Moreover, if $f$ is a slice regular function on $G$, and $\mathcal{G}$ is a slice-domain of existence of $f$ with respect to $\mathcal{G}$. Then $\mathcal{T}_\gamma(\mathcal{G})$ is also a slice-domain of existence of $f$ with respect to $\mathcal{T}_\gamma(\mathcal{G})$.
\end{prop}

\begin{proof}
	This follows immediately from Proposition \ref{pr-sdt}.
\end{proof}

\begin{prop}\label{pr-sdh}
	Let $\mathcal{G}=(G,\pi,x)$ be a slice-domain of regularity, $y\in G$, and $\mathcal{G}'=(G',\pi',x')$ be a slice-domain over $\mathbb{H}$ with distinguished point. If for each slice regular function $f$ on $G$, there exists a slice regular function $g$ such that  $f_y=g_{x'}$,  then $\mathcal{G}'\prec(G,\pi,y)$.
\end{prop}

\begin{proof}
This follows immediately from Proposition \ref{pr-sdt}, Theorem \ref{th-exsd}, and Proposition \ref{pr-tsh}.
\end{proof}

\section{An extension formula}

In this section, we introduce a new concept so-called real Euclidean and rectify the general extension formula in \cite[Theorem 4.2]{Colombo2009001} (see Proposition \ref{pr-xfh}). We prove that all the slice-domains of regularity are real Euclidean for the proof of the representation formula (see Theorem \ref{th-dere}).

\begin{defn}\label{pr-sdi} 
	A slice-domain $(G,\pi)$ over $\mathbb{H}$ is called real (locally) Euclidean. If for each point $x\in G_{\mathbb{R}}$, there exists a open neighbourhood $U$ of $x$, and a positive real number $r$, such that $\pi_U:U\rightarrow B_{\mathbb{H}}(\pi(x),r)$ is a slice-homeomorphism.
	
	A slice-domain $(G,\pi,x)$ over $\mathbb{H}$ with distinguished point is called real (locally) Euclidean, if $(G,\pi)$ is real Euclidean. A domain $U$ in $G$ is called real (locally) Euclidean (with respect to $\pi$), if $(U,\pi|_U)$ is real Euclidean.
\end{defn}

\begin{prop}\label{pr-xfh} 
	(Extension Formula) Let $I_\lambda\in\mathbb{S}$, and $U_\lambda$ be an open set in $\mathbb{C}_{I_\lambda}$, $\lambda=1,2$. If $I_1\neq I_2$, and there exists a function $f:U_1\cup U_2\rightarrow\mathbb{H}$ such that $f|_{U_1}$ and $f|_{U_2}$ are holomorphic. Then there exists a slice regular function $\widetilde{f}$ on a real Euclidean slice-open set $V$, such that $\widetilde{f}|_{U^+}=f|_{U^+}$, where
	\begin{equation*}
	\mathbb{C}_{I_\lambda}^+:=\{x+yI_\lambda\in\mathbb{C}_{I_\lambda}:x,y\in\mathbb{R},\ \mbox{and}\ y\ge 0\},\qquad\lambda=1,2,
	\end{equation*}
	\begin{equation*}
	U^+:=(U_1\cap\mathbb{C}_{I_1}^+)\cup(U_2\cap\mathbb{C}_{I_2}^+),
	\end{equation*}
	\begin{equation*}
	V':=\{x+yJ\in\mathbb{H}:x,y\in\mathbb{R}\ with\ y\ge 0,\ x+yI_\lambda\in U_\lambda,\ \lambda=1,2,\ and\ J\in\mathbb{S}\},
	\end{equation*}
	and
	\begin{equation*}
	V:=V'\cup U^+.
	\end{equation*}
	Moreover, if $W$ is a slice-connected component of $V$ with $W\cap U^+\neq\varnothing$, then there exists unique $\widetilde{f}|_W$ is the unique slice regular extension on $W$ of $f|_{W\cap U^+}$ and
	\begin{equation}\label{eq-ex}
	\begin{split}
	\widetilde{f}|_W(x+yJ)=&((I_1-I_2)^{-1}I_1+J(I_1-I_2)^{-1})\widetilde{f}|_W(x+yI_1)\\&+((I_2-I_1)^{-1}I_2+J(I_2-I_1)^{-1})\widetilde{f}|_W(x+yI_2)
	\end{split}
	\end{equation}
	for each $x,y\in\mathbb{R}$, $J\in\mathbb{S}$ with $y\ge 0$ and $x+yJ\in W$.
\end{prop}

\begin{proof}
	For each $J\in\mathbb{S}$ and $q\in V_J':=V'\cap\mathbb{C}_J$, there exists $r\in\mathbb{R}^+$, such that the open Euclidean ball $B_{I_\lambda}(P_J^{I_\lambda}(q),r)$ in $\mathbb{C}_{I_\lambda}$ is contained in $U_\lambda$, for each $\lambda\in\{1,2\}$. It is clear that
	\begin{equation*}
	B_J(q,r)\subset V',\qquad\forall q\in V_J'.
	\end{equation*}
	Then the $V_J'$ is an open set in $\mathbb{C}_J$ for each $J\in\mathbb{S}$. And according to Proposition \ref{pr-sdo}, $V'$ is a slice-open set. If $q\in\mathbb{R}$, then the Euclidean ball $B_\mathbb{H}(q,r)$ in $\mathbb{H}$ is contained in $V'$. It follows that $V'$ is real Euclidean. We notice that
	\begin{equation*}
	U^+\cap\mathbb{R}\subset V',
	\end{equation*}
	thence $V$ is also a real Euclidean slice-open set.
	
	We define a function $F$ on $V'$ by
	\begin{equation*}
	\begin{split}
	F(x+yJ):=&((I_1-I_2)^{-1}I_1+J(I_1-I_2)^{-1})f(x+yI_1)\\&+((I_2-I_1)^{-1}I_2+J(I_2-I_1)^{-1})f(x+yI_2)
	\end{split}
	\end{equation*}
	for each $J\in\mathbb{S}$, $x,y\in\mathbb{R}$ with $y\ge 0$, $x+yI_1\in U_1$ and $x+yI_2\in U_2$. By direct calculation (see the proof of \cite[Theorem 3.2]{Colombo2009001}), $F$ is slice regular on $V'$. And since
	\begin{equation*}
	F=f\qquad\mbox{on}\qquad V'\cap U^+,
	\end{equation*}
	then the function $\widetilde{f}:V\rightarrow\mathbb{H}$, defined by
	\begin{equation*}
	\widetilde{f}(p):=\left\{
	\begin{split}
	&F(p),\qquad &&p\in V',
	\\&f(p), &&p\in U^+,
	\end{split}
	\right.
	\end{equation*}
	is a slice regular extension of $f|_{U^+}$.
	
	If $G$ is a slice regular extension on $W$ of $f|_{W\cap U^+}$. Then
	\begin{equation*}
	G=f=\widetilde{f}\qquad\mbox{on}\qquad W\cap U^+.
	\end{equation*}
	According to Identity Principle \ref{th-idh}, $G=\widetilde{f}|_W$.
\end{proof}

According to Proposition \ref{pr-xfh}, the slice regular extension $\widetilde{f}$ of $f$ and $f$ itself only coincide on $U^+$, and may not coincide on $(U_1\cup U_2)\cap V$ as \cite[Theorem 4.2]{Colombo2009001} . The reason is the value of $F$ at $x+yJ$ in (\ref{eq-ex}) only need values of $f$ at two points, but there may be four points in $U_1\cup U_2$ to choice, i.e., $x\pm yI_1$ and $x\pm yI_2$. So $F$ may not be well defined, such as $U_1=\Omega_I$ and $U_2=\Omega_K$, where $\Omega,I$ is defined in Example \ref{th-ce} and $K\in\mathbb{S}$ with $K\perp I$.

\begin{thm}\label{th-sdre}
	All the slice-domains of regularity are real Euclidean.
\end{thm}

\begin{proof}
Suppose $\mathcal{G}=(G,\pi,x)$ is a slice-domain of regularity, which is not real Euclidean. There exists a slice regular function $f$ on $G$, such that $\mathcal{G}$ is a slice-domain of existence of $f$ (with respect to $\mathcal{G}$). Since $\mathcal{G}$ is not real Euclidean, then there exists $q\in G_\mathbb{R}$ and a path $\gamma$ in $G$ form $x$ to $q$, such that
\begin{equation}\label{eq-ct}
(B_\mathbb{H}(\pi(q),r_0),id_{B_\mathbb{H}(\pi(q),r_0)},\pi(q))\nprec\mathcal{T}_{\pi(\gamma)}\mathcal{G},\qquad\forall\ r_0>0.
\end{equation}
Then there exist a positive real number $r\in\mathbb{R}^+$ and a real-connected slice-domain $U$ in $G$ containing $q$, such that $\pi|_{U_I}:U_I\rightarrow B_I(\pi(q),r)$
is a homeomorphism with respect to $\tau(G_I)$ and $\tau(\mathbb{C}_I)$, and
\begin{equation*}
\pi(U)\subset\mathbb{B}:=B_\mathbb{H}(\pi(q),r).
\end{equation*}
Thanks to Proposition \ref{pr-xfh} (setting $I_1=-I_2=I$ and $U_1=U_2=\mathbb{B}_I$), there exists a slice regular extension $f':\mathbb{B}\rightarrow\mathbb{H}$ of the slice regular function $f\circ\pi|_U^{-1}:\mathbb{B}_I\rightarrow\mathbb{H}$. We notice that
\begin{equation*}
(G,\pi,q)\succ(U,\pi_U,q),\qquad(\mathbb{B},id_{\mathbb{B}},\pi(q))\succ(U,\pi_U,q),
\end{equation*}
and $f$, $f'$ are slice regular extension of the slice regular function $f|_U$, and thanks to Theorem \ref{th-exsd}, we have for each union $\mathcal{G}'=(G',\pi',q')$ of $(G,\pi,q)$ and $(\mathbb{B},id_{\mathbb{B}},\pi(q))$, there exists a slice regular function $\widetilde{f}$ on $G'$. Since $\mathcal{G}$ is the slice-domain of existence of $f$,
\begin{equation*}
\mathcal{G}\succ\mathcal{T}_{\pi(\gamma)^{-1}}(\mathcal{G}').
\end{equation*}
According to Proposition \ref{pr-sdt}, it follows that $\mathcal{T}_{\pi(\gamma)}(\mathcal{G})\succ\mathcal{G}'$. Therefore
\begin{equation*}
(\mathbb{B},id_\mathbb{B},\pi(q))\prec\mathcal{G}'\prec\mathcal{T}_{\pi(\gamma)}\mathcal{G}
\end{equation*}
which is a contradiction to (\ref{eq-ct}).
\end{proof}

\section{Technical lemmas}

In this section, we prove two technical lemmas for the representation formula in Theorem \ref{th-dere}. For each $N\in\mathbb{N}^+$, $\sigma_N$ defined in Lemma \ref{le-ct}, actually corresponding to a complex structure on $\mathbb{R}^{2^N}$ (see Remark \ref{rm-cs}), has an intertwining relation with the elements in $\mathbb{S}^N$ (see Lemmas \ref{le-ct} and \ref{pr-ct}).

For each $N\in\mathbb{N}^+$, $I=(I_1,I_2,...,I_N)\in\mathbb{S}^N$ and $m\in\{1,2,...,2^N\}$, we set
\begin{equation}\label{eq-im}
I(m):=\prod_{\imath=N}^1 (I_\imath I_{\imath-1})^{m_\imath}=(I_N I_{N-1})^{m_N}(I_{N-1} I_{N-2})^{m_{N-1}}...(I_{1} I_{0})^{m_{1}},
\end{equation}
where $I_{0}:=1$ and $(m_N m_{N-1} ... m_1)_2$ is the binary number of $m-1$. We define a map
\begin{equation*}
\zeta:\bigsqcup_{\imath\in\mathbb{N}^+}\mathbb{S}^\imath\rightarrow\bigsqcup_{\imath\in\mathbb{N}^+}\mathbb{H}^{2^\imath},\qquad I\mapsto(I(1),I(2),...,I(2^\jmath)),\qquad\forall\ \jmath\in\mathbb{N}^+\ \mbox{and}\ I\in\mathbb{S}^\jmath.
\end{equation*}

For each set $A$ and $\imath,\jmath\in\mathbb{N}^+$, we denote the set of all $\imath\times\jmath$ matrices of $A$ by $M_{\imath\times\jmath}(A)$, and the set of all $\imath\times\imath$ matrices of $A$ by $M_\imath(A)$. We denote the $\imath\times\imath$ identity matrix in $M_\imath(\mathbb{H})$ by $\mathbb{I}_\imath$, and the $\imath\times\imath$ zero matrix in $M_\imath(\mathbb{H})$ by $0_\imath$ for each $\imath\in\mathbb{N}^+$.

For each matrix $E$, we denote the transpose of $E$ by $E^T$.

For each $\imath\in\mathbb{N}^+$, we say that $A\in M_\imath(\mathbb{H})$ is invertible, if there exists a matrix $B\in M_\imath(\mathbb{H})$, such that $AB=BA=\mathbb{I}_\imath$.

For each $n,m\in\mathbb{N}^+$, $A=\{a_{\imath,\jmath}\}_{n\times m}\in M_{n\times m}(\mathbb{H})$, and $q\in\mathbb{H}$, we set
\begin{equation*}
qA:=\{q\cdot a_{\imath,\jmath}\}_{n\times m}\qquad\mbox{and}\qquad Aq:=\{a_{\imath,\jmath}\cdot q\}_{n\times m}.
\end{equation*}

\begin{prop}\label{pr-abi}
	(\cite[Proposition 4.1]{Zhang1997001}) Let $\imath\in\mathbb{N}^+$, and $A,B\in M_\imath(\mathbb{H})$. If $AB=\mathbb{I}_\imath$, then $BA=\mathbb{I}_\imath$.
\end{prop}

For each $N\in\mathbb{N}^+$ and $J=\{J_{\imath,\jmath}\}_{2^N\times N}\in M_{2^N\times N}(\mathbb{S})$, we denote the $\imath$-th row vector of $J$ by
\begin{equation*}
J_\imath:=(J_{\imath,1},J_{\imath,2},...,J_{\imath,N})\in\mathbb{S}^N,\qquad\forall\ \imath\in\{1,2,...,2^N\}.
\end{equation*}
We set
\begin{equation*}
J^{(l)}:=\{J_{\imath,\jmath}\}_{2^l\times l}\qquad\mbox{and}\qquad J_\imath^{(l)}:=(J_{\imath,1},J_{\imath,2},...,J_{\imath,l})
\end{equation*}
for each $l\in\{1,2,...,N\}$ and $\imath\in\{1,2,...,2^N\}$. We define a map
\begin{equation*}
\mathcal{M}:\bigsqcup_{\imath\in\mathbb{N}^+}M_{2^\imath\times\imath}(\mathbb{S})\rightarrow \bigsqcup_{\imath\in\mathbb{N}^+}M_{2^\imath}(\mathbb{H}),
\end{equation*}
by
\begin{equation*}
\mathcal{M}(K):=(\zeta(K_1)^T,\zeta(K_2)^T,...,\zeta(K_{2^\jmath})^T)^T
\end{equation*}
for each $\jmath\in\{1,2,...,N\}$ and $K\in M_{2^\jmath\times\jmath}(\mathbb{S})$.

\begin{defn}\label{df-fsr}
	Let $N\in\mathbb{N}^+$ and $J=(J_{\imath,\jmath})_{2^N\times N}\in M_{2^N\times N}(\mathbb{S})$. We call $J$ has full slice-rank, if $\mathcal{M}(J^{(\imath)})$ is invertible for each $\imath\in\{1,2,...,N\}$.
\end{defn}

\begin{lem} \label{le-ct} 
	Let $N\in\mathbb{N}^+$ and $K=(K_1,K_2,...,K_N)\in\mathbb{S}^N$. Then
	\begin{equation*}
	K_N\zeta(K)=\zeta(K)\sigma_N,
	\end{equation*}
	where $$K_N\zeta(K)=(K_NK(1),K_NK(2),...,K_NK(2^N)),$$  $$\sigma_N:=(\sigma^{(N)}_{\imath,\jmath})_{2^N\times 2^N}\in M_{2^N}(\mathbb{R}),$$  and
	\begin{equation*}
	\sigma_{\imath,\jmath}^{(N)}:=\left\{
	\begin{split}
	&(-1)^{N+\jmath},\qquad&&\imath+\jmath=2^N+1,\\
	&0,&&otherwise,\\
	\end{split}\right.
	\end{equation*}
	for each $\imath,\jmath\in\{1,2,...,2^N\}$.
\end{lem}

\begin{proof}
	For each $m\in\{1,2,...,2^N\}$, let $(m_N m_{N-1} ... m_1)_2$ be the binary number of $m-1$. If $(m'_N m'_{N-1} ... m'_1)_2$ is the binary number of $(2^N+1-m)-1$, then
	\begin{equation*}
	m'_\imath=1-m_\imath,\qquad\forall\ \imath\in\{1,2,...,N\}.
	\end{equation*}
	We notice that
	\begin{equation*}
	(-1)^m=(-1)^{m_1-1},
	\end{equation*}
	and according to (\ref{eq-im}), it follows that
	\begin{equation*}
	\begin{split}
	K_NK(m)=&K_N\prod_{\imath=N}^1 (K_\imath K_{\imath-1})^{m_\imath}
	\\=&K_N^{m_N+1}\prod_{\imath={N-1}}^1(K_\imath^{m_\imath}K_\imath^{m_{\imath+1}})
	\\=&\prod_{\imath=N}^1(K_\imath^{m_\imath+1}K_{\imath-1}^{m_\imath+1})(-1)^{N-1}
	\\=&\prod_{\imath=N}^1(K_\imath^{1-m_\imath}K_{\imath-1}^{1-m_\imath})(-1)^{N-1}(-1)^{m_1}
	\\=&\prod_{\imath=N}^1(K_\imath^{m'_\imath}K_{\imath-1}^{m'_\imath})(-1)^{N+(m_1-1)}
	\\=&K(2^N+1-m)(-1)^{N+m},
	\end{split}
	\end{equation*}
	where $K_0=1$. Thence
	\begin{equation*}
	\begin{split}	
	K_N\zeta(K)=&(K_NK(1),K_NK(2),...,K_NK(2^N))\\=&(a_1,a_2,...,a_{2^N})\\=&\zeta(K)\sigma_N,
	\end{split}
	\end{equation*}
	where
	\begin{equation*}
	\begin{split}
	a_\jmath:=&K_N K(\jmath)=K(2^N+1-\jmath)(-1)^{N+\jmath}\\=&K(2^N+1-\jmath)\sigma^{(N)}_{2^N+1-\jmath,\jmath}\\=&\sum_{\imath=1}^NK(\imath)\sigma^{(N)}_{\imath,\jmath}
	\end{split}
	\end{equation*}
	 for each $\jmath\in\{1,2,...,2^N\}$.
\end{proof}

\begin{defn}
	A complex structure on a real vector space $V$ is a automorphism $L:V\rightarrow V$ that squares to minus the identity: $L\circ L=-\Id$.
\end{defn}

\begin{rmk}\label{rm-cs}
	Let $N\in\mathbb{N}^+$, we define a morphism
	\begin{equation*}
	L:\mathbb{H}^{2^N}\rightarrow\mathbb{H}^{2^N},\qquad q\mapsto\sigma_N q.
	\end{equation*}
	We notice that
	\begin{equation*}
	L(x)=\sigma_N\sigma_N q=-\mathbb{I}_{2^N} q=-q,\qquad\forall\ x\in\mathbb{H}^{2^N}.
	\end{equation*}
	It is clear that $L$ is a complex structure on $\mathbb{H}^{2^N}$.
\end{rmk}

\begin{prop} \label{pr-ct}
	Let $N\in\mathbb{N}^+$ and $J=(J_{\imath,\jmath})_{2^N\times N}\in M_{2^N\times N}(\mathbb{S})$. If $\mathcal{M}(J)$ is invertible, then
	\begin{equation*}
	\sigma_N \mathcal{M}(J)^{-1}=\mathcal{M}(J)^{-1} D_N(J),
	\end{equation*}
	where
	\begin{equation*}
	D_N(J):=\diag(J_{1,N},J_{2,N},...,J_{2^N,N})\in M_{2^N}(\mathbb{S}).
	\end{equation*}
\end{prop}

\begin{proof}
	According to Lemma \ref{le-ct},
	\begin{equation*}
	\begin{split}
	D_N(J)\mathcal{M}(J)=&(J_{1,N}\zeta(J_1)^T,J_{2,N}\zeta(J_2)^T,...,J_{2^N,N}\zeta(J_{2^N})^T)^T
	\\=&((\zeta(J_1)\sigma_N)^T,(\zeta(J_2)\sigma_N)^T,...,(\zeta(J_{2^N})\sigma_N)^T)^T
	\\=&(\zeta(J_1)^T,\zeta(J_2)^T,...,\zeta(J_{2^N})^T)^T\sigma_N
	\\=&\mathcal{M}(J)\sigma_N.
	\end{split}
	\end{equation*}
	Then,
	\begin{equation*}
	\begin{split}
	\sigma_N\mathcal{M}(J)^{-1}=&\mathcal{M}(J)^{-1}\mathcal{M}(J)\sigma_N\mathcal{M}(J)^{-1}
	\\=&\mathcal{M}(J)^{-1}D_N(J)\mathcal{M}(J)\mathcal{M}(J)^{-1}
	\\=&\mathcal{M}(J)^{-1}D_N(J).
	\end{split}
	\end{equation*}
\end{proof}

\section{Representation formula over slice-domains of regularity}

The representation formula is a key result in the theory of slice regular function. It is introduced in \cite{Colombo2010001} and then extended to a general formula in \cite{Colombo2009001}. In this section, we will prove a representation formula over slice-domains of regularity (see Theorem \ref{th-dere}), which is an earlier version of
\cite[Theorem 3.2]{Colombo2009001}. In particular, our representation formula is the same as the classical one, when the slice-domain of regularity is an axially symmetric slice domain in $\mathbb{H}$ (see Remark \ref{rm-sdr}).

Let $\mathcal{G}$ be a slice-domain over $\mathbb{H}$ with distinguished point, $N\in\mathbb{N}^+$ and $\gamma=(\gamma_1,\gamma_2,...,\gamma_N)$ be an $N$-part path in $\mathbb{C}$ (resp. $\mathbb{H}$ or $\mathcal{G}$). For each $t\in[0,1]$, we define a finite-part path $\gamma[t]$ in $\mathbb{C}$ (resp. $\mathbb{H}$ or $\mathcal{G}$) by
\begin{equation*}
\gamma[t]:=\left\{
\begin{split}
&(\gamma_1,\gamma_2,...,\gamma_{\lfloor Nt\rfloor},\gamma_{(Nt)}),\quad&&t\in[0,1),\\
&\gamma,&&t=1,\\
\end{split}\right.
\end{equation*}
where $\gamma_{(Nt)}$ is the path in $\mathbb{C}$ (resp. $\mathbb{H}$ or $\mathcal{G}$), defined by
\begin{equation*}
\gamma_{(Nt)}(s):=\left\{
\begin{split}
&\gamma_{\lfloor Nt+1\rfloor}((Nt-\lfloor Nt\rfloor)s),\quad&&Nt\neq\lfloor Nt\rfloor,\\
&\gamma(t),&&otherwise,\\
\end{split}\right.\qquad\forall\ s\in[0,1].
\end{equation*}
Let $\gamma[t^-]$ be a finite-part path in $\mathbb{C}$ (resp. $\mathbb{H}$ or $\mathcal{G}$), defined by
\begin{equation*}
\gamma[t^-]:=\left\{
\begin{split}
&(\gamma_1,\gamma_2,...,\gamma_{tN}),\quad&&t\in\{\frac{1}{N},\frac{2}{N},...,\frac{N}{N}\},\\
&\gamma[t],&&t\in[0,1]\backslash\{\frac{1}{N},\frac{2}{N},...,\frac{N}{N}\}.\\
\end{split}\right.
\end{equation*}

\begin{thm}(Representation Formula)\label{th-dere}
	Let $N\in\mathbb{N}^+$, $J\in M_{2^N\times N}(\mathbb{S})$ with full slice-rank, $\mathcal{G}=(G,\pi,x_0)$ be a slice-domain of regularity with $\pi(x_0)\in\mathbb{R}$, and $\gamma$ be an $N$-part path in $\mathbb{C}$. If
	\begin{equation}\label{eq-rj}
	\gamma^{J_\imath}\prec\mathcal{G},\qquad\imath=1,2,...,2^N,
	\end{equation}
	then
	\begin{equation*}
	\gamma^K\prec\mathcal{G},\qquad\forall\ K\in\mathbb{S}^N.
	\end{equation*}
	
	Moreover, if $f$ is a slice regular function on $G$, and $\mathcal{G}$ is a slice-domain of existence of $f$ with respect to $\mathcal{G}$, then
	\begin{equation*}
	f(\gamma^K_{\mathcal{G}}(1))=\zeta(K)\mathcal{M}(J)^{-1}f(\gamma^J_{\mathcal{G}}(1)),
	\end{equation*}
	where
	\begin{equation*}
	f(\gamma^J_{\mathcal{G}}(1)):=(f(\gamma^{J_1}_{\mathcal{G}}(1)),f(\gamma^{J_2}_{\mathcal{G}}(1)),...,f(\gamma^{J_{2^N}}_{\mathcal{G}}(1)))^T.
	\end{equation*}
\end{thm}

\begin{proof}
	Since $\mathcal{G}$ is a slice-domain of regularity, there exists a nonempty subset $W$ of $\mathcal{SR}(G)$, such that for each $h\in W$, $\mathcal{G}$ is a slice-domain of existence of $h$ with respect to $\mathcal{G}$.
	
	\textbf{1.} Fix $f\in W$, and let $A$ be a subset of $[0,1]$, such that $t\in A$, if and only if the following properties hold:
	
	(a). $t\in[0,1]$ and $\gamma^I[t]\prec\mathcal{G}$ for each $I\in\mathbb{S}^N$.
	
	(b). If $t\neq 0$, then there exist a positive real number $r>0$, and domains $U^I$ in $G_{I_{N_t}}$ for each $I\in\mathbb{S}^N$ with $\gamma_{\mathcal{G}}^I[t](1)\in U^I$ and
	\begin{equation*}
	\pi|_{U^I}:U^I\rightarrow P_{I_{N_t}}(B_\mathbb{C}(\gamma(t),r))
	\end{equation*}
	being a homeomorphism with respect to topologies $\tau(G_{I_{N_t}})$ and $\tau(\mathbb{C}_{I_{N_t}})$, such that
	\begin{equation}\label{eq-re1}
	f\circ\pi|_{U^K}^{-1}(x+yK_{N_t})=\zeta(K^{(N_t)})\mathcal{M}(J^{(N_t)})^{-1}F(x+yJ^{(N_t)})
	\end{equation}
	for each $K\in\mathbb{S}^N$, and $x,y\in\mathbb{R}$ with $y\ge 0$ and $x+yi\in B_{\mathbb{C}}(\gamma(t),r)$, where
	\begin{equation*}
	N_t:=\lceil Nt\rceil
	\end{equation*}
	and
	\begin{equation*}
	F(x+yJ^{(N_t)}):=\left(
	\begin{matrix}
	&f\circ\pi|_{U^{J_1}}^{-1}(x+yJ_{1,N_t})
	\\&f\circ\pi|_{U^{J_2}}^{-1}(x+yJ_{2,N_t})
	\\&\vdots
	\\&f\circ\pi|_{U^{J_{2^{N_t}}}}^{-1}(x+yJ_{2^{N_t},N_t})
	\end{matrix}\right).
	\end{equation*}
	
	(c). If $\{Nt\}=0$ and $t\neq 1$, then there exist a positive real number $r'$ and domains $V^I$ in $G_{I_{N_t+1}}$ for each $I\in\mathbb{S}^N$ with $\gamma_{\mathcal{G}}^I[t](1)\in V^I$ and
	\begin{equation*}
	\pi|_{V^I}:V^I\rightarrow P_{I_{N_t+1}}(B_\mathbb{C}(\gamma(t),r))
	\end{equation*}
	being a homeomorphism with respect to topologies $\tau(G_{I_{(N_t+1)}})$ and $\tau(\mathbb{C}_{I_{(N_t+1)}})$, such that
	\begin{equation}\label{eq-re2}
	f\circ\pi|_{V^K}^{-1}(x+yK_{N_t+1})=\zeta(K^{(N_t+1)})\mathcal{M}(J^{(N_t+1)})^{-1}F'(x+yJ^{(N_t+1)}),
	\end{equation}
	for each $K\in\mathbb{S}^N$, and $x,y\in\mathbb{R}$ with $y\ge 0$ and $x+yi\in B_\mathbb{C}(\gamma(t),r)$, where
	\begin{equation*}
	N_t:=\lceil Nt\rceil
	\end{equation*}
	and
	\begin{equation*}
	F'(x+yJ^{(N_t+1)}):=\left(
	\begin{matrix}
	&f\circ\pi|_{V^{J_1}}^{-1}(x+yJ_{1,N_t+1})
	\\&f\circ\pi|_{V^{J_2}}^{-1}(x+yJ_{2,N_t+1})
	\\&\vdots
	\\&f\circ\pi|_{V^{J_{2^{N_t+1}}}}^{-1}(x+yJ_{2^{N_t+1},N_t+1})
	\end{matrix}\right).
	\end{equation*}
	
	We will prove that $A=[0,1]$, by contradiction. If $A\neq[0,1]$, then we set
	\begin{equation}\label{eq-ti}
	t_1:=\inf([0,1]\backslash A).
	\end{equation}

	\textbf{1). We will prove that $t_1\notin A$, in this step.}
	
	If $t_1\in A$, then $t_1\neq 1$ (if not, thus $[0,1]\subset A$, which is a contradiction). According to (a),
	\begin{equation}\label{eq-ri}
	\gamma^I[t_1]\prec\mathcal{G},\qquad\forall\ I\in\mathbb{S}^N.
	\end{equation}
	According to (b) and (c), there exist a real number $r_1>0$ and domains $U^I_1$ in $G_{I_{N_2}}$ for each $I\in\mathbb{S}^N$ with $\gamma_{\mathcal{G}}^I[t_1](1)\in U^I_1$ and
	\begin{equation*}
	\pi|_{U^I_1}:U^I_1\rightarrow P_{I_{N_2}}(B_\mathbb{C}(\gamma(t_1),r_1))
	\end{equation*}
	being a homeomorphism with respect to topologies $\tau(G_{I_{N_2}})$ and $\tau(\mathbb{C}_{I_{N_2}})$, such that
	\begin{equation}\label{eq-re2-1}
	f\circ\pi|_{U^K_{t_1}}^{-1}(x+yK_{N_2})=\zeta(K^{(N_2)})\mathcal{M}(J^{(N_2)})^{-1}F_2(x+yJ^{(N_2)}),
	\end{equation}
	for each $K\in\mathbb{S}^N$, and $x,y\in\mathbb{R}$ with $y\ge 0$ and $x+yi\in B_\mathbb{C}(\gamma(t_1),r_1)$, where
	\begin{equation*}
	N_2:=\lfloor N{t_1}+1\rfloor,
	\end{equation*}
	and
	\begin{equation*}
	F_2(x+yJ^{(N_2)}):=\left(
	\begin{matrix}
	&f\circ\pi|_{U^{J_1}_1}^{-1}(x+yJ_{1,N_2})
	\\&f\circ\pi|_{U^{J_2}_1}^{-1}(x+yJ_{2,N_2})
	\\&\vdots
	\\&f\circ\pi|_{U^{J_{2^{N_2}}}_1}^{-1}(x+yJ_{2^{N_2},N_2})
	\end{matrix}\right).
	\end{equation*}
	
	Since $\gamma$ is continuous, there exists $t_2\in(t_1,\frac{N_2}{N})$, such that
	\begin{equation*}
	\gamma([t_1,t_2])\subset B_\mathbb{C}(\gamma(t_1),r_1).
	\end{equation*}
	Then for each $t_3\in[t_1,t_2]$, there exist a real number $r_3>0$, such that
	\begin{equation*}
	B_\mathbb{C}(\gamma(t_3),r_3)\subset B_\mathbb{C}(\gamma(t_1),r_1).
	\end{equation*}
	We notice that for each $I\in\mathbb{S}^N$,
	\begin{equation*}
	\gamma^I([t_1,t_2])=P_{I_{N_2}}(\gamma([t_1,t_2]))\subset P_{I_{N_2}}(B_\mathbb{C}(\gamma(t_1),r_1)),
	\end{equation*}
	and thanks to (\ref{eq-ri}), it follows that $\gamma^I[t_3]\prec\mathcal{G}$ and
	\begin{equation*}
	U^I_3:=\pi|_{U^I_1}^{-1}(P_{I_{N_2}}(B_\mathbb{C}(\gamma(t_3),r_3)))
	\end{equation*}
	is a domain in $G_{I_{N_2}}$ with $\gamma^I_{\mathcal{G}}[t_3](1)\in U^I_3\subset U^I_1$ and
	\begin{equation*}
	\pi|_{U^I_3}:U^I_3\rightarrow P_{I_{N_2}}(B_\mathbb{C}(\gamma(t_3),r_3))
	\end{equation*}
	being a homeomorphism with respect to topologies $\tau(G_{I_{N_2}})$ and $\tau(\mathbb{C}_{I_{N_2}})$. According to (\ref{eq-re2-1}), and since $U^I_3\subset U^I_1$ for each $I\in\mathbb{S}^N$, it is clear that
	\begin{equation*}
	\begin{split}
	f\circ\pi|_{U^K_3}^{-1}(x+yK_{N_2})=&f\circ\pi|_{U^K_1}^{-1}(x+yK_{N_2})
	\\=&\zeta(K^{(N_2)})\mathcal{M}(J^{(N_2)})^{-1}F_2(x+yJ^{(N_2)})
	\\=&\zeta(K^{(N_2)})\mathcal{M}(J^{(N_2)})^{-1}F_3(x+yJ^{(N_2)})
	\end{split}
	\end{equation*}
	for each $K\in\mathbb{S}^N$, and $x,y\in\mathbb{R}$ with $y\ge 0$ and $x+yi\in B_{\mathbb{C}}(\gamma(t_3),r_3)$, where \begin{equation*}
	N_2:=\lfloor N{t_1}+1\rfloor=\lfloor N{t_3}+1\rfloor=\lceil N{t_3}\rceil,
	\end{equation*}
	and
	\begin{equation*}
	\begin{split}
	F_3(x+yJ^{(N_2)}):=&\left(
	\begin{matrix}
	&f\circ\pi|_{U^{J_1}_3}^{-1}(x+yJ_{1,N_2})
	\\&f\circ\pi|_{U^{J_2}_3}^{-1}(x+yJ_{2,N_2})
	\\&\vdots
	\\&f\circ\pi|_{U^{J_{2^{N_2}}}_3}^{-1}(x+yJ_{2^{N_2},N_2})
	\end{matrix}\right)
	\\=&\left(
	\begin{matrix}
	&f\circ\pi|_{U^{J_1}_1}^{-1}(x+yJ_{1,N_2})
	\\&f\circ\pi|_{U^{J_2}_1}^{-1}(x+yJ_{2,N_2})
	\\&\vdots
	\\&f\circ\pi|_{U^{J_{2^{N_2}}}_1}^{-1}(x+yJ_{2^{N_2},N_2})
	\end{matrix}\right)
	\\=&F_2(x+yJ^{(N_2)}).
	\end{split}
	\end{equation*}
	Then $t_3\in A$ for each $t_3\in[t_1,t_2]$. Therefore
	\begin{equation*}
	[t_1,t_2]\subset A.
	\end{equation*}
	And according to (\ref{eq-ti}), we have $[0,t_1)\subset A$. It follows that $[0,t_2]\subset A$ and
	\begin{equation*}
	t_1=\inf([0,1]\backslash A)\ge t_2>t_1,
	\end{equation*}
	which is a contradiction. So $t_1\notin A$.
	
	\textbf{2). We will prove that $t_1\neq 0$, in this step.}
	
	If $t_1=0$, then $0\notin A$ by 1). According to Theorem \ref{th-sdre}, $\mathcal{G}$ is real Euclidean. It is clear that there exist a domain $U_0$ in $G$ containing $x_0$ and a positive real number $r_0>0$, such that
	\begin{equation*}
	\pi|_{U_0}:U_0\rightarrow B_\mathbb{H}(\pi(x_0),r_0)
	\end{equation*}
	is a slice-homeomorphism. Thence $f\circ\pi|_{U_0}^{-1}$ is a slice regular function on $B_\mathbb{H}(\pi(x_0),r_0)$, and there exists $t_0\in[0,\frac{1}{N})$, such that $\gamma([0,t_0])\subset B_{\mathbb{C}}(\pi(x_0),r_0)$.
	
	Let $t=0$, we have $N_t=N_0=1$, $\zeta(K^{(1)})=(1,K_1)$ for each $K\in\mathbb{S}$, and
	\begin{equation*}
	\mathcal{M}(J^{(1)})=\left(\begin{matrix}1&J_{1,1}\\1&J_{2,1}\end{matrix}\right).
	\end{equation*}
	We notice that
	\begin{equation*}
	\left(\begin{matrix} (J_{1,1}-J_{2,1})^{-1}J_{1,1}&-(J_{1,1}-J_{2,1})^{-1}J_{2,1}\\(J_{1,1}-J_{2,1})^{-1}&-(J_{1,1}-J_{2,1})^{-1}
	\end{matrix}\right)
	\left(\begin{matrix}
	1&J_{1,1}\\1&J_{2,1}
	\end{matrix}\right)=
	\left(\begin{matrix} 1&0\\0&1
	\end{matrix}\right).
	\end{equation*}
	According to Proposition \ref{pr-abi},
	$$\mathcal{M}(J^{(1)})^{-1}=	\left(\begin{matrix} (J_{1,1}-J_{2,1})^{-1}J_{1,1}&-(J_{1,1}-J_{2,1})^{-1}J_{2,1}\\(J_{1,1}-J_{2,1})^{-1}&-(J_{1,1}-J_{2,1})^{-1}
	\end{matrix}\right).$$
	Thanks to (\ref{eq-ex}) in Proposition \ref{pr-xfh},
	\begin{equation*}
	\begin{split}
	&f\circ\pi|_{U_0}^{-1}(x+yK_1)\\=&((J_{1,1}-J_{2,1})^{-1}J_{1,1}+K_1(J_{1,1}-J_{2,1})^{-1})f\circ\pi|_{U_0}^{-1}(x+yJ_{1,1})\\&+((J_{2,1}-J_{1,1})^{-1}J_{2,1}+K_1(J_{2,1}-J_{1,1})^{-1})f\circ\pi|_{U_0}^{-1}(x+yJ_{2,1})\\
	=&(1,K_1)
	\left(\begin{matrix} (J_{1,1}-J_{2,1})^{-1}J_{1,1}&-(J_{1,1}-J_{2,1})^{-1}J_{2,1}\\(J_{1,1}-J_{2,1})^{-1}&-(J_{1,1}-J_{2,1})^{-1}
	\end{matrix}\right)
	\left(\begin{matrix}
	f\circ\pi|_{U_0}^{-1}(x+yJ_{1,1})\\f\circ\pi|_{U_0}^{-1}(x+yJ_{2,1})
	\end{matrix}\right)\\
	=&\zeta(K^{(1)})\mathcal{M}(J^{(1)})^{-1}F_0(x+yJ^{(1)})
	\end{split}
	\end{equation*}
	for each $x,y\in\mathbb{R}$ with $y\ge 0$ and $x+yi\in B_{\mathbb{C}}(\pi(x),r_0)$, where
	\begin{equation*}
	F_0(x+yJ^{(1)}):=\left(
	\begin{matrix}
	&f\circ\pi|_{V^{J_1}_0}^{-1}(x+yJ_{1,1})
	\\&f\circ\pi|_{V^{J_2}_0}^{-1}(x+yJ_{2,1})
	\end{matrix}\right)
	\end{equation*}
	and
	\begin{equation*}
	V^{J_\imath}_0:=U_0\cap G_{J_{\imath,1}},\qquad\imath=1,2.
	\end{equation*}
	We notice that $\gamma^K[0]\prec\mathcal{G}$ for each $K\in\mathbb{S}^N$, it follows that $0\in A$, which is a contradiction. Then $t_1\neq 0$.
	
	\textbf{3). We will prove that $\{Nt_1\}=0$ and $t_1\neq 1$, in this step.}
	
	According to (\ref{eq-rj}), for each $\imath\in\{1,2,...,2^{N}\}$,
	\begin{equation*}
	\gamma^{J_\imath}\prec\mathcal{G},
	\end{equation*}
	it follows that there exists a domain $V^{J_\imath}_1$ in $G_{J_\imath,N_4}$ containing $\gamma^{J_\imath}[t_1](1)$, such that
	\begin{equation*}
	\pi|_{V^{J_\imath}_1}:V^{J_\imath}_1\rightarrow\pi(V^{J_\imath}_1)
	\end{equation*}
	is a homeomorphism with respect to the subspace topology of $\tau(G_{J_{\imath,N_4}})$ and $\tau(\mathbb{C}_{J_{\imath,N_4}})$, where
	\begin{equation*}
	N_4:=\lceil Nt_1\rceil.
	\end{equation*}
	Thanks to $\gamma^{J_\imath}(t_1)\in\mathbb{C}_{J_{\imath,N_4}}$, then there exists a real number $r'_\imath>0$ such that
	\begin{equation*}
	B_{J_{\imath,N_4}}(\gamma^{J_\imath}(t_1),r'_\imath)\subset\pi(V^{J_\imath}_1),\qquad\forall\ \imath\in\{1,2,...,2^{N_4}\}.
	\end{equation*}
	We set
	\begin{equation*}
	r_4:=\min\{r'_1,r'_2,...,r'_{2^{N_4}}\}\qquad\mbox{and}\qquad\mathbb{B}_4:=B_{\mathbb{C}}(\gamma(t_1),r_4),
	\end{equation*}
	it follows that
	\begin{equation*}
	P_{J_{\imath,N_4}}(\mathbb{B}_4)\subset B_{J_{\imath,N_4}}(\gamma^{J_\imath}(t_1),r'_\imath)\subset\pi(V^{J_\imath}_1),\qquad\forall\ \imath\in\{1,2,...,2^{N_4}\}.
	\end{equation*}
	Therefore
	\begin{equation*}
	U^{J_\imath}_4:=\pi|_{V^{J_\imath}_1}^{-1}(P_{J_{\imath,N_4}}(\mathbb{B}_4))
	\end{equation*}
	is a domain in $G_{J_{\imath,N_4}}$ with
	\begin{equation*}
	\pi_{U^{J_\imath}_4}:U^{J_\imath}_4\rightarrow P_{J_{\imath,N_4}}(\mathbb{B}_4)
	\end{equation*}
	being a homeomorphism with respect to topologies $\tau(G_{J_{\imath,N_4}})$ and $\tau(\mathbb{C}_{J_{\imath,N_4}})$ for each $\imath\in\{1,2,...,2^{N_4}\}$. For each $K\in\mathbb{S}^N$, we define a function
	\begin{equation*}
	g^{(K)}:P_{K_{N_4}}(\mathbb{B}_4)\rightarrow\mathbb{H}
	\end{equation*}
	by
	\begin{equation}\label{eq-gk}
	g^{(K)}(x+yK_{N_4}):=\zeta(K^{(N_4)})\mathcal{M}(J^{(N_4)})^{-1}F_4(x+yJ^{(N_4)})
	\end{equation}
	for each $x,y\in\mathbb{R}$ with $y\ge 0$ and $x+yi\in\mathbb{B}_4$, where
	\begin{equation*}
	F_4(x+yJ^{(N_4)}):=\left(
	\begin{matrix}
	&f\circ\pi|_{U^{J_1}_4}^{-1}(x+yJ_{1,N_4})
	\\&f\circ\pi|_{U^{J_2}_4}^{-1}(x+yJ_{2,N_4})
	\\&\vdots
	\\&f\circ\pi|_{U^{J_{2^{N_4}}}_4}^{-1}(x+yJ_{2^{N_4},N_4})
	\end{matrix}\right).
	\end{equation*}
	According to Lemma \ref{le-ct} and Proposition \ref{pr-ct},
	\begin{equation}\label{eq-rg}
	\begin{split}
	&(\frac{\partial}{\partial x}+K_{N_4}\frac{\partial}{\partial y})g^{(K)}(x+yK_{N_4})
	\\=&(\frac{\partial}{\partial x}+K_{N_4}\frac{\partial}{\partial y})\zeta(K^{(N_4)})\mathcal{M}(J^{(N_4)})^{-1}F_4(x+yJ^{(N_4)})
	\\=&\zeta(K^{(N_4)})\mathcal{M}(J^{(N_4)})^{-1}\frac{\partial}{\partial x}F_4(x+yJ^{(N_4)})
	\\&+\zeta(K^{(N_4)})\sigma_N\mathcal{M}(J^{(N_4)})^{-1}\frac{\partial}{\partial y}F_4(x+yJ^{(N_4)})
	\\=&\zeta(K^{(N_4)})\mathcal{M}(J^{(N_4)})^{-1}\frac{\partial}{\partial x}F_4(x+yJ^{(N_4)})
	\\&+\zeta(K^{(N_4)})\mathcal{M}(J^{(N_4)})^{-1}D_{N_4}(J^{(N_4)})\frac{\partial}{\partial y}F_4(x+yJ^{(N_4)})
	\\=&\zeta(K^{(N_4)})\mathcal{M}(J^{(N_4)})^{-1}(\frac{\partial}{\partial x}+D_{N_4}(J^{(N_4)})\frac{\partial}{\partial y})F_4(x+yJ^{(N_4)})
	\\=&\zeta(K^{(N_4)})\mathcal{M}(J^{(N_4)})^{-1}(a_1,a_2,...,a_{N_4})^T
	\\=&0
	\end{split}
	\end{equation}
	for each $x,y\in\mathbb{R}$ with $y\ge 0$ and $x+yi\in\mathbb{B}_4$, where
	\begin{equation*}
	a_\imath:=(\frac{\partial}{\partial x}+J_{\imath,N_4}\frac{\partial}{\partial y})f\circ\pi|_{U^{J_{\imath}}_4}^{-1}(x+yJ_{\imath,N_4})=0
	\end{equation*}
	for each $\imath=1,2,...,2^{N_4}$ (by $f\circ\pi|_{U^{J_{\imath}}_4}^{-1}$ being holomorphic on $P_{J_{\imath,N_4}}(\mathbb{B}_4)$). It follows that $g^{(K)}$ is holomorphic on $P_{K_{N_4}}(\mathbb{B}_4)$. Since $\gamma$ is continuous, there exists $t_5\in(\frac{N_4-1}{N},t_1)$ such that
	\begin{equation*}
	\gamma([t_5,t_1])\subset\mathbb{B}_4.
	\end{equation*}
	According to $t_5\in A$, there exists $r_5>0$ and domains $U^I_5$ in $G_I$ for each $I\in\mathbb{S}^N$ with $\gamma^I_{\mathcal{G}}[t_5](1)\in U_5^I$ and
	\begin{equation*}
	\pi|_{U^I_5}^{-1}:U^I_5\rightarrow P_{I_{N_4}}(B_\mathbb{C}(\gamma(t_5),r_5))
	\end{equation*}
	being a homeomorphism with respect to topologies of $\tau(G_{I_{N_4}})$ and $\tau(\mathbb{C}_{I_{N_4}})$, such that
	\begin{equation*}
	f\circ\pi|_{U^K_5}^{-1}(x+yK_{N_4})=\zeta(K^{(N_4)})\mathcal{M}(J^{(N_4)})^{-1}F_5(x+yJ^{(N_4)})
	\end{equation*}
	for each $K\in\mathbb{S}^N$, and $x,y\in\mathbb{R}$ with $y\ge 0$ and $x+yi\in B_\mathbb{C}(\gamma(t_5),r_5)$, where
	\begin{equation*}
	N_4=\lceil N{t_1}\rceil=\lceil N{t_5}\rceil
	\end{equation*}
	and
	\begin{equation*}
	F_5(x+yJ^{(N_4)}):=\left(
	\begin{matrix}
	&f\circ\pi|_{U^{J_1}_5}^{-1}(x+yJ_{1,N_4})
	\\&f\circ\pi|_{U^{J_2}_5}^{-1}(x+yJ_{2,N_4})
	\\&\vdots
	\\&f\circ\pi|_{U^{J_{2^{N_4}}}_5}^{-1}(x+yJ_{2^{N_4},N_4})
	\end{matrix}\right).
	\end{equation*}
	We notice that
	\begin{equation*}
	\gamma(t_5)\in\mathbb{B}_4\cap\mathbb{B}_5,
	\end{equation*}
	then $\mathbb{B}_4\cap\mathbb{B}_5$ is a nonempty open set in $\mathbb{C}$. It follows that there exists a complex number $z\in\mathbb{B}_4\cap\mathbb{B}_5$ and a positive real number $r_6>0$, such that
	\begin{equation*}
	B_{\mathbb{C}}(z,r_6)\cap\mathbb{R}=\varnothing.
	\end{equation*}
	We set
	\begin{equation*}
	\mathbb{B}_6:=B_{\mathbb{C}}(z,r_6)\subset\mathbb{B}_4\cap\mathbb{B}_5.
	\end{equation*}
	For each $K\in\mathbb{S}^N$, according to Proposition \ref{pr-xfh} and $P_{K_{N_4}}(\mathbb{B}_4)$ is slice-connected in $\mathbb{H}$, there exists a slice-domain $V^K_4$ in $\mathbb{H}$ with $P_{K_{N_4}}(\mathbb{B}_4)\subset V^K_4$, and a slice regular function $\widetilde{g^{(K)}}$ on $V^K_4$ with
	\begin{equation*}
	\widetilde{g^{(K)}}|_{P_{K_{N_4}}(\mathbb{B}_4)}=g^{(K)}.
	\end{equation*}
	We set
	\begin{equation*}
	V^K_6:=P_{K_{N_4}}(\mathbb{B}_6),\qquad z_K:=P_{K_{N_4}}(z)\qquad\mbox{and}\qquad q_K:=\pi|_{U^K_5}^{-1}(z_K).
	\end{equation*}
	We notice that
	\begin{equation*}
	\mathcal{G}_1:=(V^K_4,id_{V^K_4},z_K),\qquad\mathcal{G}_2:=(V^K_6,id_{V^K_6},z_K)\qquad\mbox{and}\qquad\mathcal{G}_3:=(G,\pi,q_K)
	\end{equation*}
	are slice-domains over $\mathbb{H}$ with distinguished point. And
	\begin{equation*}
	\widetilde{g^{(K)}},\qquad g^{(K)}|_{V^K_6}\qquad\mbox{and}\qquad f
	\end{equation*}
	are respectively slice regular functions on $V^K_4$, $V^K_6$ and $G$. Notice that
	\begin{equation*}
	U^{J_\imath}_5\subset U^{J_\imath}_4\qquad\mbox{and}\qquad\mathbb{B}_6\subset\mathbb{B}_5,
	\end{equation*}
	we have
	\begin{equation*}
	\begin{split}
	f|_{V^K_6}(x+yK_{N_4})=&f\circ\pi|_{U^K_5}^{-1}(x+yK_{N_4})
	\\=&\zeta(K^{(N_4)})\mathcal{M}(J^{(N_4)})^{-1}F_5(x+yJ^{(N_4)})
	\\=&\zeta(K^{(N_4)})\mathcal{M}(J^{(N_4)})^{-1}F_4(x+yJ^{(N_4)})
	\\=&g^{(K)}(x+yK_{N_4})
	\end{split}
	\end{equation*}
	for each $x,y\in\mathbb{R}$ with $y\ge 0$ and $x+yi\in\mathbb{B}_6$. Then $f$ and $\widetilde{g^{(K)}}$ are slice regular extensions of $g^{(K)}|_{V^K_6}$. According to Proposition \ref{pr-tsh}, $\mathcal{G}_3$ is a slice-domain of existence of $f$. It follows that $\mathcal{G}_3$ is also a slice-domain of existence of $g^{(K)}|_{V^K_6}$. Thanks to Theorem \ref{th-exsd},
	\begin{equation*}
	\mathcal{G}_1\prec\mathcal{G}_3.
	\end{equation*}
	We denote the fiber preserving map from $\mathcal{G}_1$ to $\mathcal{G}_3$ by
	\begin{equation*}
	\varphi:V^K_4\rightarrow G.
	\end{equation*}
	And we set
	\begin{equation*}
	q_5:=P_{K_{N_4}}(\gamma(t_5))\in\mathbb{H}.
	\end{equation*}
	According to Proposition \ref{pr-sdt},
	\begin{equation*}
	(V^K_4,id_{V^K_4},q_5)\prec(G,\pi,\pi|_{U^K_5}^{-1}(q_5)).
	\end{equation*}
	Due to Proposition \ref{pr-tsp}, and since
	\begin{equation*}
	\gamma^K([t_5,t_1])\subset V^K_4,\qquad\gamma^K[t_5]\prec\mathcal{G}\qquad\mbox{and}\qquad\gamma^K[t_5](1)=\gamma^K(t_5)=q_5,
	\end{equation*}
	it follows that $\gamma^K[t_1]\prec\mathcal{G}$ for each $K\in\mathbb{S}^N$. We notice that $\varphi(V^K_4)$ is a domain in $G$ containing $\gamma^K[t_1](1)$ and
	\begin{equation*}
	U^I_4=\pi|_{U^I_4}^{-1}(P_{I_{N_4}}(\mathbb{B}_4))=\pi|_{\varphi(V^K_4)}^{-1}(P_{I_{N_4}}(\mathbb{B}_4)),\qquad\forall\ I\in\{J_1,J_2,...,J_{2^{N_4}}\}.
	\end{equation*}
	Then we can well define the domain
	\begin{equation*}
	U^K_4:=\pi|_{\varphi(V^K_4)}^{-1}(P_{K_{N_4}}(\mathbb{B}_4))
	\end{equation*}
	in $G_{K_{N_4}}$ for each $K\in\mathbb{S}^N$ with
	\begin{equation*}
	\pi|_{U^K_4}^{-1}:U^K_4\rightarrow P_{K_{N_4}}(\mathbb{B}_4)
	\end{equation*}
	being a homeomorphism with respect to topologies $\tau(G_{K_{N_4}})$ and $\tau(\mathbb{C}_{K_{N_4}})$. And we notice that
	\begin{equation*}
	{\id}_{V^K_4}=\pi\circ\varphi\qquad\mbox{and}\qquad P_{K_{N_4}}(\mathbb{B}_4)\subset V^K_4,
	\end{equation*}
	thence
	\begin{equation}\label{eq-fp}
	\varphi|_{P_{K_{N_4}}(\mathbb{B}_4)}=\pi|_{U^K_4}^{-1}.
	\end{equation}
	Thanks to (\ref{eq-gk}) and (\ref{eq-fp}), it follows that
	\begin{equation}\label{eq-rf3}
	\begin{split}
	f\circ\pi|_{U^K_4}^{-1}(x+yK_{N_4})=&f\circ\varphi(x+yK_{N_4})
	\\=&f|_{V_4^K}(x+yK_{N_4})
	\\=&g^{(K)}(x+yK_{N_4})
	\\=&\zeta(K^{(N_4)})\mathcal{M}(J^{(N_4)})^{-1}F_4(x+yJ^{(N_4)})
	\end{split}
	\end{equation}
	for each $x,y\in\mathbb{R}$ with $y\ge 0$ and $x+yi\in\mathbb{B}_4$, where
	\begin{equation*}
	F_4(x+yJ^{(N_4)}):=\left(
	\begin{matrix}
	&(f\circ\pi|_{U^{J_1}_4}^{-1}(x+yJ_{1,N_4})
	\\&f\circ\pi|_{U^{J_2}_4}^{-1}(x+yJ_{2,N_4})
	\\&\vdots
	\\&f\circ\pi|_{U^{J_{2^{N_4}}}_4}^{-1}(x+yJ_{2^{N_4},N_4})
	\end{matrix}\right).
	\end{equation*}
	It follows that conditions (a) and (b) hold for $t=t_1$. If $\{Nt_1\}\neq 0$ or $t_1=1$, then $t_1\in A$, which is a contradiction. It follows that $\{Nt_1\}=0$ and $t_1\neq 1$. We will proof that the condition (c) holds for $t=t_1$ in the next step.
	
	\textbf{4). Following 3), we will prove that the condition (c) holds for $t=t_1$.}
	
	According to 2) and 3),  $\{Nt_1\}=0$ and $t_1\notin\{0,1\}$. We notice that the center $\gamma(t_1)$ of $\mathbb{B}_4$ is in $\mathbb{R}$. Thanks to Proposition \ref{pr-tsh} and Proposition \ref{pr-xfh}, there exists a domain $U^I_7$ in $\mathcal{G}$ containing $\gamma^I_\mathcal{G}(t_1)$ for each $I\in\mathbb{S}^N$, such that
	\begin{equation*}
	\pi|_{U^I_7}^{-1}:U^I_7\rightarrow \mathbb{B}_7^I
	\end{equation*}
	is a homeomorphism with respect to topologies $\tau(G_{I_{N_4+1}})$ and $\tau(\mathbb{C}_{I_{N_4+1}})$, where
	\begin{equation*}
	\mathbb{B}_7:=B_\mathbb{H}(\gamma(t_1),r_4)\qquad\mbox{and}\qquad\mathbb{B}_7^I:=\mathbb{B}_7\cap\mathbb{C}_{I_{N_4+1}}.
	\end{equation*}
	For each $K\in\mathbb{S}^N$, we define a function $g^{(K)}_1$ on $\mathbb{B}_7^K$, by
	\begin{equation}\label{eq-dsa}
	g_1^{(K)}(x+yK_{N_4+1})=\zeta(K^{(N_4+1)})\mathcal{M}(J^{(N_4+1)})^{-1}F_7(x+yJ^{(N_4+1)})
	\end{equation}
	for each $x,y\in\mathbb{R}$ with $y\ge 0$ and $x+yi\in\mathbb{B}_4$, where
	\begin{equation*}
	F_7(x+yJ^{(N_4+1)}):=\left(
	\begin{matrix}
	&(f\circ\pi|_{U^{J_1}_7}^{-1}(x+yJ_{1,N_4+1})
	\\&f\circ\pi|_{U^{J_2}_7}^{-1}(x+yJ_{2,N_4+1})
	\\&\vdots
	\\&f\circ\pi|_{U^{J_{2^{N_4+1}}}_7}^{-1}(x+yJ_{2^{N_4+1},N_4+1})
	\end{matrix}\right).
	\end{equation*}
	We can prove that $g_1^{(K)}$ is holomorphic on $\mathbb{B}_7^K$, by direct calculation exactly as (\ref{eq-rg}). We notice that
	\begin{equation*}
	(U_4^I)_\mathbb{R}=\pi^{-1}_{U_4^I}(\mathbb{B}_4^I\cap\mathbb{R})=\pi^{-1}_{V_4^I}(B_\mathbb{R}(\gamma(t_1),r_4))=\pi^{-1}_{U_7^I}(\mathbb{B}_7^I\cap\mathbb{R})=(U_7^I)_\mathbb{R}
	\end{equation*}
	for each $I\in\mathbb{S}^N$. Then
	\begin{equation*}
	\begin{split}
	f\circ\pi|_{U_7^{J_\imath}}^{-1}(x)=f\circ\pi|_{U_4^{J_\imath}}^{-1}(x)
	=\zeta(J_\imath^{(N_4)})\mathcal{M}(J^{(N_4)})^{-1}F_4(x)
	\end{split}
	\end{equation*}
	for each $x\in\mathbb{B}_4\cap\mathbb{R}$ and $\imath\in\{1,2,...,2^{N_4+1}\}$, where
	\begin{equation*}
	F_4(x)=(f\circ\pi|_{U^{J_1}_4}^{-1}(x),f\circ\pi|_{U^{J_2}_4}^{-1}(x),...,f\circ\pi|_{U^{J_{2^{N_4}}}_4}^{-1}(x))^T.
	\end{equation*}
	Then we set
	\begin{equation}\label{eq-sla}
	\begin{split}
	F_7(x):=&(f\circ\pi|_{U^{J_1}_7}^{-1}(x),f\circ\pi|_{U^{J_2}_7}^{-1}(x),...,f\circ\pi|_{U^{J_{2^{N_4+1}}}_7}^{-1}(x))^T
	\\=&(\zeta(J_1^{(N_4)}),\zeta(J_2^{(N_4)}),...,\zeta(J_{2^{N_4+1}}^{(N_4)}))^T\mathcal{M}(J^{(N_4)})^{-1}F_4(x)
	\\=&\mathcal{M}(J^{(N_4+1)})(\mathbb{I}_{2^{N_4}},0_{2^{N_4}})^T\mathcal{M}(J^{(N_4)})^{-1}F_4(x)
	\end{split}
	\end{equation}
	for each $x\in\mathbb{B}_4\cap\mathbb{R}$. Let
	\begin{equation*}
	V_7^I:=\pi|_{U^I_7}^{-1}(\mathbb{B}_7^I)
	\end{equation*}
	be a domain in $G_{I_{N_4+1}}$ for each $I\in\mathbb{S}^N$. Then $\pi|_{V_7^I}:V_7^I\rightarrow\mathbb{B}_7^I$ is a homeomorphism with respect to topologies $\tau(G_{I_{N_4+1}})$ and $\tau(\mathbb{C}_{I_{N_4+1}})$. Thanks to (\ref{eq-rf3}), (\ref{eq-dsa}), (\ref{eq-sla}) and
	\begin{equation*}
	\pi|_{U^K_7}^{-1}(\mathbb{B}_7\cap\mathbb{R})\subset V_7^K\cap U^K_4,
	\end{equation*}
	it follows that
	\begin{equation*}
	\begin{split}
	g_1^{(K)}(x)=&\zeta(K^{(N_4+1)})\mathcal{M}(J^{(N_4+1)})^{-1}F_7(x)
	\\=&\zeta(K^{(N_4+1)})\mathcal{M}(J^{(N_4+1)})^{-1}\mathcal{M}(J^{(N_4+1)})(\mathbb{I}_{2^{N_4}},0_{2^{N_4}})^T\mathcal{M}(J^{(N_4)})^{-1}F_4(x)
	\\=&\zeta(K^{(N_4+1)})(\mathbb{I}_{2^{N_4}},0_{2^{N_4}})^T\mathcal{M}(J^{(N_4)})^{-1}F_4(x)
	\\=&\zeta(K^{(N_4)})\mathcal{M}(J^{(N_4)})^{-1}F_4(x)
	\\=&f\circ\pi|_{U^K_4}^{-1}(x)
	\\=&f\circ\pi|_{V^K_7}^{-1}(x)
	\end{split}
	\end{equation*}	
	for each $x\in\mathbb{B}_4\cap\mathbb{R}$ and $K\in\mathbb{S}^N$. According to the Identity Principle \ref{th-idh} and (\ref{eq-dsa}),
	\begin{equation*}
	\begin{split}
	f\circ\pi|_{V^K_7}^{-1}(x+yK_{N_4+1})=&g_1^{(K)}(x+yK_{N_4+1})
	\\=&\zeta(K^{(N_4+1)})\mathcal{M}(J^{(N_4+1)})^{-1}F_7(x+yJ^{(N_4+1)})
	\end{split}
	\end{equation*}
	for each $x,y\in\mathbb{R}$ with $y\ge 0$ and $x+yi\in\mathbb{B}_4$, and $K\in\mathbb{S}^N$. Then condition (c) holds when $t=t_1$. It follows that $t_1\in A$, which is a contradiction. Then $A=[0,1]$.

	\textbf{2.} Since $1\in A$ and condition (a), it follows that
	\begin{equation*}
	\gamma^K=\gamma^K[1]\prec\mathcal{G}
	\end{equation*}
	for each $K\in\mathbb{S}^N$. According to condition (b), there exists a positive real number $r_8>0$, and domains $U^I_8$ in $G_{I_N}$ for each $I\in\mathbb{S}^N$ with $\gamma^I_{\mathcal{G}}(1)=\gamma^I_{\mathcal{G}}[1](1)\in U^I_8$ and
	\begin{equation*}
	\pi|_{U^I_8}:U^I_8\rightarrow P_{I_N}(B_\mathbb{C}(\gamma(1),r_8))
	\end{equation*}
	being a homeomorphism with respect to topologies $\tau(G_{I_N})$ and $\tau(\mathbb{C}_{I_N})$, such that
	\begin{equation}\label{eq-t1r}
	f\circ\pi|_{U^K_8}^{-1}(x+yK_N)=\zeta(K)\mathcal{M}(J)^{-1}F_8(x+yJ)
	\end{equation}
	for each $K\in\mathbb{S}^N$, and $x,y\in\mathbb{R}$ with $y\ge 0$ and $x+yi\in B_\mathbb{C}(\gamma(1),r)$, where
	\begin{equation*}
	F_8(x+yJ):=(f\circ\pi|_{U^{J_1}_8}^{-1}(x+yJ_{1,N}),f\circ\pi|_{U^{J_2}_8}^{-1}(x+yJ_{2,N}),...,f\circ\pi|_{U^{J_{2^N}}_8}^{-1}(x+yJ_{2^N,N}))^T.
	\end{equation*}
	We write
	\begin{equation*}
	\gamma(1)=x_1+y_1 i,
	\end{equation*}
	for some $x_1,y_1\in\mathbb{R}$. Then
	\begin{equation*}
		\pi(\gamma^I_{\mathcal{G}}(1))=\gamma^I(1)=P_{I_N}(\gamma(1))=x_1+y_1 I_N,\qquad\forall\ I\in\mathbb{S}^N.
	\end{equation*}
	According to (\ref{eq-t1r}) and $\gamma^K_{\mathcal{G}}(1)\in U^K_8$,
	\begin{equation*}
	\begin{split}
	f(\gamma^K_\mathcal{G}(1))=&f\circ\pi|_{U_8^I}^{-1}(x_1+y_1 K_N)
	\\=&\zeta(K)\mathcal{M}(J)^{-1}F_8(x_1+y_1J)
	\\=&\zeta(K)\mathcal{M}(J)^{-1}(b_1,b_2,...,b_{2^N})^T
	\\=&\zeta(K)\mathcal{M}(J)^{-1}(f(\gamma^{J_1}_\mathcal{G}(1)),f(\gamma^{J_2}_\mathcal{G}(1)),...,f(\gamma^{J_{2^N}}_\mathcal{G}(1)))^T
	\\=&\zeta(K)\mathcal{M}(J)^{-1}f(\gamma^{J}_\mathcal{G}(1))
	\end{split}
	\end{equation*}
	for each $K\in\mathbb{S}^N$, where
	\begin{equation*}
	b_\imath=f\circ\pi|_{U^{J_\imath}_8}^{-1}(x_1+y_1 J_{\imath,N}),\qquad\forall\ \imath\in\{1,2,...,2^N\}.
	\end{equation*}
\end{proof}

Finally, we demonstrate that our representation formula is the same as the classical one over quaternions \cite[Theorem 3.2]{Colombo2009001} when the slice-domain of regularity is an axially symmetric domain.

\begin{defn}
	(\cite[Definition 2.25]{Colombo2010001}) Let $\Omega\subset\mathbb{H}$. We say that $\Omega$ is axially symmetric if, for all $x+yI\in\Omega$, the whole $2$-sphere $x+y\mathbb{S}$ is contained in $\Omega$.
\end{defn}

\begin{rmk}\label{rm-sdr}
	Let $\mathcal{G}=(G,\pi,x_0)$ be a slice-domain of regularity with $\pi(x_0)\in\mathbb{R}$, and $f$ be a slice regular function on $G$ with $\mathcal{G}$ being a slice-domain of existence of $f$. If $G$ is axially symmetric slice domain in $\mathbb{H}$, then for each $q=x+yK\in G$ and $\pi={\id}_G$, then there exists a path $\gamma^K=P_K(\gamma)$ in $\mathbb{C}_K$ from $\pi(x_0)$ to $q$, where $\gamma$ is a path in $\mathbb{C}$, $x,y\in\mathbb{R}$ and $K\in\mathbb{S}$. Then for each $J=(J_1,J_2)\in\mathbb{S}^2$, we have

\begin{equation*}
		f(x+yK)
	=f\circ{\id}_G^{-1}(\gamma^K(1))
	=f(\gamma^K_\mathcal{G}(1)).
	 	\end{equation*}
Theorem \ref{th-dere} tells us that
\begin{equation*}
	 f(\gamma^K_\mathcal{G}(1))
=(1,K)\left(\begin{matrix}1&J_1\\1&J_2\end{matrix}\right)^{-1}
	\left(\begin{matrix}f(\gamma^{J_1}_\mathcal{G}(1))\\f(\gamma^{J_2}_\mathcal{G}(1))\end{matrix}\right).
	 	\end{equation*}
Notice that
$$\left(\begin{matrix}1&J_1\\1&J_2\end{matrix}\right)^{-1}=\left(\begin{matrix}(J_1-J_2)^{-1}J_1&-(J_1-J_2)^{-1}J_2\\(J_1-J_2)^{-1}&-(J_1-J_2)^{-1}\end{matrix}\right)$$
and
$$f(\gamma^{J_\imath}_\mathcal{G}(1))=f\circ{\id}_G(\gamma^{J_\imath}(1))=f(x+yJ_\imath), \qquad \forall\ \imath =1,2.$$
A direct calculation shows that  the classical representation formula (see \cite[Theorem 3.2]{Colombo2009001}) holds true, i.e.,
\begin{equation*}
\begin{split}
f(x+yK)
		=(J_1-J_2)^{-1}[J_1f(x+yJ_1)-J_2f(x+yJ_2)]&
\\ \qquad +K(J_1-J_2)^{-1}[f(x+yJ_1)-f(x+yJ_2)].&
	\end{split}
	\end{equation*}
	
	We remark that Theorem \ref{th-dere} is just suitable for slice-domains of regularity. In the coming article, we will prove a general representation formula over Riemann slice-domains (see \cite[Theorem \ref{tm-rf}]{Dou2018002}), which covers the classical one over quaternions \cite[Theorem 3.2]{Colombo2009001}.
\end{rmk}

\section*{Acknowledgement}

The authors would like to thank Dr. Xieping Wang for the discussion of the theory about Riemann domains over $\mathbb{C}^n$ and pointing out several inaccuracies in a draft of this paper.

\bibliographystyle{bibstyle}
\bibliography{mybibfile}

\begin{thebibliography}{10}
\expandafter\ifx\csname url\endcsname\relax
  \def\url#1{\texttt{#1}}\fi
\expandafter\ifx\csname urlprefix\endcsname\relax\def\urlprefix{URL }\fi
\expandafter\ifx\csname href\endcsname\relax
  \def\href#1#2{#2} \def\path#1{#1}\fi

\bibitem{Gentili2006001}
G.~Gentili, D.~C. Struppa, \href{https://doi.org/10.1016/j.crma.2006.03.015}{A
  new approach to {C}ullen-regular functions of a quaternionic variable}, C. R.
  Math. Acad. Sci. Paris 342~(10) (2006) 741--744.
\newblock \href {http://dx.doi.org/10.1016/j.crma.2006.03.015}
  {\path{doi:10.1016/j.crma.2006.03.015}}.
\newline\urlprefix\url{https://doi.org/10.1016/j.crma.2006.03.015}

\bibitem{Gentili2007001}
G.~Gentili, D.~C. Struppa, \href{https://doi.org/10.1016/j.aim.2007.05.010}{A
  new theory of regular functions of a quaternionic variable}, Adv. Math.
  216~(1) (2007) 279--301.
\newblock \href {http://dx.doi.org/10.1016/j.aim.2007.05.010}
  {\path{doi:10.1016/j.aim.2007.05.010}}.
\newline\urlprefix\url{https://doi.org/10.1016/j.aim.2007.05.010}

\bibitem{Colombo2011001B}
F.~Colombo, I.~Sabadini, D.~C. Struppa,
  \href{https://doi.org/10.1007/978-3-0348-0110-2}{Noncommutative functional
  calculus}, Vol. 289 of Progress in Mathematics, Birkh\"auser/Springer Basel
  AG, Basel, 2011, theory and applications of slice hyperholomorphic functions.
\newblock \href {http://dx.doi.org/10.1007/978-3-0348-0110-2}
  {\path{doi:10.1007/978-3-0348-0110-2}}.
\newline\urlprefix\url{https://doi.org/10.1007/978-3-0348-0110-2}

\bibitem{Gentili2013001B}
G.~Gentili, C.~Stoppato, D.~C. Struppa,
  \href{https://doi.org/10.1007/978-3-642-33871-7}{Regular functions of a
  quaternionic variable}, Springer Monographs in Mathematics, Springer,
  Heidelberg, 2013.
\newblock \href {http://dx.doi.org/10.1007/978-3-642-33871-7}
  {\path{doi:10.1007/978-3-642-33871-7}}.
\newline\urlprefix\url{https://doi.org/10.1007/978-3-642-33871-7}

\bibitem{Alpay2016001B}
D.~Alpay, F.~Colombo, I.~Sabadini,
  \href{https://doi.org/10.1007/978-3-319-42514-6}{Slice hyperholomorphic
  {S}chur analysis}, Vol. 256 of Operator Theory: Advances and Applications,
  Birkh\"auser/Springer, Cham, 2016.
\newblock \href {http://dx.doi.org/10.1007/978-3-319-42514-6}
  {\path{doi:10.1007/978-3-319-42514-6}}.
\newline\urlprefix\url{https://doi.org/10.1007/978-3-319-42514-6}

\bibitem{Gentili2008003}
G.~Gentili, D.~C. Struppa,
  \href{https://doi.org/10.1080/17476930701778869}{Regular functions on a
  {C}lifford algebra}, Complex Var. Elliptic Equ. 53~(5) (2008) 475--483.
\newblock \href {http://dx.doi.org/10.1080/17476930701778869}
  {\path{doi:10.1080/17476930701778869}}.
\newline\urlprefix\url{https://doi.org/10.1080/17476930701778869}

\bibitem{Colombo2009002}
F.~Colombo, I.~Sabadini, D.~C. Struppa,
  \href{https://doi.org/10.1007/s11856-009-0055-4}{Slice monogenic functions},
  Israel J. Math. 171 (2009) 385--403.
\newblock \href {http://dx.doi.org/10.1007/s11856-009-0055-4}
  {\path{doi:10.1007/s11856-009-0055-4}}.
\newline\urlprefix\url{https://doi.org/10.1007/s11856-009-0055-4}

\bibitem{Colombo2015001}
F.~Colombo, R.~L\'avi\~vcka, I.~Sabadini, V.~Sou\~vcek,
  \href{https://doi.org/10.1007/s00208-015-1182-3}{The {R}adon transform
  between monogenic and generalized slice monogenic functions}, Math. Ann.
  363~(3-4) (2015) 733--752.
\newblock \href {http://dx.doi.org/10.1007/s00208-015-1182-3}
  {\path{doi:10.1007/s00208-015-1182-3}}.
\newline\urlprefix\url{https://doi.org/10.1007/s00208-015-1182-3}

\bibitem{Ren2017001}
G.~Ren, X.~Wang, \href{https://doi.org/10.2140/pjm.2017.290.169}{Growth and
  distortion theorems for slice monogenic functions}, Pacific J. Math. 290~(1)
  (2017) 169--198.
\newblock \href {http://dx.doi.org/10.2140/pjm.2017.290.169}
  {\path{doi:10.2140/pjm.2017.290.169}}.
\newline\urlprefix\url{https://doi.org/10.2140/pjm.2017.290.169}

\bibitem{Gentili2008002}
G.~Gentili, D.~C. Struppa, F.~Vlacci,
  \href{https://doi.org/10.1007/s00209-007-0254-9}{The fundamental theorem of
  algebra for {H}amilton and {C}ayley numbers}, Math. Z. 259~(4) (2008)
  895--902.
\newblock \href {http://dx.doi.org/10.1007/s00209-007-0254-9}
  {\path{doi:10.1007/s00209-007-0254-9}}.
\newline\urlprefix\url{https://doi.org/10.1007/s00209-007-0254-9}

\bibitem{Gentili2010001}
G.~Gentili, D.~C. Struppa,
  \href{https://doi.org/10.1216/RMJ-2010-40-1-225}{Regular functions on the
  space of {C}ayley numbers}, Rocky Mountain J. Math. 40~(1) (2010) 225--241.
\newblock \href {http://dx.doi.org/10.1216/RMJ-2010-40-1-225}
  {\path{doi:10.1216/RMJ-2010-40-1-225}}.
\newline\urlprefix\url{https://doi.org/10.1216/RMJ-2010-40-1-225}

\bibitem{Colombo2010003}
F.~Colombo, I.~Sabadini, D.~C. Struppa, Duality theorems for slice
  hyperholomorphic functions, J. Reine Angew. Math. 645 (2010) 85--105.
\newblock \href {http://dx.doi.org/10.1515/CRELLE.2010.060}
  {\path{doi:10.1515/CRELLE.2010.060}}.

\bibitem{Wang2017001}
X.~Wang, \href{https://doi.org/10.1007/s12220-017-9784-5}{On geometric aspects
  of quaternionic and octonionic slice regular functions}, J. Geom. Anal.
  27~(4) (2017) 2817--2871.
\newblock \href {http://dx.doi.org/10.1007/s12220-017-9784-5}
  {\path{doi:10.1007/s12220-017-9784-5}}.
\newline\urlprefix\url{https://doi.org/10.1007/s12220-017-9784-5}

\bibitem{Ghiloni2011001}
R.~Ghiloni, A.~Perotti, \href{https://doi.org/10.1016/j.aim.2010.08.015}{Slice
  regular functions on real alternative algebras}, Adv. Math. 226~(2) (2011)
  1662--1691.
\newblock \href {http://dx.doi.org/10.1016/j.aim.2010.08.015}
  {\path{doi:10.1016/j.aim.2010.08.015}}.
\newline\urlprefix\url{https://doi.org/10.1016/j.aim.2010.08.015}

\bibitem{Ghiloni20171001}
R.~Ghiloni, A.~Perotti, C.~Stoppato,
  \href{https://doi.org/10.1016/j.aim.2016.10.009}{Singularities of slice
  regular functions over real alternative {$^*$}-algebras}, Adv. Math. 305
  (2017) 1085--1130.
\newblock \href {http://dx.doi.org/10.1016/j.aim.2016.10.009}
  {\path{doi:10.1016/j.aim.2016.10.009}}.
\newline\urlprefix\url{https://doi.org/10.1016/j.aim.2016.10.009}

\bibitem{Ghiloni20171002}
R.~Ghiloni, A.~Perotti, C.~Stoppato,
  \href{https://doi.org/10.1090/tran/6816}{The algebra of slice functions},
  Trans. Amer. Math. Soc. 369~(7) (2017) 4725--4762.
\newblock \href {http://dx.doi.org/10.1090/tran/6816}
  {\path{doi:10.1090/tran/6816}}.
\newline\urlprefix\url{https://doi.org/10.1090/tran/6816}

\bibitem{Colombo2009003}
F.~Colombo, I.~Sabadini, \href{https://doi.org/10.1007/s12220-009-9075-x}{On
  some properties of the quaternionic functional calculus}, J. Geom. Anal.
  19~(3) (2009) 601--627.
\newblock \href {http://dx.doi.org/10.1007/s12220-009-9075-x}
  {\path{doi:10.1007/s12220-009-9075-x}}.
\newline\urlprefix\url{https://doi.org/10.1007/s12220-009-9075-x}

\bibitem{Colombo2011001}
F.~Colombo, I.~Sabadini, \href{https://doi.org/10.1016/j.aim.2011.04.001}{The
  quaternionic evolution operator}, Adv. Math. 227~(5) (2011) 1772--1805.
\newblock \href {http://dx.doi.org/10.1016/j.aim.2011.04.001}
  {\path{doi:10.1016/j.aim.2011.04.001}}.
\newline\urlprefix\url{https://doi.org/10.1016/j.aim.2011.04.001}

\bibitem{Ghiloni2018001}
R.~Ghiloni, V.~Recupero, \href{https://doi.org/10.1090/tran/7354}{Slice regular
  semigroups}, Trans. Amer. Math. Soc. 370~(7) (2018) 4993--5032.
\newblock \href {http://dx.doi.org/10.1090/tran/7354}
  {\path{doi:10.1090/tran/7354}}.
\newline\urlprefix\url{https://doi.org/10.1090/tran/7354}

\bibitem{Ghiloni2013001}
R.~Ghiloni, V.~Moretti, A.~Perotti,
  \href{https://doi.org/10.1142/S0129055X13500062}{Continuous slice functional
  calculus in quaternionic {H}ilbert spaces}, Rev. Math. Phys. 25~(4) (2013)
  1350006, 83.
\newblock \href {http://dx.doi.org/10.1142/S0129055X13500062}
  {\path{doi:10.1142/S0129055X13500062}}.
\newline\urlprefix\url{https://doi.org/10.1142/S0129055X13500062}

\bibitem{Colombo2009001}
F.~Colombo, G.~Gentili, I.~Sabadini, D.~Struppa,
  \href{https://doi.org/10.1016/j.aim.2009.06.015}{Extension results for slice
  regular functions of a quaternionic variable}, Adv. Math. 222~(5) (2009)
  1793--1808.
\newblock \href {http://dx.doi.org/10.1016/j.aim.2009.06.015}
  {\path{doi:10.1016/j.aim.2009.06.015}}.
\newline\urlprefix\url{https://doi.org/10.1016/j.aim.2009.06.015}

\bibitem{Colombo2010001}
F.~Colombo, G.~Gentili, I.~Sabadini,
  \href{https://doi.org/10.1007/s10455-009-9191-7}{A {C}auchy kernel for slice
  regular functions}, Ann. Global Anal. Geom. 37~(4) (2010) 361--378.
\newblock \href {http://dx.doi.org/10.1007/s10455-009-9191-7}
  {\path{doi:10.1007/s10455-009-9191-7}}.
\newline\urlprefix\url{https://doi.org/10.1007/s10455-009-9191-7}

\bibitem{Gentili2012001}
G.~Gentili, C.~Stoppato, \href{https://doi.org/10.1007/s00208-010-0631-2}{Power
  series and analyticity over the quaternions}, Math. Ann. 352~(1) (2012)
  113--131.
\newblock \href {http://dx.doi.org/10.1007/s00208-010-0631-2}
  {\path{doi:10.1007/s00208-010-0631-2}}.
\newline\urlprefix\url{https://doi.org/10.1007/s00208-010-0631-2}

\bibitem{Stoppato2012001}
C.~Stoppato, \href{https://doi.org/10.1016/j.aim.2012.05.023}{A new series
  expansion for slice regular functions}, Adv. Math. 231~(3-4) (2012)
  1401--1416.
\newblock \href {http://dx.doi.org/10.1016/j.aim.2012.05.023}
  {\path{doi:10.1016/j.aim.2012.05.023}}.
\newline\urlprefix\url{https://doi.org/10.1016/j.aim.2012.05.023}

\bibitem{Colombo2012001}
F.~Colombo, I.~Sabadini, D.~C. Struppa,
  \href{https://doi.org/10.1002/mana.201000149}{Sheaves of slice regular
  functions}, Math. Nachr. 285~(8-9) (2012) 949--958.
\newblock \href {http://dx.doi.org/10.1002/mana.201000149}
  {\path{doi:10.1002/mana.201000149}}.
\newline\urlprefix\url{https://doi.org/10.1002/mana.201000149}

\bibitem{Gentili2016001}
G.~Gentili, G.~Sarfatti, D.~C. Struppa,
  \href{https://doi.org/10.4310/MRL.2016.v23.n6.a4}{Ideals of regular functions
  of a quaternionic variable}, Math. Res. Lett. 23~(6) (2016) 1645--1663.
\newblock \href {http://dx.doi.org/10.4310/MRL.2016.v23.n6.a4}
  {\path{doi:10.4310/MRL.2016.v23.n6.a4}}.
\newline\urlprefix\url{https://doi.org/10.4310/MRL.2016.v23.n6.a4}

\bibitem{Alpay2012001}
D.~Alpay, F.~Colombo, I.~Sabadini,
  \href{https://doi.org/10.1007/s00020-011-1935-7}{Schur functions and their
  realizations in the slice hyperholomorphic setting}, Integral Equations
  Operator Theory 72~(2) (2012) 253--289.
\newblock \href {http://dx.doi.org/10.1007/s00020-011-1935-7}
  {\path{doi:10.1007/s00020-011-1935-7}}.
\newline\urlprefix\url{https://doi.org/10.1007/s00020-011-1935-7}

\bibitem{Alpay2015001}
D.~Alpay, F.~Colombo, J.~Gantner, I.~Sabadini,
  \href{https://doi.org/10.1007/s12220-014-9499-9}{A new resolvent equation for
  the {$S$}-functional calculus}, J. Geom. Anal. 25~(3) (2015) 1939--1968.
\newblock \href {http://dx.doi.org/10.1007/s12220-014-9499-9}
  {\path{doi:10.1007/s12220-014-9499-9}}.
\newline\urlprefix\url{https://doi.org/10.1007/s12220-014-9499-9}

\bibitem{Alpay2016001}
D.~Alpay, F.~Colombo, T.~Qian, I.~Sabadini,
  \href{https://doi.org/10.1016/j.jfa.2016.06.009}{The {$H^\infty$} functional
  calculus based on the {$S$}-spectrum for quaternionic operators and for
  {$n$}-tuples of noncommuting operators}, J. Funct. Anal. 271~(6) (2016)
  1544--1584.
\newblock \href {http://dx.doi.org/10.1016/j.jfa.2016.06.009}
  {\path{doi:10.1016/j.jfa.2016.06.009}}.
\newline\urlprefix\url{https://doi.org/10.1016/j.jfa.2016.06.009}

\bibitem{Colombo2018001}
F.~Colombo, J.~Gantner, \href{https://doi.org/10.1090/tran/7013}{Fractional
  powers of quaternionic operators and {K}ato's formula using slice
  hyperholomorphicity}, Trans. Amer. Math. Soc. 370~(2) (2018) 1045--1100.
\newblock \href {http://dx.doi.org/10.1090/tran/7013}
  {\path{doi:10.1090/tran/7013}}.
\newline\urlprefix\url{https://doi.org/10.1090/tran/7013}

\bibitem{Teleman2003001B}
C.~Teleman, \href{https://math.berkeley.edu/~teleman/math/Riemann.pdf}{Riemann
  surfaces}, Unpublished, 2003.
\newline\urlprefix\url{https://math.berkeley.edu/~teleman/math/Riemann.pdf}

\bibitem{Fritzsche2002001B}
K.~Fritzsche, H.~Grauert, \href{https://doi.org/10.1007/978-1-4684-9273-6}{From
  holomorphic functions to complex manifolds}, Vol. 213 of Graduate Texts in
  Mathematics, Springer-Verlag, New York, 2002.
\newblock \href {http://dx.doi.org/10.1007/978-1-4684-9273-6}
  {\path{doi:10.1007/978-1-4684-9273-6}}.
\newline\urlprefix\url{https://doi.org/10.1007/978-1-4684-9273-6}

\bibitem{Dou2018002}
X.~Dou, G.~Ren, Riemann slice-domains over quaternions {II}, submitted.

\bibitem{Gentili2008001}
G.~Gentili, C.~Stoppato, \href{https://doi.org/10.1307/mmj/1231770366}{Zeros of
  regular functions and polynomials of a quaternionic variable}, Michigan Math.
  J. 56~(3) (2008) 655--667.
\newblock \href {http://dx.doi.org/10.1307/mmj/1231770366}
  {\path{doi:10.1307/mmj/1231770366}}.
\newline\urlprefix\url{https://doi.org/10.1307/mmj/1231770366}

\bibitem{Alpay2013001}
D.~Alpay, F.~Colombo, I.~Sabadini,
  \href{https://doi.org/10.1007/s11854-013-0028-8}{Pontryagin-de
  {B}ranges-{R}ovnyak spaces of slice hyperholomorphic functions}, J. Anal.
  Math. 121 (2013) 87--125.
\newblock \href {http://dx.doi.org/10.1007/s11854-013-0028-8}
  {\path{doi:10.1007/s11854-013-0028-8}}.
\newline\urlprefix\url{https://doi.org/10.1007/s11854-013-0028-8}

\bibitem{Gentili2014001}
G.~Gentili, S.~Salamon, C.~Stoppato,
  \href{https://doi.org/10.4171/JEMS/488}{Twistor transforms of quaternionic
  functions and orthogonal complex structures}, J. Eur. Math. Soc. (JEMS)
  16~(11) (2014) 2323--2353.
\newblock \href {http://dx.doi.org/10.4171/JEMS/488}
  {\path{doi:10.4171/JEMS/488}}.
\newline\urlprefix\url{https://doi.org/10.4171/JEMS/488}

\bibitem{Zhang1997001}
F.~Zhang, \href{https://doi.org/10.1016/0024-3795(95)00543-9}{Quaternions and
  matrices of quaternions}, Linear Algebra Appl. 251 (1997) 21--57.
\newblock \href {http://dx.doi.org/10.1016/0024-3795(95)00543-9}
  {\path{doi:10.1016/0024-3795(95)00543-9}}.
\newline\urlprefix\url{https://doi.org/10.1016/0024-3795(95)00543-9}

\end{thebibliography}

\end{document}